\documentclass{article} 

\usepackage{authblk,breakcites}
\usepackage[round]{natbib}
\usepackage{microtype}
\usepackage{graphicx}
\usepackage{sidecap}
\usepackage{caption}
\usepackage{subfigure}
\usepackage{subcaption}
\usepackage{booktabs} 
\usepackage{multirow}
\usepackage[svgnames,dvipsnames]{xcolor}
\usepackage{float} 
\usepackage{wrapfig}
\usepackage[margin=3cm]{geometry}
\usepackage{enumerate}
\newcommand{\Cpdata}{C_{*}}
\newcommand{\Lpdata}{L_{*}}

\usepackage{hyperref}
\hypersetup{
    colorlinks=true,
    linkcolor=blue,
    filecolor=magenta,      
    urlcolor=cyan,
    citecolor=Emerald
    }

\newcommand{\rmd}{\mathrm{d}}
\newcommand{\rme}{\mathrm{e}}
\newcommand{\eqsp}{\,}
\newcommand{\kl}{\mathrm{KL}}
\newcommand{\ola}{\overleftarrow}
\newcommand{\ora}{\overrightarrow}
\newcommand{\eqdef}{:=}

\usepackage{amsmath}
\usepackage{amssymb}
\usepackage{mathtools}
\usepackage{amsthm}
\usepackage[capitalize,noabbrev]{cleveref}

\theoremstyle{plain}
\newtheorem{theorem}{Theorem}[section]
\newtheorem{proposition}[theorem]{Proposition}
\newtheorem{lemma}[theorem]{Lemma}
\newtheorem{corollary}[theorem]{Corollary}
\theoremstyle{definition}

\theoremstyle{remark}

\usepackage{ulem}

\newcommand{\sched}{\beta}
\newcommand{\param}{\theta}
\newcommand{\rset}{\mathbb{R}}
\newcommand{\gausspdf}{\varphi}
\newcommand{\pihat}{\widehat{\pi}_N^{(\beta,\theta)}} 
\newcommand{\R}{\mathbb{R}}
\newcommand{\Wc}{\mathcal{W}}
\newcommand{\1}{\mathbb{1}}
\def \E{\mathbb{E}}
\def \Nc{\mathcal{N}}

\newcounter{hypH}
\newenvironment{hypH}{
    \refstepcounter{hypH}
    \begin{itemize}
    \item[{\bf H\arabic{hypH}}]
    }
{\end{itemize}}

\title{An analysis of the noise schedule for score-based generative models}

\author[$\dag$]{Stanislas Strasman}
\author[*]{Antonio Ocello}
\author[$+,\circ$]{Claire Boyer}
\author[$\dag$]{Sylvain Le Corff}
\author[$\dag$]{Vincent Lemaire}

\affil[$\dag$]{{\small LPSM, 
       Sorbonne Universit\'e, UMR CNRS 8001, Paris, France.}}
\affil[*]{{\small CMAP, 
       \'Ecole Polytechnique, Institut Polytechnique de Paris, France.}}
\affil[$+$]{{\small LMO, 
       Universit\'e Paris-Saclay, 
       UMR CNRS 8628,
       Orsay, France.}}
\affil[$\circ$]{{\small IUF, 
       Institut Universitaire de France.}}

\date{}

\begin{document}

\maketitle

\begin{abstract}
Score-based generative models (SGMs) aim at estimating a target data distribution by learning score functions using only noise-perturbed samples from the target.
Recent literature has focused extensively on assessing the error between the target and estimated distributions, gauging the generative quality through the Kullback-Leibler (KL) divergence and Wasserstein distances.  
Under mild assumptions on the data distribution, we establish an upper bound for the KL divergence between the target and the estimated distributions, explicitly depending on any time-dependent noise schedule. Under additional regularity assumptions, taking advantage of favorable underlying contraction mechanisms, we provide a tighter error bound in Wasserstein distance compared to state-of-the-art results. 
In addition to being tractable, this upper bound jointly incorporates properties of the target distribution and SGM hyperparameters that need to be tuned during training. Finally, we illustrate these bounds through numerical experiments using simulated and CIFAR-10 datasets \footnote{Code available at \url{ https://github.com/StanislasStrasman/Noise_Schedule_for_Score-based_Generative_Models}.}, identifying an optimal range of noise schedules within a parametric family.
\end{abstract}

\section{Introduction}
\label{sec:intro}

Recent years have seen impressive advances in machine learning and artificial intelligence, with one of the most notable breakthroughs being the success of diffusion models, introduced by \citet{dickstein2015}. 
Diffusion models in generative modeling refer to a class of algorithms that generate new samples given training samples of an unknown distribution $\pi_{\mathrm{data}}$.
This method is now recognized for its ability to produce high-quality images that appear genuine to human observers \citep[see $e.g.$,][for text-to-image generation]{ramesh2022hierarchical}. Its range of applications is expanding rapidly, yielding impressive outcomes in areas such as computer vision \citep{li2022srdiff,lugmayr2022repaint} or natural language generation \citep{gong2022diffuseq},  among others, see \citet{yang2023diffusion} for a comprehensive overview of the latest advances in this topic.

\paragraph{Score-based generative models (SGMs).} Generative diffusion models aim at creating synthetic instances of a target distribution when only a genuine sample ($e.g.$, a dataset of real-life images) is accessible. It is crucial to note that the complexity of real data prohibits a thorough depiction of the distribution $\pi_{\mathrm{data}}$ through 
standard non-parametric density estimation strategies.
Score-based Generative Models (SGMs) are probabilistic models designed to address this challenge using two main phases. The first phase, the noising phase (also referred to as the forward phase), involves progressively perturbing the empirical distribution by adding noise to the training data until its distribution approximately reaches an {easy-to-sample} distribution $\pi_{\infty}$. The second phase involves learning to reverse this noising dynamics by sequentially removing the noise, which is referred to as the sampling phase (or backward phase). 
Reversing the dynamics during the backward phase would require in principle knowledge of the score function, i.e., the gradient of the logarithm of the density at each time step of the diffusion. 
To circumvent this issue, the score function is learned based on the evolution of the noised data samples and using a deep neural network architecture. When applying these learned reverse dynamics to samples from  $\pi_\infty$, we obtain a generative distribution that approximates  $\pi_{\rm data}$.

\paragraph{Related works.}
Significant attention has been paid to understanding the sources of errors that affect the quality of data generation associated with SGMs \citep{block2020generative, debortoli2022convergence, lee2022convergence, lee2023convergence, chen2023improved, chen2023sampling}. In particular, a key area of interest has been the derivation of upper bounds for distances or pseudo-distances between the training and generated sample distributions. 
Note that all the mathematical theory for diffusion models developed so far covers general time discretizations of time-homogeneous SGMs \citep[see][in the variance-preserving case]{song2019generative}, which means that the strength of the noise is prescribed to be constant during the forward phase.
\citet{debortoli2021,chen2023importance}
provided upper bounds in terms of total variation, by assuming smoothness properties of the score and its derivatives. On the other hand, the upper bounds in total variation and Wasserstein distances provided by \citet{lee2023convergence,gao2023wasserstein} also require smoothness assumptions on the data distribution, either involving non-explicit constants, or focusing on iteration complexity sharpness. 
More recently, \citet{conforti2023,benton2024nearly} established an upper bound in terms of Kullback--Leibler (KL) divergence avoiding strong assumptions about the score regularity, and relying on mild conditions about the data distribution (e.g., assumed to be of finite Fisher information w.r.t.\ the Gaussian distribution).
Regarding time-inhomogeneous SGMs, the central role of the noise schedule has already been exhibited in numerical experiments, see for instance \citet{chen2023importance,nichol2021improved,guo2023rethinking}.
However, a rigorous theoretical analysis of it is still missing. 

\paragraph{Contributions.} 
In this paper, we conduct a thorough mathematical analysis of the role of the noise schedule in score-based generative models.  We propose a unified framework for time-inhomogeneous SGMs, to conduct joint theoretical analyses in KL and Wasserstein metrics, 
with state-of-the-art set of assumptions, using exponential integration of the backward process.  In our opinion, these upper-bounds provide numerical insights into proper SGM training.   
\begin{itemize}
    \item  We establish an upper bound on the Kullback-Leibler divergence between the data distribution and the law of the SGM. This bound holds under the mildest assumptions used in the SGM literature and explicitly depends on the noise schedule used to train the SGM. The proof follows the same steps as \cite{conforti2023}. However, it requires to establish a Kullback-Leibler upper bound for an inhomogeneous forward diffusion which involves determining a non-asymptotic rate of convergence for the mixing time using Fokker-Planck equations and a log-Sobolev inequality that depends on the noise schedule,
 and not only on the diffusion time horizon, see Lemma~\ref{lem:mix}. In addition, taking into account the backward contraction for the diffusion process (Proposition~\ref{prop:back_contractivity_kl}) provides state-of-the-art results on mixing time convergence for SGM under the Ornstein-Uhlenbeck forward process, whether inhomogeneous or not. 
       \item By making additional assumptions on the Lipschitz and strong log-concavity properties of the score function, we establish a bound in terms of Wasserstein distance explicitly depending on the noise schedule. 
    This extends the similar result for the KL in the Gaussian setting. These  results are in the same line of work as \cite{bruno2023diffusion,gao2023wasserstein}, incorporating to the time inhomogeneous setting a refinement of the mixing time error based on an analysis of the modified score function.
    \item We illustrate, through numerical experiments, the upper bounds obtained in practice in regard of the effective empirical KL divergences and Wassertein metrics. These simulations highlights the relevancy of the upper bound, reflecting in practice the effect of the noise schedule on the quality of the generative distribution.
    Additionally, the simulations conducted provide theoretically-inspired guidelines for improving SGM training.
    For reproducibility purposes, the code for the numerical experiments is available at \url{ https://github.com/StanislasStrasman/Noise_Schedule_for_Score-based_Generative_Models}.
\end{itemize}

\section{Mathematical framework for SGMs}

\paragraph{Forward process.}
Denote as $\sched : [0,T] \mapsto \mathbb{R}_{>0}$ the noise schedule, assumed to be continuous and non decreasing.
Although originally developed using a finite number of noising steps \citep{dickstein2015, song2019generative, ho2020denoising,song2021score}, most recent approaches consider time-continuous noise perturbations through the use of stochastic differential equations (SDEs) \citep{song2021score}. 
Consider, therefore, a forward process given by
\begin{equation} \label{eq:forward-SDE:beta}
    \rmd  \overrightarrow{X}_t = - \frac{\sched(t)}{2 \sigma^2}  \overrightarrow{X}_t \rmd t + \sqrt{ \sched(t)}  \rmd B_t, \quad \ora{X}_0 \sim \pi_{\mathrm{data}}\eqsp. 
\end{equation} 
We denote by $p_t$ the density of $\ora X_t$ at time $t \in (0, T]$.
Note that, up to the time change  $t\mapsto \int_0^t\sched(s) / 2\rmd s$, this process corresponds to the standard Ornstein–Uhlenbeck (OU) process, solution to
\begin{align*}
    \rmd  \overrightarrow{X}_t = - \frac{1}{\sigma^2}  \overrightarrow{X}_t \rmd t + \sqrt{ 2}  \rmd B_t, \quad \ora{X}_0 \sim \pi_{\mathrm{data}}\eqsp,
\end{align*}
see, e.g., \citet[][Chapter 3]{karatzas2012brownian}. 
Due to the linear nature of the drift with respect to $(X_t)_t$, an exact simulation can be performed for this process (Section \ref{sec:exact_simulation_forward}).
The stationary distribution $\pi_{\infty}$ of the forward process is the Gaussian distribution with mean $0$ and variance $\sigma^2 \mathrm{I}_d$.
In the literature, when $\beta(t)$ is constant equal to $2$ (meaning that there is no time change), this diffusion process is referred  to as the Variance-Preserving SDE 
\citep[VPSDE,][]{debortoli2021, conforti2023, chen2023sampling}, leading to the so-called  
Denoising Diffusion Probabilistic Models \citep[DDPM,][]{ho2020denoising}.
Understanding the effects of the general diffusion model \eqref{eq:forward-SDE:beta}, in particular when reversing the dynamic,
remains a challenging problem, to which we devote the rest of our analysis.

\paragraph{Backward process.} The corresponding backward process is given by
\begin{align*}
\left\{
\begin{array}{l}
\rmd \overleftarrow{X}_t  =\eta(t,\overleftarrow{X}_t)\rmd t +   \sqrt{\bar \sched(t)} \rmd B_t, \\
\ola X_0 \sim \pi_{\infty},
\end{array}
\right.
\text{ with }
\left\{
\begin{array}{l}
\bar \sched(t) \eqdef \sched(T- t) \\ 
\eta(t,\overleftarrow{X}_t) \eqdef \frac{\bar \sched (t)}{2\sigma^2}\overleftarrow{X}_t + \bar \sched (t) \nabla \log  p_{T-t}\left( \ola{X}_t \right).
\end{array}
\right.
\end{align*}

We consider the marginal time distribution of the forward process divided by the density of its stationary distribution, introducing 
\begin{align}
\label{eq:renormalized_density}
\forall x \in \mathbb{R}^d, \quad \tilde p_{t}(x) \eqdef {p_t(x)} / {\gausspdf_{\sigma^2}(x)}, 
\end{align}
where $\gausspdf_{\sigma^2}$ denote the density function of $\pi_\infty$, a Gaussian distribution with mean $0$ and variance $\sigma^2 \mathrm{I}_d$.
Thus, the backward process can be rewritten as
\begin{align}\label{eq:backward_SDE}
\rmd \overleftarrow{X}_t =  \bar \eta \left(t,\overleftarrow{X}_t\right)  \rmd t  +  \sqrt{\bar \sched(t)} \rmd B_t, \quad \ola X_0 \sim \pi_{\infty},
\end{align}
where $\bar \eta(t,\overleftarrow{X}_t) \eqdef  - \frac{\bar \sched (t)}{2 \sigma^2} \overleftarrow{X}_t + \bar \sched (t) \nabla \log  \tilde p_{T-t}( \ola{X}_t)$. 
The benefit of using the renormalization $\tilde{p}_t$ in our analysis results in considering the backward equation as a perturbation of an OU process. This trick is crucial to highlight the central role of the relative Fisher information in the performance of the SGM. It has already been used by \citet{conforti2023}.

\paragraph{Score estimation.}
Simulating the backward process means knowing how to operate the score.
However, the (modified) score function $\nabla \log \tilde p_t(x)= \nabla \log p_t(x) + x/\sigma^2$ cannot be evaluated directly, because it depends on the unknown data distribution. To work around this problem, the score function $\nabla \log p_t$ needs to be estimated. In \citet{hyvarinen2005estimation}, the authors proposed to estimate the score function associated with a distribution by minimizing the expected $\mathrm{L}^2$-squared distance between the true score function  and the proposed approximation. In the context of diffusion models, this is typically done with the use of a deep neural network architecture $s_{\param} : [0,T] \times \mathbb{R}^d \mapsto \mathbb{R}^d$ parameterized by $\theta \in \Theta$,  and trained to minimize:
\begin{align} \label{eq:explicit_score_matching_objective}
\mathcal{L}_{\rm explicit} (\theta) &= \mathbb{E} \left[ \left\|  s_{\theta} \left(\tau, \ora X_{\tau}\right) - \nabla \log p_{\tau} \left(\ora X_{\tau}\right)  \right\|^2 \right]\eqsp,
\end{align}
with $\tau \sim \mathcal{U}(0,T)$ independent of the forward process $(\ora X_t)_{t\geq 0}$. However, this estimation problem still suffers from the fact that the regression target is not explicitly known. A tractable optimization problem sharing the same optima can be defined though, through the marginalization over $\pi_{\rm data}$ of $p_{\tau}$ \citep[see][]{Vincent,song2021score}: 
\begin{align}
\label{eq:cond_score_matching_objective}
\mathcal{L}_{\mathrm{score}}(\param) = \mathbb{E}\left[ \left\|s_{\param} \left( \tau, \overrightarrow{X}_{\tau} \right) - \nabla \log p_\tau \left(\ora X _\tau| X _0 \right)\right\|^2\right]
\end{align}
where $\tau$ is uniformly distributed on $[0,T]$, and independent of $X_0\sim \pi_{\mathrm{data}}$ and $\ora X _\tau \sim p_\tau(\cdot| X _0)$. This loss function is appealing as it only requires to know the transition kernel of the forward process. In \eqref{eq:forward-SDE:beta}, this is a Gaussian kernel with explicit mean and variance. 

\paragraph{Discretization.} 
Once the score function is learned, it remains that, in most cases, the backward dynamics no longer enjoys a linear drift, which makes its exact simulation challenging.  
To address this issue, one solution is to discretize the continuous dynamics of the backward process. 
In this way, \citet{song2021score} propose an Euler-Maruyama (EM) discretization scheme  in which both the drift and the diffusion coefficients are discretized recursively (see \eqref{def_Euler_Maruyama}). 
The Euler Exponential Integrator \citep[EI, see ][]{durmus2015quantitative}, as already used in \citet{conforti2023},
only requires to discretize the part associated with the modified score function.
Introduce $ \tilde s_{\theta} (t, x) \eqdef s_{\theta}(t, x) + x/\sigma^2$ and consider the regular time discretization $0 = t_0 \leq t_1 \leq \dots \leq t_{N}= T$. Then,  $(\ola X_t^{\param})_{t \in [0,T]}$ is such that, for $t \in  [ t_k , t_{k+1} ]$,
\begin{align}
\label{eq:backward_EI_main}
\rmd \overleftarrow{X}^{\param}_t  &=  \bar \sched (t)\left( - \frac{1}{2 \sigma^2} \overleftarrow{X}^{\param}_t +  \tilde s_{\param} \left( T-{t_k}, \ola{X}_{t_k}^{\param}  \right) \right)  \rmd t  
+  \sqrt{\bar \sched(t)} \rmd B_t\eqsp.
\end{align} 
This scheme can be seen as a refinement of the classical EM one as it handles the linear drift term by integrating it explicitly. In addition, $(\ola X_t^{\param})_{t \in \{t_0, \hdots, t_N\}}$ can be sampled exactly, see Appendix~\ref{sec:app:notations}. 
We consider therefore such a scheme in our further theoretical developments.

\section{Non-asymptotic Kullback-Leibler bound}
\label{sec:main}

In this section, we provide a theoretical analysis of the effect of the noise schedule used when training an SGM. Its impact is scrutinized through a bound on the KL divergence between the data distribution and the generative one.

\paragraph{Statement.}   
The data distribution $\pi_{\rm data}$ is assumed to be absolutely continuous with respect to the Gaussian measure $\pi_{\infty}$. Define the relative Fisher information $\mathcal{I}(\pi_{\rm data} | \pi_{\infty})$ by 
\begin{align*}
\mathcal{I}(\pi_{\rm data} | \pi_{\infty}) \eqdef
\int \left\| \nabla \log \left( \frac{\rmd \pi_{\rm data}}{\rmd \pi_{\infty}} \right) \right\|^2 \, \rmd \pi_{\rm data}\eqsp, 
\end{align*}
and consider the following assumptions.

\begin{hypH}
\label{hyp:sched}
    The noise schedule $\sched$ is continuous, positive, non decreasing and such that $\int_0^\infty \beta(t)\rmd t = \infty$.
\end{hypH}

\begin{hypH}
\label{hyp:fisher_info}
    The data distribution is such that $\mathcal{I}(\pi_{\rm data} | \pi_{\infty})<\infty$.
\end{hypH}

\begin{hypH}
\label{hyp:novikov}
The NN parameter $\param\in\Theta$ and the schedule $\sched$ satisfy
    \begin{align*}
    \mathbb{E}\left[
        \exp\left\{
            \frac{1}{2} \int_0^T \bar \sched (t)
                \left\|  \left(  \tilde s\left(T-t,\ola X_t\right)   -\tilde s_{\param} \left( T-t_k, \ola{X}_{t_k}    \right) \right)  \right\|^2  \rmd t\right\}\right]&<\infty\eqsp,
    \end{align*}
    with 
    $\tilde s(t,x)\!\eqdef\! \nabla \log \tilde p_{t}(x)$ and $\tilde{p}_t$ defined in \eqref{eq:renormalized_density}.
\end{hypH}

Assumption H\ref{hyp:sched} is necessary to ensure that the forward process converges to the stationary distribution when the diffusion time tends to infinity. 
Assumption H\ref{hyp:fisher_info} is inherent to the data distribution, as it involves only the $L^2$-integrability of the score function.  Such a kind of hypothesis has already been considered in the literature, see \citet{conforti2023}.
We stress that, in this section, we do not require extra assumptions about the smoothness of the score function.
Lastly, Assumption H\ref{hyp:novikov} is the guarantor of a good approximation of the score by the neural network $\tilde{s}_\theta$, weighted by the level of noise in play.
We  are now in position to provide an upper bound for the relative entropy between the distribution $\pihat$ of samples obtained from \eqref{eq:backward_EI_main}, and the target data distribution $\pi_{\rm data}$.

\begin{theorem}
\label{th:main}
Assume that H\ref{hyp:sched}, H\ref{hyp:fisher_info} and H\ref{hyp:novikov} hold. Then, 
\begin{equation*}
\kl \left( \pi_{\rm data} \middle|\middle| \pihat  \right) \leq \mathcal{E}_1^{\rm KL}(\sched) + \mathcal{E}_2^{\rm KL}(\param,\sched) + \mathcal{E}_3^{\rm KL}(\sched)\eqsp,
\end{equation*}
where 
\begin{align*}
    &\mathcal{E}_1^{\rm KL}(\sched) = \kl\left( \pi_{\rm data} \middle|\middle| \pi_\infty  \right)  \exp\Big\{- \frac{1}{\sigma^2} \int_0^T \sched(s) \rmd s\Big\}\eqsp,\\
    &\mathcal{E}_2^{\rm KL}(\param,\sched) = \sum_{k = 0}^{N-1}  \mathbb{E} \bigg[  \Big\|   \nabla \log \tilde p_{T-t_k}\Big(\overrightarrow X_{T-t_k}\Big)  - \tilde s_{\param}\Big(T-t_k,  \overrightarrow X_{T-t_k}\Big) \Big\|^2 \bigg]
    \int^{T-t_k}_{T-t_{k+1}} \sched (t) \rmd t\eqsp,\\
    &\mathcal{E}_3^{\rm KL}(\sched) = 2 h  \sched (T) \mathcal{I}(\pi_{\rm data} | \pi_{\infty} )\eqsp,
\end{align*}

with $h \eqdef \sup_{k \in \{1 , \ldots , N \} } (t_{k}-t_{k-1})$ small enough and $t_0\eqdef 0$.
\end{theorem}

The obtained bound is composed of three terms, all depending on the noise schedule, through either its integrated version over the diffusion time, or its final value at time $T$. 
If the result was derived for the EI discretization scheme, it could be adapted to the Euler one up to minor technicalities.
Remark also that using Pinsker's inequality, the obtained bound could be transferred in terms of total variation.

\paragraph{Dissecting the upper bound. } The upper bound of Theorem \ref{th:main} involves  three different types of error that affect the training of an SGM.
The term $\mathcal{E}_1^{\rm KL}$ represents the \textit{mixing time} of the OU forward process, 
arising from the practical limitation of considering the forward process up to a finite time $T$. Indeed, $\mathcal{E}_1^{\rm KL}$ is shrinked to 0 when $T$ grows to infinity. Note that the multiplicative term in $\mathcal{E}_1^{\rm KL}$ corresponds to the KL divergence between $\pi_{\mathrm{data}}$ and $\pi_\infty$ which is ensured to be finite by Assumption H\ref{hyp:fisher_info}.
The second term $\mathcal{E}_2^{\rm KL}$   corresponds to the \textit{approximation error}, which stems from the use of a deep neural network to estimate the score function. 
Note that if we assume that the error of the score approximation is uniformly (in time) bounded by $M_\theta$ \citep[see][Equation (8)]{debortoli2021}, the term $\mathcal{E}_2^{\rm KL}$ admits as a crude bound $M_\theta \int_0^T \beta(t) \rmd t$, with the disadvantage of exploding when $T\to + \infty$. 
Otherwise, by considering \citet[][Assumption H3]{conforti2023}, one can make this bound finer and finite, by balancing the quality of the score approximation, the discretization grid and the final time $T$. 
Finally, $\mathcal{E}_3^{\rm KL}$ is the  \textit{discretization error} of the EI discretization scheme. This last term vanishes as the discretization grid is refined (i.e., $h\to 0$).

\paragraph{Comparison with existing bounds.}
Under perfect score approximation, 

and infinitely precise discretization (i.e., when $\mathcal{E}_2^{\rm KL}(\theta,\beta) = \mathcal{E}_3^{\rm KL}(\beta) =0$), we recover that the Variance Preserving SDE  \citep[VPSDE,][]{debortoli2021, conforti2023, chen2023sampling} converge exponentially fast to the target distribution.  Beyond this idealized setting,
the bound established in Theorem \ref{th:main} recovers that of
\citet[][Theorem 1]{conforti2023} when choosing $\sched(t)=2$, $\sigma^2 = 1$, $T=1$, and using a discretization step size $h\leq 1$. 

\paragraph{Refined analysis of the mixing time error}
Still assuming  ``perfect score approximation'' and infinitely precise discretization (i.e., $\mathcal{E}_2^{\rm KL}(\theta,\beta) = \mathcal{E}_3^{\rm KL}(\beta) =0$), 
one can assess the sharpness of the term $\mathcal{E}_1^{\rm KL}(\beta)$ in the upper bound of Theorem \ref{th:main}.
In particular, when restricting the data distribution to be Gaussian $\mathcal{N}(\mu_0, \Sigma_0)$, 
one can exploit the backward contraction assuming that 
$\lambda_{\max}(\Sigma_0)\leq\sigma^2$, where $\lambda_{\max}(\Sigma_0)$ denotes the largest eigenvalue of $\Sigma_0$.  
In this specific case, we can obtain a refined version for $\mathcal{E}_1^{\rm KL}$ (see Proposition~\ref{prop:back_contractivity_kl}), given by 
 \begin{align}
    \label{eq:KL_refined_mixing_error}
     \kl\left(\pi_{\rm data} \| \pi_\infty Q_T \right) \le \kl &\left(\pi_{\rm data} \| \pi_\infty \right) \exp\Big( - \frac{2}{\sigma^2} \int_0^T \beta(s) \rmd s \Big),
\end{align}
where $(Q_t)_{0\leq t \leq T}$ is the Markov semi-group associated with the backward SDE. 
This idea is exploited in Section \ref{sec:wasserstein} to establish Wasserstein bounds for more general data distributions than Gaussian, but requiring extra regularity of the score.

\section{Non-asymptotic Wasserstein bound}
\label{sec:wasserstein}

In the literature, much attention is paid to derive upper bounds with other metrics such as the $\mathcal{W}_2$ distance, which has the advantage to be a distance and to have easier-to-handle and implementable estimators. 
In \citet{lee2023convergence}, the authors obtain a control for the 2-Wasserstein and total variation distances. However, those results rely on additional assumptions on $\pi_{\mathrm{data}}$ (which is assumed to have bounded support for instance in \citet{debortoli2022convergence}).

\paragraph{Regularity assumptions. }
We consider extra regularity assumptions about the modified marginal density $\tilde p_t $ at any time of the diffusion.
\begin{hypH} \label{hyp:score_regularity}
\begin{enumerate}[(i)]
\item \label{hyp:strong-log-concavity}
    For all $t\geq 0$, there exists $C_t\geq0$ such that for all $ x,y\in\R^d$,
    \begin{equation*}
    \left(\nabla \log \tilde p_t(y) \!- \!\nabla \log \tilde p_t(x)\right)^\top\!\!(x-y)\geq C_t\left\|x-y\right\|^2\eqsp.
    \end{equation*}
\item \label{hyp:score_lipschitz}
For all $t\geq 0$, there exists $L_t\geq0$ such that $\nabla \log \tilde{p}_t$ is $L_t$-Lipschitz continuous.
\end{enumerate}
\end{hypH}

The strong log-concavity \eqref{hyp:strong-log-concavity} \citep[see, e.g.,][]{saumard2014log} plays a crucial role in terms of contraction of the backward SDE. 
Classical distributions satisfying H\ref{hyp:score_regularity}\eqref{hyp:strong-log-concavity} include logistic densities restricted to a compact, or Gaussian laws with a positive definite covariance matrix, see \citet{saumard2014log} for other examples. 
We observe, notably, that when the density of the data distribution is log-concave, this property propagates within the probability flow $(\tilde{p}_t)_{0\leq t\leq T}$ (see Proposition \ref{prop:from_C_0_to_C_t}).  
Similar conclusions can be drawn regarding the Lipschitz continuity of the score (Proposition \ref{prop:from_L_0_to_L_t}). 
This property is formalized in the Lemma~\ref{lem:gaussian:contract} for Gaussian distributions.

\begin{lemma}\label{lem:gaussian:contract}
Assume that $\pi_{\rm data}$ is a Gaussian distribution $\mathcal{N}(\mu_0,\Sigma_0)$, such that  $\Sigma_0$ is invertible and $\lambda_{\max}(\Sigma_0) < \sigma^2$.
Let $m_t:=\exp(-\int_0^t \beta(s)/2\sigma^2 \rmd s)$.
Then, the probability flow $\tilde{p}_t$ given by \eqref{eq:forward-SDE:beta} initialized at $\pi_{\rm data}$ is $C_t$-strongly log concave, with 
\begin{align*}
C_t\eqdef \frac{m^2_t\left(\sigma^2 - \lambda_{\max}(\Sigma_0) \right)}{m_t^2\lambda_{\max}(\Sigma_0) + \sigma^2\left(1-m_t^2\right)}\eqsp.
\end{align*}
In addition, the associated score $\nabla \log \tilde{p}_t$ is $L_t$-Lipschitz continuous with 
$$L_t := \min \left\{ \frac{1}{\sigma^2\left(1-m^2_t\right)} ;  \frac{1}{ \lambda_{\rm min} (\Sigma_0 ) m^2_t} \right\} + \frac{1}{\sigma^2} \eqsp.
$$ 
\end{lemma}
This result, restricted to the Gaussian case, sets the focus on the importance of calibrating the parameter $\sigma^2$ depending on the covariance structure of the data distribution, in order to enhance strong log concavity of the probability flow through the diffusion.

\paragraph{Error bound.}
To establish a 2-Wasserstein bound explicitly depending on the noise schedule, we consider the following additional assumptions, respectively about uniform approximation of the score, and Lipschitz continuity in time of the renormalized score.
\begin{hypH}
\label{hyp:sup_approx}
There exists $\varepsilon \geq 0$ such that
$\underset{k \in \{0,..,N-1 \} }{\sup} \left\| \tilde s\left(T-t_k, \bar{X}_{t_{k}}^{\theta} \right) - \tilde s_{\theta} \left( T-t_k, \bar{X}_{t_{k}}^{\theta} \right) \right\|_{L_2}
    \leq \varepsilon \eqsp.
$
\end{hypH}

\begin{hypH}
\label{hyp:lipschitz_score_time}
For a regular discretization $\{t_k , 0\leq k \leq N\}$ of $[0,T]$ of constant step size $h$, there exists $M \geq 0$ such that
\begin{align*}
    &\sup_{k \in \{0,..,N-1 \} } \sup_{t_k\leq t \leq t_{k+1}} \left\| \tilde s\left(T-t, x \right) - \tilde s\left(T-t_k,x \right) \right\|_{L_2} 
 \leq M h (1+\|x\|) \eqsp.
\end{align*}
\end{hypH}
We now have all the ingredients to present our theoretical guarantee in terms of Wasserstein distance.

\begin{theorem} \label{thm:wasserstein_bound}
    Assuming H\ref{hyp:score_regularity}, H\ref{hyp:sup_approx} and H\ref{hyp:lipschitz_score_time} and that the time step $h$ is small enough, it holds that
    \begin{align}
    \label{eq:W2_bound}
      \Wc_2 \left( \pi_{\rm data} ,\pihat \right) \leq \mathcal{E}_1^{\rm \mathcal{W}_2}(\sched) + \mathcal{E}_2^{\rm \mathcal{W}_2}(\param,\sched)
    \end{align}
    \begin{align*}
    \text{with }\qquad &\mathcal{E}_1^{\rm \mathcal{W}_2}(\sched) =\Wc_2\left(\pi_{\rm data},\pi_\infty\right)
    \exp\left(- 
    \int_0^T \frac{\beta(t)}{\sigma^2}\left(1+  C_t\sigma^2\right)\rmd t\right) ,\\
    &\mathcal{E}_2^{\rm \mathcal{W}_2}(\param,\sched) =  \sum_{k=0}^{N-1}
        \left(\int_{t_k}^{t_{k+1}} \bar L_{t} 
        \bar \beta (t) \rmd t \right)
         \left(
        \frac{\sqrt{2h\beta(T)}}{\sigma}  + \frac{h\beta(T)}{2\sigma^2}
        +\int_{t_k}^{t_{k+1}}  2\bar{L}_{t} \bar{\beta} (t) \rmd t
    \right) B  \\
    & \qquad\qquad\qquad + \varepsilon T\beta(T) +  M h T \beta(T)\left(1+2B \right) , \\
    &B= (  \mathbb{E} [ \| X_0 \|^2 ] + \sigma^2 d )^{1/2}\eqsp, \eqsp\mathrm{and\,for\,all\,}t\in[0,T], \bar L_t = L_{T-t}\eqsp .
\end{align*}
\end{theorem}

In Theorem \ref{thm:wasserstein_bound}, we exploit the contraction entailed by Assumption H\ref{hyp:score_regularity}\eqref{hyp:strong-log-concavity} of the backward diffusion processes in top of that of the forward phase. To our knowledge, in the state-of-the-art results, this feature has never been considered. This idea leads to an improvement of all the existing bounds in Wasserstein metrics, by refining their mixing time term. 
The previous result can be established when the target distribution has a Lipschitz continuous score and is strongly log-concave: by propagating these properties, the constants $L_t$ and $C_t$ can be characterized in function of $L_0$ and $C_0$ (see Propositions \ref{prop:from_C_0_to_C_t} and \ref{prop:from_L_0_to_L_t}).

\begin{corollary}
\label{cor:wasserstein_time0}
    Assume that $\nabla \log \tilde{p}_{\rm data}$ is $L_0$-Lipschitz, that $\log \tilde{p}_{\rm data}$ is $C_0$-strongly concave such that $C_0>1/\sigma^2$. 
    Under Assumption  H\ref{hyp:sup_approx} and H\ref{hyp:lipschitz_score_time}, with a time step $h$ small enough, 
    \begin{align*}
        \Wc_2 \Big( \pi_{\rm data} ,\pihat \Big)
        &\!\leq 
        \Wc_2\left(\pi_{\rm data},\pi_\infty\right)
    \exp\Big(\!\!-\!\!
    \int_0^T \!\frac{\beta(t)}{\sigma^2}\left(1+  C_t\sigma^2\right)\rmd t\Big) + c_1 \sqrt{h} + c_2 h + \varepsilon T \beta(T)\eqsp,
    \end{align*}
    with $c_1 = L_0 \beta(T) T \sqrt{2\beta(T)}/\sigma$ and $c_2 = \beta(T) T \left( L_0  \left( 1/(2\sigma^2) + 2L_0\right) \beta(T) B + M (1+2B)\right)$. 
\end{corollary}

This provides an easy-to-handle upper bound in Wasserstein distance, encompassing the three types of error (e.g., mixing time, score approximation and discretization error), for Lipschitz scores and strongly-log concave distributions. We remark that it also exhibits an extra term in $\sqrt{h}$ compared to the more general $\kl$ bound obtained under milder assumptions. Note however that this term is in line with what can be found in the literature for Wasserstein bound for SDE approximation \citep[see][]{alfonsi2015optimal}.

By considering early stopping techniques (typical in the literature on SGM), we could adapt these bounds to more general data scenarios ($e.g.$, not strongly log concave) exploiting the regularization properties of the convolution with a Gaussian kernel.

\section{Evaluation of the theoretical upper bounds}
\label{sec:exp}

\begin{wrapfigure}[12]{r}{0.35\textwidth} 
\vspace{-0.5cm}
\centering 
\includegraphics[width=0.8\linewidth]{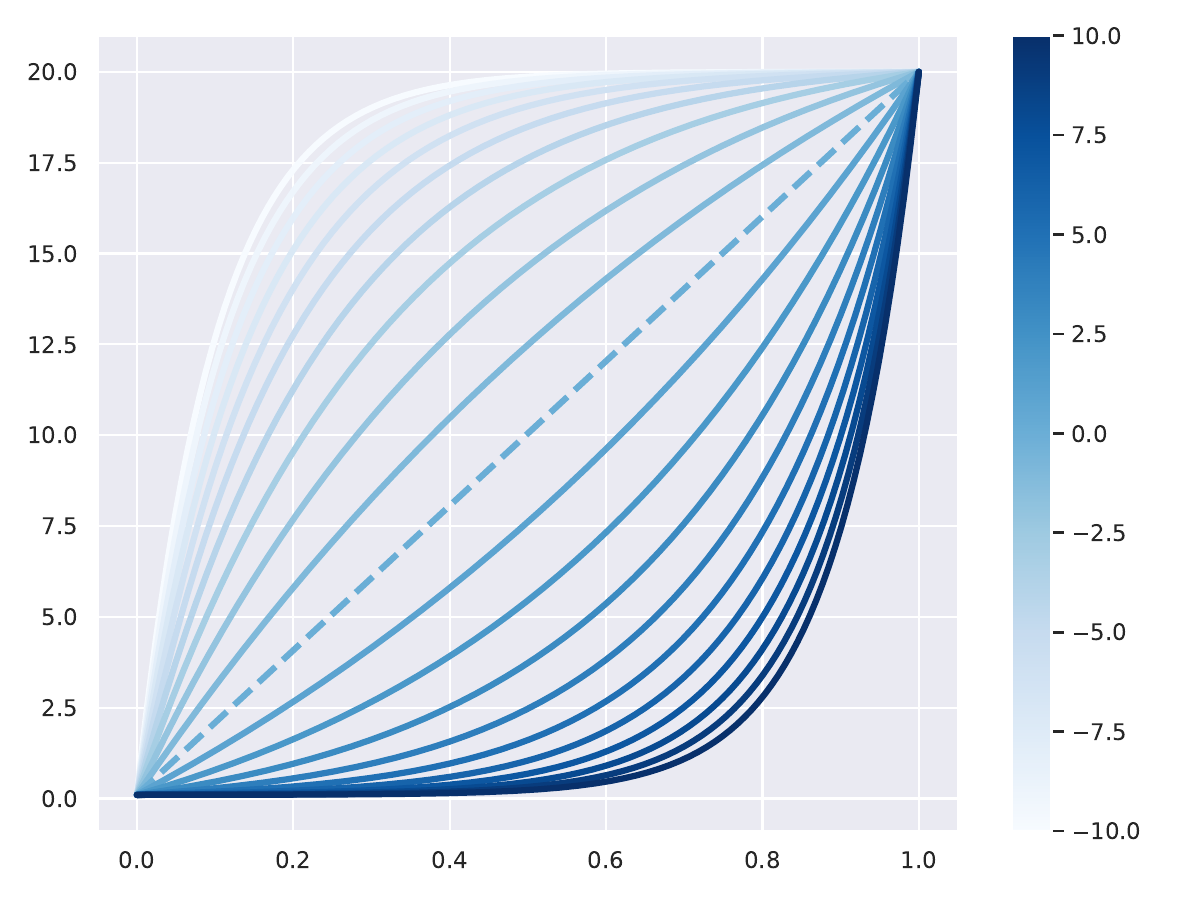}    
\caption{Noise schedule $\beta_a$ over time for $a \in \{-10, -9,..,10 \}$ with the linear schedule $a=0$ shown as a dashed line.}
    \label{fig:noise_schedules}
\end{wrapfigure}
The goal of this section is to numerically illustrate the validity of the theoretical bounds obtained in Theorem \ref{th:main} and Theorem \ref{thm:wasserstein_bound}. More precisely, we aim at unraveling the contributions of each error term of the upper bounds. We consider a simulation design where the target distribution is known, and the associated constants of interest (i.e., the strong log concavity parameter, the Lipschitz constant, $\Wc_2\left(\pi_{\rm data},\pi_\infty\right)$, $\mathcal{I} \left( \pi_{\rm data} | \pi_{\infty} \right)$ or $\mathrm{KL}\left(\pi_{\rm data} || \pi_\infty\right)$) can be evaluated. 

The error bounds are assessed for different choices of noise schedules of the form 
\begin{align}
\label{eq:noise_schedules}
\beta_a(t) &\propto (\rme^{at} -1)/(\rme^{aT} -1),
\end{align}
with $a\in\mathbb{R}$ ranging from $-10$ to $10$ with a unit step size. We set $T=1$ and adjust schedules so that they all start at $\beta (0) = 0.1$ and end at $\beta (1) = 20$ (see Figure \ref{fig:noise_schedules}). This choice has been made so that when $a = 0$ the schedule is linear and matches exactly the  classical VPSDE implementation \citep{song2019generative, song2021score}.

\subsection{Gaussian setting}
\label{sec:exp_gaussian}

\paragraph{Target distributions.} We consider the setting where the true distribution $\pi_{\rm data}$ is Gaussian in dimension $d=50$ with mean ${\bf 1}_d$ and different choices of covariance structure:
\begin{enumerate}
    \item (Isotropic, denoted by $\pi_{\rm data}^{\mathrm{(corr)}}$) $\Sigma^{\mathrm{(iso)}} = 0.5 \mathrm{I}_d$. 
    \item (Heteroscedastic, denoted by $\pi_{\rm data}^{\mathrm{(heterosc)}}$) $\Sigma^{\mathrm{(heterosc)}}\in\mathbb{R}^{d\times d}$ is a diagonal matrix such that $\Sigma^{\mathrm{(heterosc)}}_{jj}=1$ for $1\leq j \leq 5$, and $\Sigma^{\mathrm{(heterosc)}}_{jj}=0.01$ otherwise. 
    \item (Correlated, denoted by $\pi_{\rm data}^{\mathrm{(corr)}}$) $\Sigma^{\mathrm{(corr)}}\in\mathbb{R}^{d\times d}$ is a full matrix whose diagonal entries are equal to one and the off-diagonal terms are   $\Sigma^{\mathrm{(corr)}}_{jj'}=1/\sqrt{|j-j'|}$ for $1\leq j\neq j' \leq d$. 
\end{enumerate}

\paragraph{SGM simulations.} We simulate $\widehat{\pi}_N^{(\beta_a, \theta)}$ from SGM using the forward process defined in \eqref{eq:forward-SDE:beta} with $t \mapsto \beta_a(t)$ for the noise schedule. The score is learned via a dense neural network with 3 hidden layers of width 256 over 150 epochs (see Figure \ref{fig:nn_architecture}) trained to optimize $\mathcal{L}_{\rm explicit}$ \eqref{eq:explicit_score_matching_objective}. 
This is feasible because the score is analytically derived when $\pi_{\rm data}$ is Gaussian (Lemma \ref{lem:exactscore}). 
Numerical experiments have also been run with the commonly used  conditional loss  $\mathcal{L}_{\rm score}$, without changing the nature of the conclusions, see Appendix \ref{sec:conditional_training_SGM}. 
For backward process simulation, we use an Euler-Maruyama scheme with 500 steps, as being the most encountered discretization in practice. 
For each value of $a$, and each data distribution, we train the SGM using $n=10000$ training samples.

\paragraph{KL bound.} 
In Figure \ref{fig:gaussian_xp} (top), we compare the empirical KL divergence between $\pi_{\rm data}$ and samples from $\widehat{\pi}_N^{(\beta_a, \theta)}$ to the upper bound from Theorem~\ref{th:main}. We refer the reader to Appendix \ref{sec:appendix_KL_bound} for implementation details.
For Gaussian distributions, both the bound and KL divergence can be computed using closed-form expressions (see Lemma \ref{KL divergence Gaussian rvs} and \ref{Info Fisher Gaussian rvs}). In all scenarios the noise schedule significantly impacts the value of $\mathrm{KL}(\pi_{\rm data}\| \widehat{\pi}_N^{(\beta_a, \theta)})$, and thereby the quality of the learned distribution. 
In the isotropic case (Figure \ref{fig:gaussian_xp} (a) (top)), the behavior of the upper bound does not exactly match the one of $\mathrm{KL}(\pi_{\rm data}\| \widehat{\pi}_N^{(\beta_a, \theta)})$ suggesting that the refinement relying on contraction arguments specific to the Gaussian setting (see \eqref{eq:KL_refined_mixing_error}) is indeed more informative in such a case. When considering $\pi_{\rm data}^{\mathrm{(heterosc)}}$ and $\pi_{\rm data}^{\mathrm{(corr)}}$ (Figure \ref{fig:gaussian_xp} (b,c) (top)), the upper bound remains clearly relevant to assess the efficiency of the noise schedule used during training. In all these experiments, the KL upper bound indicates possible values for $a$ improving over the classical linear noise schedule. 

\begin{figure}[h!]
    \centering
    \begin{tabular}{ccc}
    \hspace{-1cm}
        \includegraphics[width=0.32\linewidth]{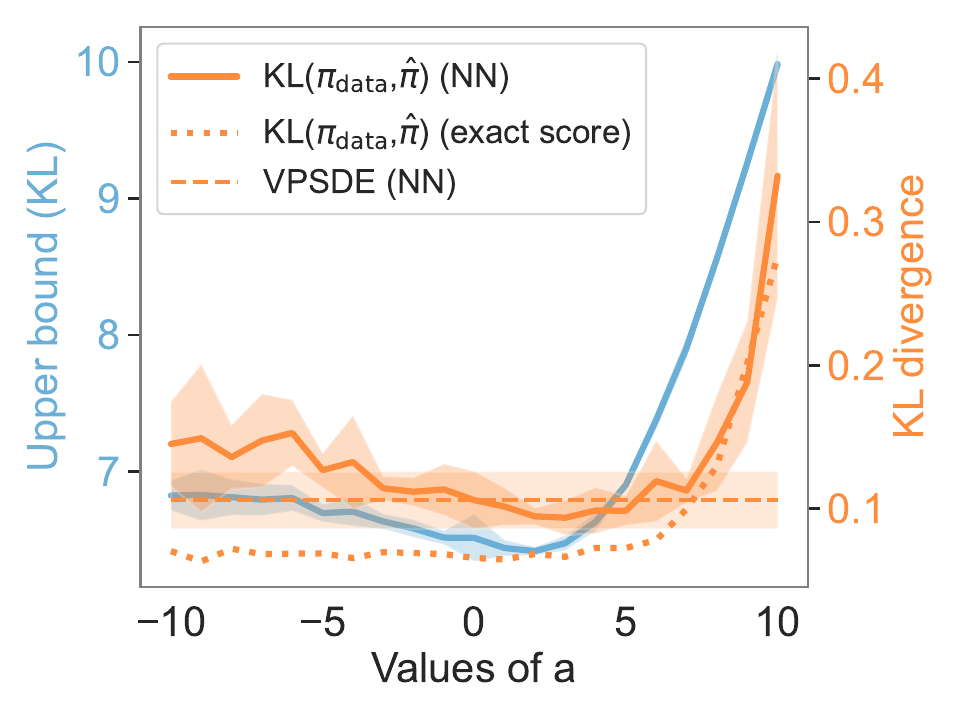}&
        \includegraphics[width=0.32\linewidth]{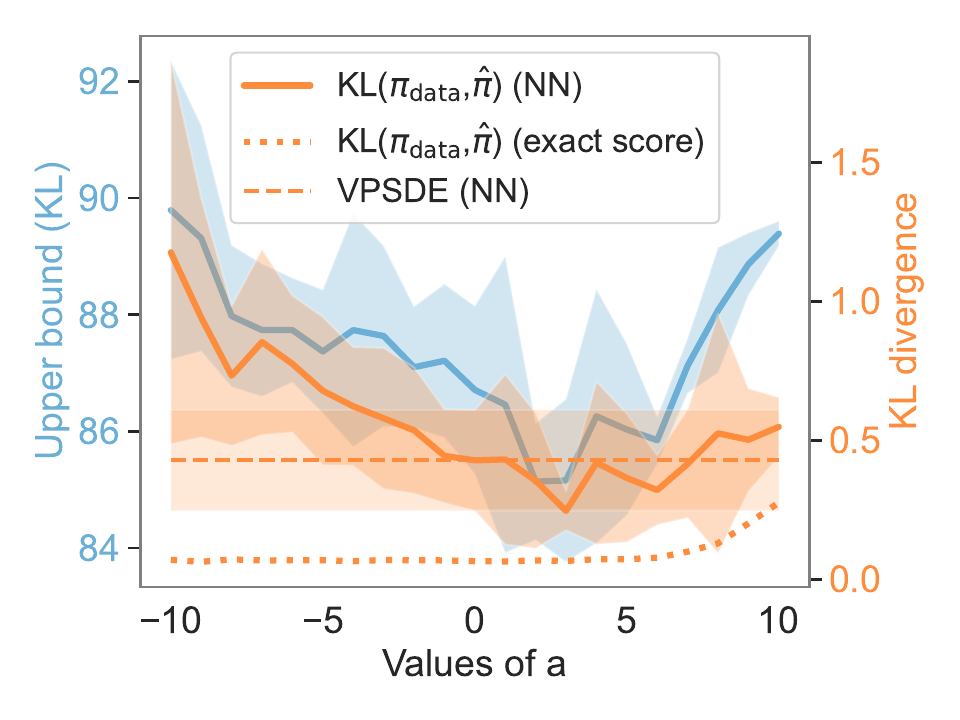}&
        \includegraphics[width=0.32\linewidth]{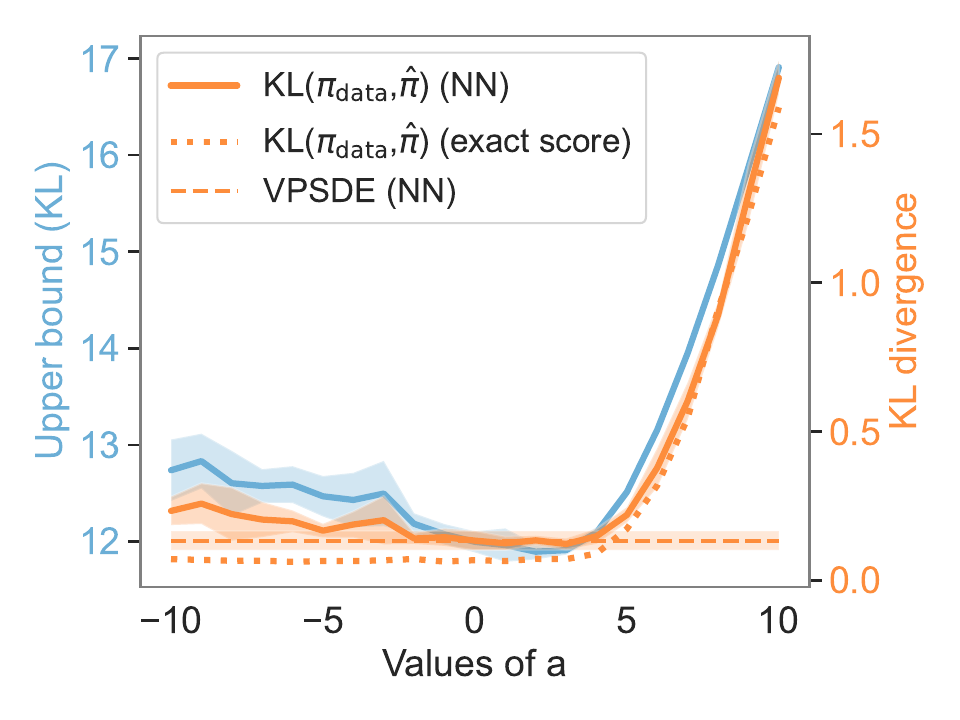} \\
        \hspace{-1cm}  
        \includegraphics[width=0.32\linewidth]{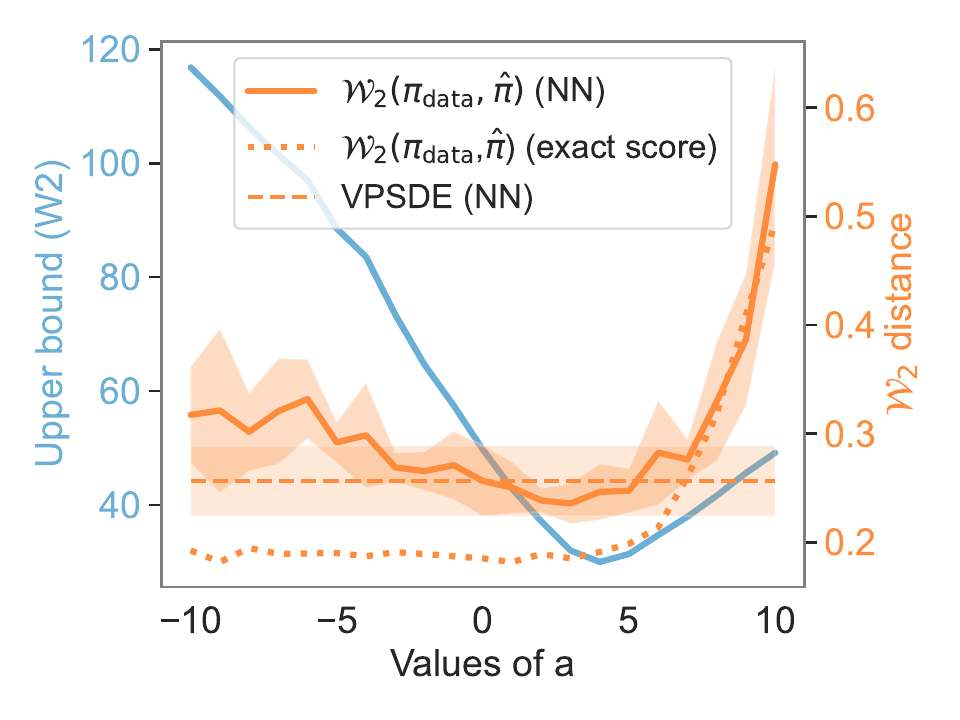}&
        \includegraphics[width=0.32\linewidth]{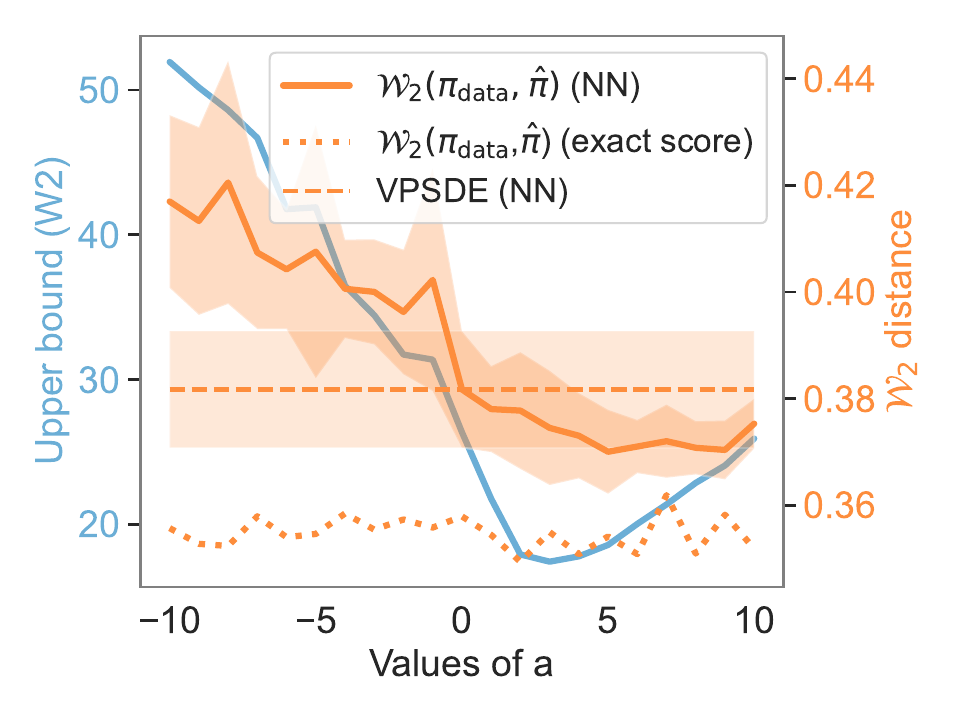}&
        \includegraphics[width=0.32\linewidth]{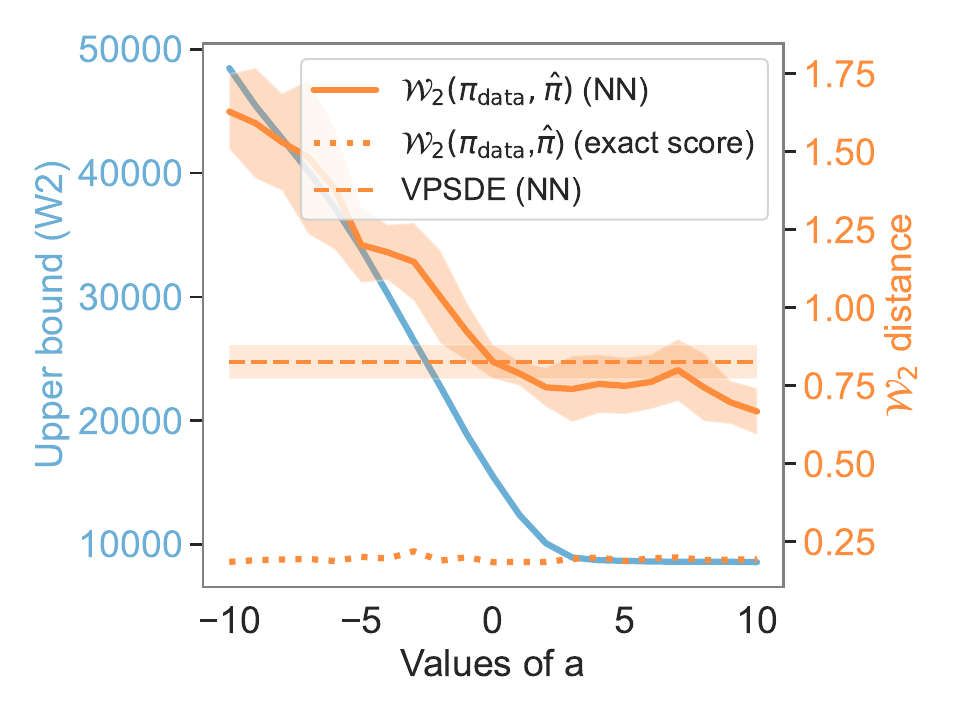} \\
        \hspace{-1cm} {(a) Isotropic setting 
        } & {(b) Heteroscedastic setting } & {(c) Correlated setting }
    \end{tabular}
    \caption{\label{fig:gaussian_xp} Comparison of the empirical KL divergence (top) and $\mathcal{W}_2$  distance (bottom) (mean ± std over 10 runs) between $\pi_{\mathrm{data}}$ and $\pihat$ (orange) and the related upper bounds (blue) from Theorem \ref{th:main} and Theorem \ref{thm:wasserstein_bound}  across parameter $a$ for noise schedule $\beta_a$, $d=50$. We also show the metrics for the linear VPSDE model (dashed line) and our model (dotted line) with exact score evaluation.
    }
\end{figure}

\paragraph{2-Wasserstein bound.} In Figure \ref{fig:gaussian_xp} (bottom), we compare the empirical $\mathcal{W}_2$ distance between $\pi_{\rm data}$ and samples from $\widehat{\pi}_N^{(\beta_a, \theta)}$ to the upper bound from Theorem \ref{thm:wasserstein_bound}. 
For Gaussian distributions, both the bound and the $\mathcal{W}_2$ distance can be computed using closed-form expressions (see Lemma  \ref{lem:gaussian:contract}, \ref{W2 Gaussian rvs}, and \ref{prop:bar_l_t}). 
For the isotropic case, the proposed $\mathcal{W}_2$ upper bound reflects the SGM performances, as already highlighted by the KL bound. 
However, in non-isotropic cases, the raw distributions $\pi_{\rm data}^{\mathrm{(heterosc)}} $ and $\pi_{\rm data}^{\mathrm{(corr)}}$ do not directly satisfy  Assumption \ref{hyp:score_regularity} (\ref{hyp:strong-log-concavity}) when the variance of the stationary distribution is set to 1. Therefore,
scaling the distributions in play becomes crucial for the theoretical $\mathcal{W}_2$ upper bound to hold. 
That is why we propose the following preprocessing: train an SGM with centered and standardized samples of covariance $\Sigma^{\text{(stand)}}$ rescaled in turn by a factor $1/(2\lambda_{\rm max}(\Sigma^{\text{(stand)}}))^{1/2}$. This choice ensures that $\lambda_{\max} \left( \Sigma^{\text{(scaled)}}\right) < \sigma^2 = 1$, for $\Sigma^{\text{(scaled)}}$ the resulting covariance matrix, and thus the strong log-concavity of $\tilde{p}_0 = p_{\rm data} / \varphi_{\sigma^2}$. 
We call  $\widehat{\pi}_{N, \text{scaled}}^{(\beta_a, \theta)}$ the resulting generative distribution, and the evaluated metrics is adjusted (see \eqref{eq:w2_bound_recale}) to ensure a fair numerical comparison. 
After this preprocessing, not only the $\mathcal{W}_2$ upper bound of Theorem \ref{thm:wasserstein_bound} aligns with the empirical performances but the SGM performances can be also boosted (see degraded empirical performances on raw distributions in Appendix \ref{sec:appendix_W2_bound}). 
This highlights the importance of properly calibrating the training sample to the stationary distribution of the SGM. 
Note that data normalization does not only enforce the strong log-concavity of the modified score at time 0, but can lower the ratio $L_0/C_0$. To see this, consider the heteroscedastic case, for which $\lambda_{\min}(\Sigma^{\rm (heterosc)})/\lambda_{\max}(\Sigma^{\rm (heterosc)})=100$, whereas $\lambda_{\min}(\Sigma^{\rm (scaled)})/\lambda_{\max}(\Sigma^{\rm (scaled)})=1$ after scaling.  
This Gaussian set-up reveals that data renormalization improves the conditioning of the covariance matrix, and thereby the conditioning of SGM training. In particular, this is captured in the upper bound of Theorem \ref{thm:wasserstein_bound} by limiting the growth of $L_t$ and inducing a more balanced second term. 

We now consider a varying dimension in $\{5, 10, 25, 50\}$, and we compare the empirical $\mathcal{W}_2$ distance obtained by (i) $\beta_0$ the classical VPSDE \citep[][]{song2021score}, with a linear noise schedule (i.e., $a=0$), (ii) $\beta_{\cos}$ the SGM with cosine schedule \citep{nichol2021improved}, and
(iii) $\beta_{a^\star}$ the SGM with parametric schedule with $a = a^\star$ approximately minimizing the upper bound from Theorem \ref{th:main}. In Figure \ref{fig:gaussian_xp_dim_w2}, we observe that the SGMs run with $\beta_{a^\star}$ consistently outperforms those run with linear schedule $\beta_{0}$ providing significant improvement in the data generation quality. It  displays lower average $\mathcal{W}_2$ distances between $\pi_{\rm data}$ and the generated sample distribution, but also reduces the standard deviation of the resulting $\mathcal{W}_2$ distances yielding more stable generation (see Table \ref{tab:W2_table_gaussian}). These performances are comparable to, and often surpass, those achieved with state-of-the-art schedules like the cosine schedule, particularly in higher dimensions.

\begin{figure*}[h]
    \centering
    \begin{tabular}{ccc}
        \includegraphics[width=0.3\linewidth]{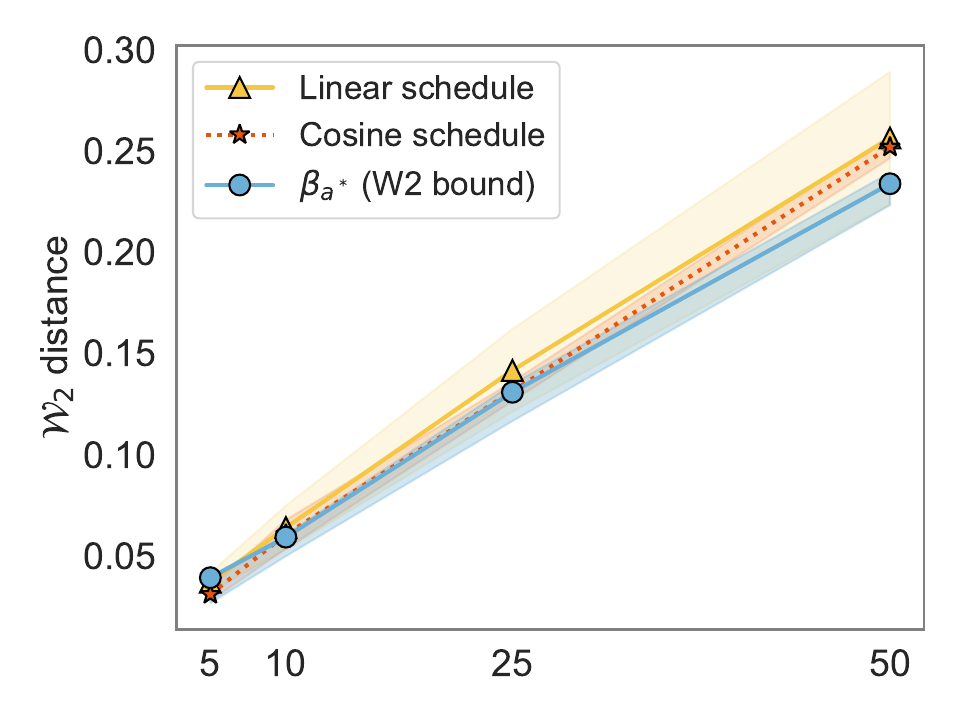}&
        \includegraphics[width=0.3\linewidth]{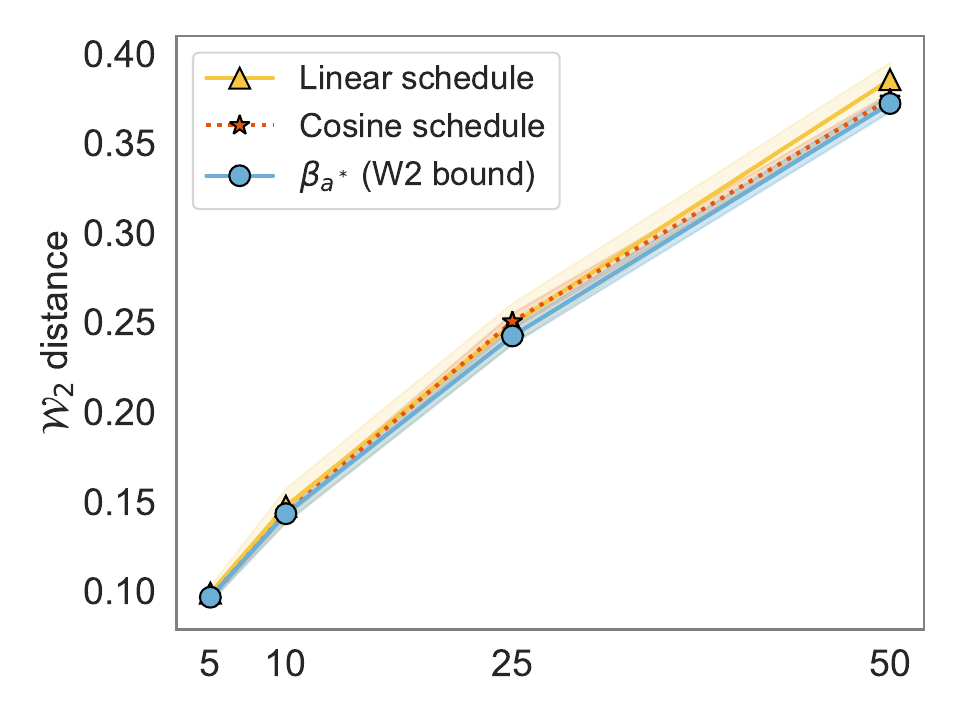}&
        \includegraphics[width=0.3\linewidth]{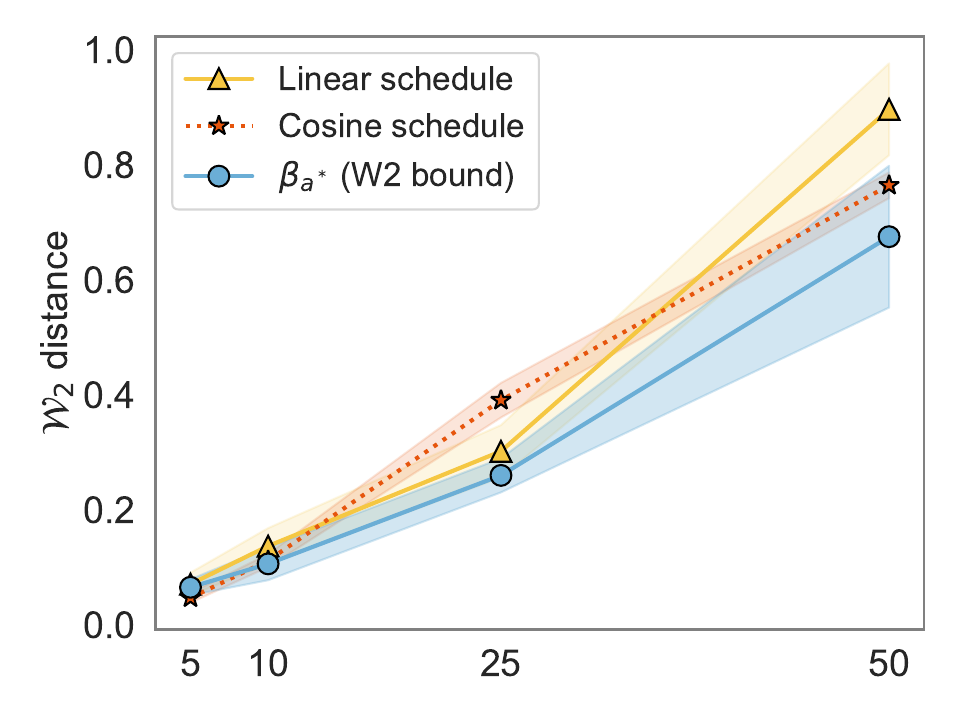} \\
        {(a) Isotropic setting } & 
        {(b) Rescaled heterosc.\ setting} & 
        {(c) Rescaled correlated setting} 
    \end{tabular}
    \caption{\label{fig:gaussian_xp_dim_w2} 
    Comparison of the empirical $\mathcal{W}_2$ distance (mean value $\pm$ std over 10 runs) between $\pi_{\mathrm{data}}$ and the generative distribution $\widehat{\pi}_N^{(\beta, \theta)}$ across various dimensions. The distributions compared include SGMs with different noise schedules: $\beta_{a^\star}$ (blue solid), $\beta_0$ (yellow dashed), and $\beta_{\cos}$ (orange dotted).   
    }
\end{figure*}
\subsection{More general target distributions}

Beyond Gaussian distributions, numerical analysis in terms of KL divergence is not tractable as standard estimators of the KL terms do not scale well with dimension. On the contrary, there exist computationally-efficient estimators of Wasserstein distances, as for instance the sliced $\mathcal{W}_2$ estimate \citep{pot_library}. We use the latter to assess the relevancy of Theorem \ref{thm:wasserstein_bound} when the target distribution corresponds to a  
50-dimensional Funnel distribution defined as:
$\pi_{\rm data} (x) =  \varphi_{a^2} (x_1)\prod_{j=2}^{d} \varphi_{\exp(2bx_1)}(x_j)$,
with $a=1$ and $b=0.5$ (see Section \ref{sec:appendix_synthetic_data} for more details and additional experiments on a Gaussian mixture model). As previously, the samples are standardized and rescaled. 
In Figure \ref{fig:funnel_main}, 
empirical results demonstrate that the minimum of the upper bound closely aligns with that of the empirical sliced 2-Wasserstein distance between the simulated and training data. 
Moreover, implementing SGM with the optimal parameter $a$ yields consistent improvements of the data generation quality across different metrics w.r.t.\ to classical noise schedule competitors (linear or cosine). 
These experiments not only support the relevance of the theoretical upper bound beyond the assumptions required in Section \ref{sec:wasserstein}, but also the validity of theoretically-inspired data preprocessing for improving SGM training with arbitrary target distributions. 
\begin{wrapfigure}[14]{R}{0.5\textwidth} 
\vspace{-0.2cm}
\includegraphics[width=0.8\linewidth]{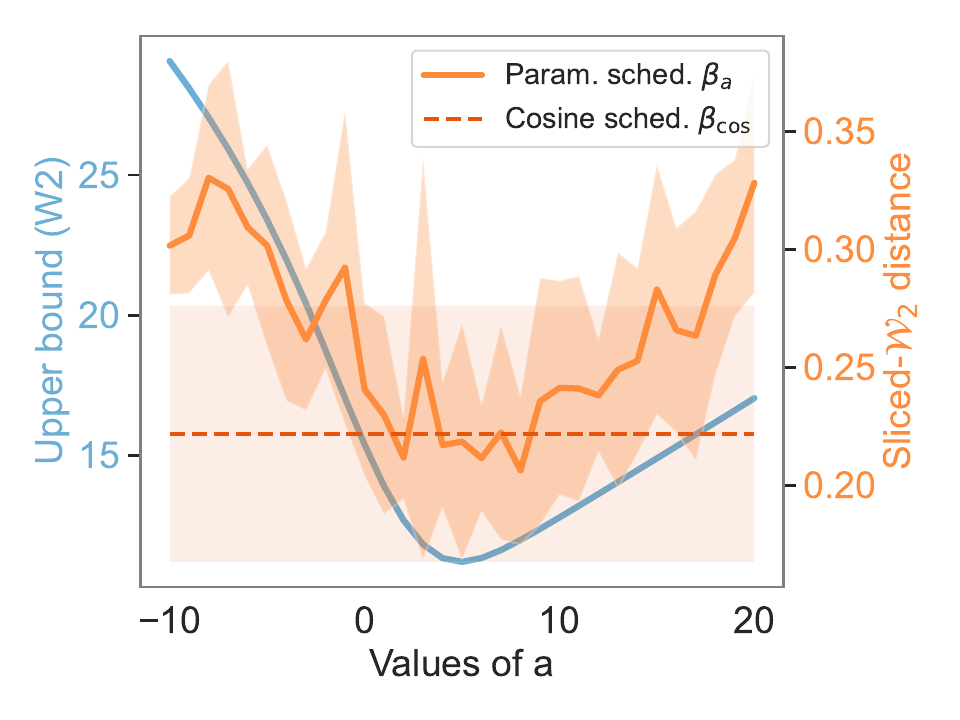}    \caption{Upper bound and sliced 2-Wasserstein distance on a Funnel dataset in dimension 50.}
    \label{fig:funnel_main}
\end{wrapfigure}

When dealing with high-dimensional real-world datasets, directly evaluating our theoretical upper bounds (Theorems \ref{th:main} and \ref{thm:wasserstein_bound}) becomes more challenging because relevant quantities (distances and constants) are either poorly estimated or unavailable. 
As a first step toward real data, we evaluate the impact of the noise schedule on the sampling quality of models pre-trained using CIFAR-10 dataset. In Figure \ref{fig:CIFAR_main}, we display the FID score with 50,000 generated samples using Euler-Maruyama discretization scheme for various noise schedules drawn from the parametric family in Equation \eqref{eq:noise_schedules}.  
Additional implementation details are available in Appendix \ref{sec:CIFAR_EXP}. Although the assumptions underpinning our results cannot be verified in this setting, the empirical performance trends mirror closely those observed in the simulated settings. 
\begin{wrapfigure}[12]{r}{0.5\textwidth} 
\vspace{-1.9cm}
\includegraphics[width=0.95\linewidth]{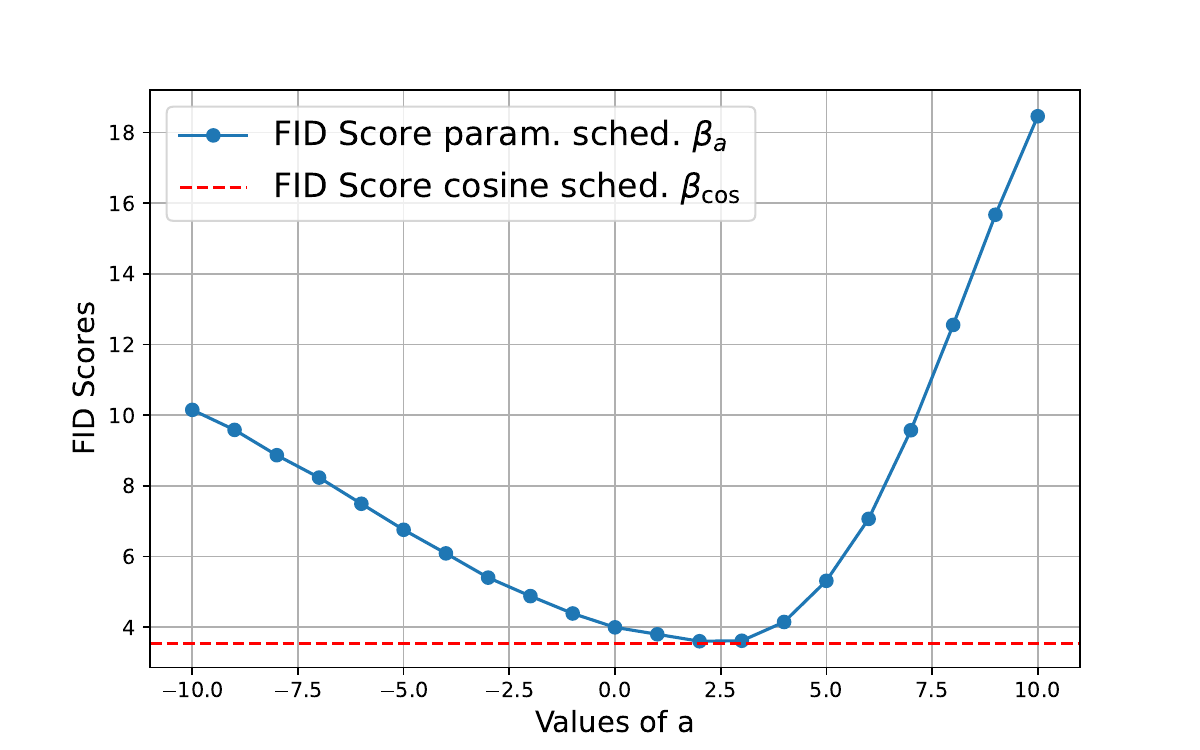}    \caption{FID Scores using 50,000 generated samples for the parametric and cosine schedules (CIFAR-10 dataset).}
    \label{fig:CIFAR_main}
\end{wrapfigure}
This consistency highlights that analyzing and optimizing noise schedules could be a promising direction for improving SGM-based generation in more complex scenarios.

\section{Discussion}
\label{sec:discussion}
In this paper, we propose a unified framework to analyze the impact of the noise schedule for time-inhomogeneous
SGMs,  providing upper bounds in KL and Wasserstein metrics. The KL upper bound follows the steps of recent works using the mildest assumptions used in the SGM literature. We also provide an improved upper bound in the Gaussian setting with numerical experiments highlighting the impact of the backward contraction of the forward noise process. Following \cite{bruno2023diffusion,gao2023wasserstein}, under additional assumptions on the Lipschitz and strong log-concavity properties of the score function, we establish upper bounds for the  Wasserstein distance.  This bound highlights the role of the noise schedule and provides a detailed analysis based on the modified score function. Our results are supported by numerical experiments in simple settings to highlight the several terms of the upper bounds and the role of the noise schedule. There are many perspectives to this work.  Studying multi-dimensional noise schedules is of particular interest. Indeed, they could be useful to understand how to deal with target distributions with complex covariance structures, and thereby an alternative solution to data normalization issues. Establishing upper bounds for Wasserstein distances under milder assumptions remains an exciting open problem, which would shed light on the performances and limitations of score-based generative models. A specific perspective would be to adapt our result using early-stopping to avoid the explosion of error terms in the neighborhood of 0 and to provide other assumptions to control the corresponding error close to 0.

\subsubsection*{Acknowledgements}

We would like to thank Gabriel  Victorino Cardoso for his valuable insights and thoughtful help on the numerical experiments involving real-world datasets.

Antonio Ocello was funded by the European Union (ERC-2022-SYG-OCEAN-101071601). Views and opinions expressed are however those of the author only and do not necessarily reflect those of the European Union or the European Research Council Executive Agency. Neither the European Union nor the granting authority can be held responsible for them.

\bibliography{24-noise_diff}

\begin{thebibliography}{42}
\providecommand{\natexlab}[1]{#1}
\providecommand{\url}[1]{\texttt{#1}}
\expandafter\ifx\csname urlstyle\endcsname\relax
  \providecommand{\doi}[1]{doi: #1}\else
  \providecommand{\doi}{doi: \begingroup \urlstyle{rm}\Url}\fi

\bibitem[Alfonsi et~al.(2015)Alfonsi, Jourdain, and
  Kohatsu-Higa]{alfonsi2015optimal}
Aur{\'e}lien Alfonsi, Benjamin Jourdain, and Arturo Kohatsu-Higa.
\newblock Optimal transport bounds between the time-marginals of a
  multidimensional diffusion and its euler scheme.
\newblock \emph{Electronic Journal of Probability}, 2015.

\bibitem[Bakry et~al.(2014)Bakry, Gentil, Ledoux, et~al.]{bakry2014analysis}
Dominique Bakry, Ivan Gentil, Michel Ledoux, et~al.
\newblock \emph{Analysis and geometry of Markov diffusion operators}, volume
  103.
\newblock Springer, 2014.

\bibitem[Baldi(2017)]{baldi}
Paolo Baldi.
\newblock \emph{Stochastic Calculus}.
\newblock Springer International Publishing AG, 1 edition, 2017.
\newblock ISBN 978-3319622255.

\bibitem[Benton et~al.(2024)Benton, De~Bortoli, Doucet, and
  Deligiannidis]{benton2024nearly}
Joe Benton, Valentin De~Bortoli, Arnaud Doucet, and George Deligiannidis.
\newblock Nearly d-linear convergence bounds for diffusion models via
  stochastic localization.
\newblock In \emph{The Twelfth International Conference on Learning
  Representations}, 2024.

\bibitem[Block et~al.(2020)Block, Mroueh, and Rakhlin]{block2020generative}
Adam Block, Youssef Mroueh, and Alexander Rakhlin.
\newblock Generative modeling with denoising auto-encoders and langevin
  sampling.
\newblock \emph{arXiv preprint arXiv:2002.00107}, 2020.

\bibitem[Bruno et~al.(2023)Bruno, Zhang, Lim, Akyildiz, and
  Sabanis]{bruno2023diffusion}
Stefano Bruno, Ying Zhang, Dong-Young Lim, {\"O}mer~Deniz Akyildiz, and
  Sotirios Sabanis.
\newblock On diffusion-based generative models and their error bounds: The
  log-concave case with full convergence estimates.
\newblock \emph{arXiv preprint arXiv:2311.13584}, 2023.

\bibitem[Chafai(2004)]{chafai_entropies_2004}
Djalil Chafai.
\newblock Entropies, convexity, and functional inequalities.
\newblock \emph{Kyoto Journal of Mathematics}, 44\penalty0 (2), 2004.
\newblock ISSN 2156-2261.
\newblock \doi{10.1215/kjm/1250283556}.

\bibitem[Chen et~al.(2023{\natexlab{a}})Chen, Lee, and Lu]{chen2023improved}
Hongrui Chen, Holden Lee, and Jianfeng Lu.
\newblock Improved analysis of score-based generative modeling: User-friendly
  bounds under minimal smoothness assumptions.
\newblock In \emph{International Conference on Machine Learning}, pages
  4735--4763. PMLR, 2023{\natexlab{a}}.

\bibitem[Chen et~al.(2023{\natexlab{b}})Chen, Chewi, Li, Li, Salim, and
  Zhang]{chen2023sampling}
Sitan Chen, Sinho Chewi, Jerry Li, Yuanzhi Li, Adil Salim, and Anru~R. Zhang.
\newblock Sampling is as easy as learning the score: theory for diffusion
  models with minimal data assumptions, 2023{\natexlab{b}}.

\bibitem[Chen(2023)]{chen2023importance}
Ting Chen.
\newblock On the importance of noise scheduling for diffusion models.
\newblock \emph{arXiv preprint arXiv:2301.10972}, 2023.

\bibitem[Collet and Malrieu(2008)]{collet_logarithmic_2008}
Jean-François Collet and Florent Malrieu.
\newblock Logarithmic sobolev inequalities for inhomogeneous markov semigroups.
\newblock \emph{European Series in Applied and Industrial Mathematics
  ({ESAIM}): Probability and Statistics}, 12:\penalty0 492--504, 2008.
\newblock ISSN 1292-8100.
\newblock \doi{10.1051/ps:2007042}.

\bibitem[Conforti et~al.(2023)Conforti, Durmus, and Silveri]{conforti2023}
Giovanni Conforti, Alain Durmus, and Marta~Gentiloni Silveri.
\newblock Score diffusion models without early stopping: finite fisher
  information is all you need, 2023.

\bibitem[De~Bortoli(2022)]{debortoli2022convergence}
Valentin De~Bortoli.
\newblock Convergence of denoising diffusion models under the manifold
  hypothesis.
\newblock \emph{Transactions on Machine Learning Research}, 2022.

\bibitem[De~Bortoli et~al.(2021)De~Bortoli, Thornton, Heng, and
  Doucet]{debortoli2021}
Valentin De~Bortoli, James Thornton, Jeremy Heng, and Arnaud Doucet.
\newblock Diffusion schr{\"o}dinger bridge with applications to score-based
  generative modeling.
\newblock \emph{Advances in Neural Information Processing Systems},
  34:\penalty0 17695--17709, 2021.

\bibitem[Del~Moral et~al.(2003)Del~Moral, Ledoux, and
  Miclo]{del_moral_contraction_2003}
P.~Del~Moral, M.~Ledoux, and L.~Miclo.
\newblock On contraction properties of markov kernels.
\newblock \emph{Probability Theory and Related Fields}, 126\penalty0
  (3):\penalty0 395--420, 2003.
\newblock ISSN 0178-8051.
\newblock \doi{10.1007/s00440-003-0270-6}.

\bibitem[Durmus and Moulines(2015)]{durmus2015quantitative}
Alain Durmus and {\'E}ric Moulines.
\newblock Quantitative bounds of convergence for geometrically ergodic markov
  chain in the wasserstein distance with application to the metropolis adjusted
  langevin algorithm.
\newblock \emph{Statistics and Computing}, 25:\penalty0 5--19, 2015.

\bibitem[Flamary et~al.(2021)Flamary, Courty, Gramfort, Alaya, Boisbunon,
  Chambon, Chapel, Corenflos, Fatras, Fournier, Gautheron, Gayraud, Janati,
  Rakotomamonjy, Redko, Rolet, Schutz, Seguy, Sutherland, Tavenard, Tong, and
  Vayer]{pot_library}
R{\'e}mi Flamary, Nicolas Courty, Alexandre Gramfort, Mokhtar~Z. Alaya,
  Aur{\'e}lie Boisbunon, Stanislas Chambon, Laetitia Chapel, Adrien Corenflos,
  Kilian Fatras, Nemo Fournier, L{\'e}o Gautheron, Nathalie~T.H. Gayraud,
  Hicham Janati, Alain Rakotomamonjy, Ievgen Redko, Antoine Rolet, Antony
  Schutz, Vivien Seguy, Danica~J. Sutherland, Romain Tavenard, Alexander Tong,
  and Titouan Vayer.
\newblock Pot: Python optimal transport.
\newblock \emph{Journal of Machine Learning Research}, 22\penalty0
  (78):\penalty0 1--8, 2021.

\bibitem[Franzese et~al.(2023)Franzese, Rossi, Yang, Finamore, Rossi,
  Filippone, and Michiardi]{franzese2023diffusion}
G.~Franzese, S.~Rossi, L.~Yang, A.~Finamore, D.~Rossi, M.~Filippone, and
  P.~Michiardi.
\newblock How much is enough? a study on diffusion times in score-based
  generative models.
\newblock \emph{Entropy}, 25:\penalty0 633, 2023.
\newblock \doi{10.3390/e25040633}.

\bibitem[Gao et~al.(2023)Gao, Nguyen, and Zhu]{gao2023wasserstein}
Xuefeng Gao, Hoang~M. Nguyen, and Lingjiong Zhu.
\newblock Wasserstein convergence guarantees for a general class of score-based
  generative models, 2023.

\bibitem[Gong et~al.(2023)Gong, Li, Feng, Wu, and Kong]{gong2022diffuseq}
Shansan Gong, Mukai Li, Jiangtao Feng, Zhiyong Wu, and LingPeng Kong.
\newblock Diffuseq: Sequence to sequence text generation with diffusion models.
\newblock In \emph{Proceedings of International Conference on Learning
  Representations}, 2023.

\bibitem[Guo et~al.(2023)Guo, Liu, Yu, and Luo]{guo2023rethinking}
Qiushan Guo, Sifei Liu, Yizhou Yu, and Ping Luo.
\newblock Rethinking the noise schedule of diffusion-based generative models.
\newblock \emph{visible on Open Review}, 2023.

\bibitem[Ho et~al.(2020)Ho, Jain, and Abbeel]{ho2020denoising}
Jonathan Ho, Ajay Jain, and Pieter Abbeel.
\newblock Denoising diffusion probabilistic models.
\newblock In \emph{Advances in Neural Information Processing Systems}, 2020.

\bibitem[Hyv{\"a}rinen and Dayan(2005)]{hyvarinen2005estimation}
Aapo Hyv{\"a}rinen and Peter Dayan.
\newblock Estimation of non-normalized statistical models by score matching.
\newblock \emph{Journal of Machine Learning Research}, 6\penalty0 (4), 2005.

\bibitem[Karatzas and Shreve(2012)]{karatzas2012brownian}
Ioannis Karatzas and Steven Shreve.
\newblock \emph{Brownian motion and stochastic calculus}, volume 113.
\newblock Springer Science \& Business Media, 2012.

\bibitem[Karras et~al.(2022)Karras, Aittala, Aila, and Laine]{karras2022edm}
Tero Karras, Miika Aittala, Timo Aila, and Samuli Laine.
\newblock Elucidating the design space of diffusion-based generative models.
\newblock In \emph{Advances in Neural Information Processing Systems},
  volume~35, pages 8595--8607, 2022.

\bibitem[Lee et~al.(2022)Lee, Lu, and Tan]{lee2022convergence}
Holden Lee, Jianfeng Lu, and Yixin Tan.
\newblock Convergence for score-based generative modeling with polynomial
  complexity.
\newblock \emph{Advances in Neural Information Processing Systems},
  35:\penalty0 22870--22882, 2022.

\bibitem[Lee et~al.(2023)Lee, Lu, and Tan]{lee2023convergence}
Holden Lee, Jianfeng Lu, and Yixin Tan.
\newblock Convergence of score-based generative modeling for general data
  distributions.
\newblock In \emph{International Conference on Algorithmic Learning Theory},
  pages 946--985. PMLR, 2023.

\bibitem[Li et~al.(2022)Li, Yang, Chang, Chen, Feng, Xu, Li, and
  Chen]{li2022srdiff}
Haoying Li, Yifan Yang, Meng Chang, Shiqi Chen, Huajun Feng, Zhihai Xu, Qi~Li,
  and Yueting Chen.
\newblock Srdiff: Single image super-resolution with diffusion probabilistic
  models.
\newblock \emph{Neurocomputing}, 479:\penalty0 47--59, 2022.

\bibitem[Lugmayr et~al.(2022)Lugmayr, Danelljan, Romero, Yu, Timofte, and
  Van~Gool]{lugmayr2022repaint}
Andreas Lugmayr, Martin Danelljan, Andres Romero, Fisher Yu, Radu Timofte, and
  Luc Van~Gool.
\newblock Repaint: Inpainting using denoising diffusion probabilistic models.
\newblock In \emph{Proceedings of the IEEE/CVF Conference on Computer Vision
  and Pattern Recognition}, pages 11461--11471, 2022.

\bibitem[Nichol and Dhariwal(2021)]{nichol2021improved}
Alexander~Quinn Nichol and Prafulla Dhariwal.
\newblock Improved denoising diffusion probabilistic models.
\newblock In Marina Meila and Tong Zhang, editors, \emph{Proceedings of the
  38th International Conference on Machine Learning}, volume 139 of
  \emph{Proceedings of Machine Learning Research}, pages 8162--8171. PMLR,
  18--24 Jul 2021.

\bibitem[Ramesh et~al.(2022)Ramesh, Dhariwal, Nichol, Chu, and
  Chen]{ramesh2022hierarchical}
Aditya Ramesh, Prafulla Dhariwal, Alex Nichol, Casey Chu, and Mark Chen.
\newblock Hierarchical text-conditional image generation with clip latents.
\newblock \emph{arXiv preprint arXiv:2204.06125}, 1\penalty0 (2):\penalty0 3,
  2022.

\bibitem[Saremi et~al.(2023)Saremi, Park, and Bach]{saremi2023chain}
Saeed Saremi, Ji~Won Park, and Francis Bach.
\newblock Chain of log-concave markov chains.
\newblock \emph{arXiv preprint arXiv:2305.19473}, 2023.

\bibitem[Saumard and Wellner(2014)]{saumard2014log}
Adrien Saumard and Jon~A. Wellner.
\newblock {Log-concavity and strong log-concavity: A review}.
\newblock \emph{Statistics Surveys}, 8\penalty0 (none):\penalty0 45 -- 114,
  2014.
\newblock \doi{10.1214/14-SS107}.

\bibitem[Sohl-Dickstein et~al.(2015)Sohl-Dickstein, Weiss, Maheswaranathan, and
  Ganguli]{dickstein2015}
Jascha Sohl-Dickstein, Eric Weiss, Niru Maheswaranathan, and Surya Ganguli.
\newblock Deep unsupervised learning using nonequilibrium thermodynamics.
\newblock In Francis Bach and David Blei, editors, \emph{Proceedings of the
  32nd International Conference on Machine Learning}, volume~37 of
  \emph{Proceedings of Machine Learning Research}, pages 2256--2265, Lille,
  France, 07--09 Jul 2015. PMLR.

\bibitem[Song and Ermon(2019)]{song2019generative}
Yang Song and Stefano Ermon.
\newblock Generative modeling by estimating gradients of the data distribution.
\newblock In \emph{Advances in Neural Information Processing Systems}, 2019.

\bibitem[Song et~al.(2021)Song, Sohl-Dickstein, Kingma, Kumar, Ermon, and
  Poole]{song2021score}
Yang Song, Jascha Sohl-Dickstein, Diederik~P Kingma, Abhishek Kumar, Stefano
  Ermon, and Ben Poole.
\newblock Score-based generative modeling through stochastic differential
  equations.
\newblock \emph{International Conference on Learning Representations (ICLR)},
  2021.

\bibitem[Talagrand(1996)]{talagrand1996transportation}
Michel Talagrand.
\newblock Transportation cost for gaussian and other product measures.
\newblock \emph{Geometric \& Functional Analysis GAFA}, 6\penalty0
  (3):\penalty0 587--600, 1996.

\bibitem[Thin et~al.(2021)Thin, Idrissi, Corff, Ollion, Moulines, Doucet,
  Durmus, and Robert]{thin2021neo}
Achille Thin, Yazid Janati~El Idrissi, Sylvain~Le Corff, Charles Ollion, Eric
  Moulines, Arnaud Doucet, Alain Durmus, and Christian~P Robert.
\newblock {NEO}: Non equilibrium sampling on the orbits of a deterministic
  transform.
\newblock In A.~Beygelzimer, Y.~Dauphin, P.~Liang, and J.~Wortman Vaughan,
  editors, \emph{Advances in Neural Information Processing Systems}, 2021.
\newblock URL
  \url{https://proceedings.neurips.cc/paper/2021/hash/76tTYokjtG-abstract.html}.

\bibitem[Villani(2021)]{villani2021topics}
C{\'e}dric Villani.
\newblock \emph{Topics in optimal transportation}, volume~58.
\newblock American Mathematical Soc., 2021.

\bibitem[Vincent(2011)]{Vincent}
Pascal Vincent.
\newblock A connection between score matching and denoising autoencoders.
\newblock \emph{Neural Computation}, 23\penalty0 (7):\penalty0 1661--1674,
  2011.
\newblock \doi{10.1162/NECO_a_00142}.

\bibitem[Wang et~al.(2009)Wang, Kulkarni, and Verdu]{KLviaKNN}
Qing Wang, Sanjeev~R. Kulkarni, and Sergio Verdu.
\newblock Divergence estimation for multidimensional densities via
  $k$-nearest-neighbor distances.
\newblock \emph{IEEE Transactions on Information Theory}, 55\penalty0
  (5):\penalty0 2392--2405, 2009.
\newblock \doi{10.1109/TIT.2009.2016060}.

\bibitem[Yang et~al.(2023)Yang, Zhang, Song, Hong, Xu, Zhao, Zhang, Cui, and
  Yang]{yang2023diffusion}
Ling Yang, Zhilong Zhang, Yang Song, Shenda Hong, Runsheng Xu, Yue Zhao, Wentao
  Zhang, Bin Cui, and Ming-Hsuan Yang.
\newblock Diffusion models: A comprehensive survey of methods and applications.
\newblock \emph{ACM Computing Surveys}, 56\penalty0 (4):\penalty0 1--39, 2023.

\end{thebibliography}
\bibliographystyle{plainnat}

\newpage
\appendix

\onecolumn

\section{Notations and assumptions. } 
\label{sec:app:notations}
Consider the following notations, used throughout the appendices. For all $d\geq 1$, $\mu\in\rset^d$ and definite positive matrices $\Sigma\in\rset^{d\times d}$, let $\gausspdf_{\mu,\Sigma}$ be the probability density function of a Gaussian random variable with mean $\mu$ and variance $\Sigma$. We also use the notation  $\gausspdf_{\sigma^2} = \gausspdf_{0,\sigma^2 \mathrm{I}_d}$. 
When the context is clear, we may indifferently use the measure and the associated density w.r.t.\ the reference measure. 
For all twice-differentiable real-valued function $f$, let $\Delta f$ be the Laplacian of  $f$. For all matrix $A\in\rset^{m\times n}$, $\|A\|_{\mathrm{Fr}}$ is the  Frobenius norm of $A$, i.e., $\|A\|_{\mathrm{Fr}} = (\sum_{i=1}^m\sum_{j=1}^n|A_{i,j}|^2)^{1/2}$. For all time-dependent real-valued functions $h:t\mapsto h_t$ or $f:t\mapsto f(t)$, we write $\bar h_t = h_{T-t}$ and $\bar f(t) = f(T-t)$ for all $t\in[0,T]$.

Let $\pi_0$ be a probability density function with respect to the Lebesgue measure on $\rset^d$ and $\alpha: \rset\to \rset$ and $g: \rset\to \rset$ be two continuous and increasing functions. Consider the general forward process
\begin{equation} \label{eq:forward:general}
    \rmd  \overrightarrow{X}_t = -\alpha (t) \overrightarrow{X}_t \rmd t + g(t) \rmd B_t, \quad \ora{X}_0 \sim \pi_0 \eqsp,
\end{equation} 
 and introduce $\tilde p_{t}:x \mapsto  p_t(x)/\gausspdf_{\sigma^2}(x)$, where $p_t$ is the probability density function of $\overrightarrow{X}_t$. The backward process associated with \eqref{eq:forward:general} is referred to as $( \ola{X}_t )_{t \in [0,T]}$ and given by
\begin{equation} \label{eq:backward:general}
    \rmd \overleftarrow{X}_t = \left\{ \left( \bar \alpha(t) - \frac{\bar g^2(t)}{\sigma^2} \right)  \overleftarrow{X}_t + \bar g^2(t) \nabla \log \tilde p_{T-t}\left( \ola{X}_t \right) \right\} \rmd t +  \bar g(t) \rmd \bar B_t \quad \ola{X}_0 \sim p_T \eqsp,
\end{equation}
with $\bar B$ a standard Brownian motion in $\R^d$.
Moreover, consider
\begin{align}\label{eq:cond_variance}
\sigma_t^2 \eqdef \exp \left( {-2 \int_0^t \alpha (s) \rmd s} \right)\int_0^t g^2(s)\exp\left({ 2 \int_0^s \alpha (u) \rmd u}\right) \rmd s. 
\end{align}
The approximate EI discretization of \eqref{eq:backward:general} considered in this paper is, for $t_k \leq t \leq t_{k+1}$, $0\leq k \leq N-1$,
\begin{equation*} 
    \rmd \overleftarrow{X}^\theta_t =  \left\{\bar \alpha(t)  \overleftarrow{X}^\theta_t + \bar g^2(t) s_{\theta} (T-t_{k}, \ola{X}_{t_k}^\theta) \right\} \rmd t +  \bar g(t) \rmd \bar B_t\eqsp.
\end{equation*}
Sampling from this backward SDE is  possible recursively for $k \in \{ 0, \ldots ,N-1\}$, with $(Z_k)_{1 \leq k \leq N} \overset{\text{i.i.d}}{\sim} \mathcal{N}(0, I_d)$. For $k \in \{ 0, \ldots ,N-1\}$, writing $\tau_k = T-t_k$,
\begin{multline*}
    \ola{X}_{t_{k+1}}^\theta  = \rme^{- \int_{\tau_k}^{\tau_{k+1}} \alpha (s) \rmd s}  \ola{X}_{t_k}^\theta
    + s_{\theta} (\tau_{k}, \ola{X}_{t_k}^\theta) \rme^{-\int_{\tau_{k}}^{\tau_{k+1}}\alpha(s)\rmd s} \int_{\tau_{k}}^{\tau_{k+1}}  g^2(t) \rme^{\int_{\tau_k}^t \alpha(v)\rmd v} \rmd t \\
     + \left(\rme^{-2\int_{\tau_k}^{\tau_{k+1}} \alpha (s) \rmd s} \int_{\tau_{k+1}}^{\tau_k} \rme^{2 \int_{\tau_k}^t \alpha(s) \rmd s} g^2(t) \rmd t\right)^{1/2}Z_{k+1} \eqsp.
\end{multline*}
We denote by $ \mathbb{Q}_T \in \mathcal{P}(C([0,T], \mathbb{R}^d))$ the path measure associated with the backward diffusion and by $(Q_t)_{0\leq t \leq T}$ its Markov semi-group. We also write $\ola{X}^\infty_{T}\sim \pi_\infty Q_T$ and, for each time step $t_k$ for $0\leq k \leq N$,  $\ola{X}_{t_{k}}^{\infty} \sim \pi_{\infty} Q_{t_k}$. 
For each time step $t_k$ for $0\leq k \leq N$, the kernel associated with the backward discretization is denoted by  $Q_{t_k}^{N, \theta}$, so that we have $ \bar X_{t_k}^{\theta} \sim \pi_{\infty} Q_{t_k}^{N, \theta}$. 

In Appendix~\ref{sec:proof:wasserstein}, these notations are used for the specific case where $\alpha: t\mapsto  \beta(t)/(2\sigma^2)$ and $g :t\mapsto \beta(t)^{1/2}$ and the associated backward discretization is given in \eqref{eq:backward_EI-theta}. 

\section{Proofs of Section~\ref{sec:main}}
\subsection{Proof of Theorem~\ref{th:main}}
\label{sec:proofs}

{
    We are interested in the relative entropy of the training data distribution $\pi_{\rm data}$ with respect to the generated data distribution $ \pihat $. Leveraging the time-reverse property we have:

\begin{align*}
    \kl \left( \pi_{\rm data} \middle\| \pihat   \right)  = \kl\left( p_T Q_T \middle\| \pihat   \right)\eqsp.
\end{align*}
By the data processing inequality, 
\begin{align*}
     \kl \left( p_T Q_T \middle\| \pihat  \right) \leq \kl \left( p_T \mathbb{Q}_T \middle\| \pi_{\infty} \mathbb{Q}_T^{N,\theta}\right)\eqsp.
\end{align*}

where $\mathbb{Q}_T$ and $\mathbb{Q}_T^{N,\theta}$ denote the path measures of, respectively, the backward process and the SGM generation. Writing the backward time $\tau_t = T-t$ and its discretized version $\tau_k = T-t_k$, with $0=t_0<t_1<\ldots<t_N=T$, we have (by Lemma \ref{tvboundgirsanovproof}) that
\begin{multline*}
 \kl \left( \pi_{\rm data} \| \pihat   \right) \leq  \kl \left( p_T \| \pi_{\infty} \right) + \frac{1}{2} \int_0^T \frac{1}{\bar \sched(t)} \mathbb{E}\Bigg[ \Bigg\|  \frac{-  \bar \sched (t)}{2 \sigma^2} \overleftarrow{X}_t + \bar \sched (t) \nabla \log \tilde p_{\tau_t}\left( \ola{X}_t \right)    \\
- \left( - \frac{ \bar \sched (t)}{2 \sigma^2} \overleftarrow{X}_t + \bar \sched (t) \tilde s_{\param} \left( \tau_k, \ola{X}_{t_k}    \right) \right) \Bigg\|^2 \Bigg] \rmd t \eqsp.
\end{multline*}
From there, the KL divergence can be split into the theoretical mixing time of the forward OU process and the approximation error for the score function made by the neural network, as follows:
\begin{align*}
 &\kl \left( \pi_{\rm data} \| \pihat \right) \leq \kl \left( p_T \| \pi_{\infty} \right) +  \frac{1}{2} \int_0^T \frac{1}{\bar \sched(t)} \mathbb{E}\Bigg[ \Bigg\| \bar \sched (t) \left(  \tilde s\left(\tau_t,\ola X_t\right) - \tilde s_{\param} ( \tau_k, \ola{X}_{t_k}    ) \right)  \Bigg\|^2 \Bigg] \rmd t\eqsp.
\end{align*}
By using the regular discretization of the interval $[0,T]$, one can disentangle the last term as follows:

\begin{align*}
 \kl \left( \pi_{\rm data} \middle\| \pihat \right)& \leq\kl \left( p_T \| \pi_{\infty} \right) + \frac{1}{2} \sum_{k = 0}^{N-1} \int_{t_k}^{t_{k+1}} \bar \sched (t) \mathbb{E} \left[  \left\|   \tilde s\left(\tau_t,\ola X_t\right)- \tilde s_{\param}\left(\tau_k, \ola X_{t_k} \right) \right\|^2 \right] \rmd t \\
&\leq E_1(\sched) + E_2(\param,\sched) + E_3(\sched)\eqsp,
\end{align*}
where
\begin{align}
    E_1(\sched) &= \kl \left( p_T \| \gausspdf_{\sigma^2} \right)\eqsp,\label{eq:err:mix}\\
    E_2(\param,\sched) &= \sum_{k = 0}^{N-1} \int_{t_k}^{t_{k+1}} \bar \sched (t)\mathbb{E} \left[  \left\|  \tilde s\left( \tau_k, \ola X_{t_k} \right)  - \tilde s_{\param}\left(\tau_k, \ola X_{t_k} \right) \right\|^2 \right] \rmd t\eqsp,\label{eq:err:approx}\\
    E_3(\sched) &= \sum_{k = 0}^{N-1} \int_{t_k}^{t_{k+1}} \bar \sched (t)\mathbb{E} \left[  \left\| \tilde s\left(\tau_t,\ola X_t\right) -\tilde s\left( \tau_k, \ola X_{t_k} \right)   \right\|^2 \right] \rmd t\eqsp. \label{eq:err:discr}
\end{align}
Finishing the proof of Theorem~\ref{th:main} amounts to obtaining upper bounds for $E_1(\sched)$, $E_2(\param,\sched)$ and $E_3(\sched)$. This is done in Lemmas~\ref{lem:mix}, \ref{lem:approx} and \ref{lem:fish}, so that $E_1(\sched) \leq \mathcal{E}_1(\sched)$, $E_2(\param,\sched) \leq \mathcal{E}_2(\param,\sched)$ and $E_3(\sched) \leq \mathcal{E}_3(\sched)$.
}

\begin{lemma} 
\label{lem:mix}For any noise schedule $\beta$,
    \begin{align*}
    E_1(\sched) = 
    \kl\left(p_T \| \pi_{\infty} \right) \leq  
    \kl\left(\pi_{\rm data} \|  \pi_{\infty} \right) 
    \exp\left(- \frac{1}{\sigma^2} \int_0^T \sched(s) \rmd s\right)  \eqsp.
    \end{align*}
\end{lemma}
\begin{proof}

The proof follows the same lines as \citet[][Lemma 1]{franzese2023diffusion}. The Fokker-Planck equation associated with \eqref{eq:forward-SDE:beta} is
\begin{align*}
    \partial_{t} p_{t}(x) = \frac{\sched(t)}{2\sigma^2} \text{div} \left( x p_{t} (x) \right) + \frac{\sched(t)}{2} \Delta p_{t} (x) =
    \frac{\sched(t)}{2}\text{div} \left(\frac{1}{\sigma^2}x p_{t} (x) + \nabla p_{t} (x)\right)
    \eqsp,
\end{align*}
for $t\in[0,T], x\in\R^d$. Combing this with the derivation under the integral theorem, we get
\begin{align*}
    \frac{\partial}{ \partial t} \kl \left( p_t \|\gausspdf_{\sigma^2}  \right) &= \frac{\partial}{\partial t} \int_{\mathbb{R}^d}  \log \frac{p_t(x)}{\gausspdf_{\sigma^2}(x)} p_t(x) \rmd x \\
    &= \int_{\mathbb{R}^d} \frac{\partial}{\partial t} p_t(x) \log \frac{p_t(x)}{\gausspdf_{\sigma^2}(x)}\rmd x + \int_{\mathbb{R}^d}\frac{p_t(x) \partial_t p_t(x)}{p_t(x)} \rmd x 
    \\
    &= \int_{\mathbb{R}^d} \frac{\partial}{\partial t} p_t(x) \log \frac{p_t(x)}{\gausspdf_{\sigma^2}(x)} \rmd x + \int_{\mathbb{R}^d} \frac{\partial}{\partial t} p_t(x) \rmd x
    \\
    &= \int_{\mathbb{R}^d} \frac{\sched(t)}{2} \text{div} \left( \frac{x}{\sigma^2} p_t(x) + \nabla p_t(x) \right) \log \frac{p_t(x)}{\gausspdf_{\sigma^2}(x)} \rmd x
    \\
    &= \frac{\sched(t)}{2} \int_{\mathbb{R}^d}  \text{div} \left(  - \nabla \log \gausspdf_{\sigma^2}(x) \; p_t(x)  +  \nabla p_t(x) \right) \log \frac{p_t(x)}{\gausspdf_{\sigma^2}(x)} \rmd x 
    \\
    &=  - \frac{\sched(t)}{2} \int_{\mathbb{R}^d} \left(- \nabla \log \gausspdf_{\sigma^2}(x) \; p_t(x)  +  \nabla p_t(x) \right)^\top \nabla \log \frac{p_t(x)}{\gausspdf_{\sigma^2}(x)} \rmd x 
    \\
    &= - \frac{\sched(t)}{2} \int_{\mathbb{R}^d} p_t(x) \left( - \nabla \log \gausspdf_{\sigma^2}(x)   +  \nabla \log p_t(x)\right)^\top \nabla \log \frac{p_t(x)}{\gausspdf_{\sigma^2}(x)} \rmd x\\
    &= - \frac{\sched(t)}{2} \int_{\mathbb{R}^d} p_t(x) \left\| \nabla \log \frac{p_t(x)}{\gausspdf_{\sigma^2}(x)} \right\|^2 \rmd x\eqsp.
\end{align*}
Using the Stam-Gross logarithmic Sobolev inequality given in  Proposition~\ref{stam-gross}, we get
\begin{align*}
    \frac{\partial}{ \partial t} \kl \left( p_t \|\gausspdf_{\sigma^2}  \right)
    & \leq - \frac{ \sched(t)}{\sigma^2} \kl \left( p_t \|\gausspdf_{\sigma^2}  \right)\eqsp.
\end{align*}
Applying Grönwall's inequality, we obtain
\begin{align*}
    \kl \left( p_T \|\gausspdf_{\sigma^2} \right) \leq \kl \left( p_0 \|\gausspdf_{\sigma^2} \right) \exp\left\{- \frac{1}{\sigma^2} \int_0^T \sched(s) \rmd s\right\}\eqsp, 
\end{align*}
which concludes the proof.
\end{proof}

\begin{lemma}
\label{lem:approx}
For all $\param$ and all $\beta$,
\begin{align*}
    E_2(\param,\sched) = \sum_{k = 1}^{N}  \mathbb{E} \left[  \left\|   \nabla \log \tilde p_{t_k}\left(\ora X_{t_k}\right) - \tilde s_{\param}\left(t_k, \ora X_{t_k}\right) \right\|^2 \right] \int_{t_k}^{t_{k+1}} \bar \sched (t) \rmd t\eqsp,
\end{align*}
    where $E_2(\param,\sched)$ is defined by \eqref{eq:err:approx}.
\end{lemma}

\begin{proof}
By definition of $E_2(\param,\sched)$,
\begin{align*}
  E_2(\param,\sched) & = \sum_{k = 0}^{N-1} \int_{t_k}^{t_{k+1}} \bar \sched (t) \mathbb{E} \left[  \left\|   \nabla \log \tilde p_{T-t_k}\left(\ola X_{t_k}\right) - \tilde s_{\param}\left(T - t_k, \ola X_{t_k} \right) \right\|^2 \right] \rmd t \\
  & = \sum_{k = 0}^{N-1}  \mathbb{E} \left[  \left\|   \nabla \log \tilde p_{T-t_k}\left(\ola X_{t_k}\right) - \tilde s_{\param}\left(T - t_k, \ola X_{t_k} \right) \right\|^2 \right] \int_{t_k}^{t_{k+1}} \bar \sched (t) \rmd t \\
  & ={\sum_{k = 0}^{N-1}} \mathbb{E} \left[  \left\|   \nabla \log \tilde p_{t_k}\left(\ora X_{t_k}\right) - \tilde s_{\param}\left(t_k, \ora X_{t_k}\right) \right\|^2 \right] \int_{t_k}^{t_{k+1}} \bar \sched (t) \rmd t\eqsp,
\end{align*}
where the last equality comes from the fact that the forward and backward processes have same marginals since $ \ora X_{T} \sim  p_T$.
\end{proof}

\begin{lemma}
\label{lem:fish}
Assume that H\ref{hyp:sched} holds. For all $T,\sigma>0$, $\param$ and all $\sched$,
    \begin{align*}
    E_3(\sched) \leq 2 h  \beta (T) \max \left\{ \frac{  h  \beta (T)  }{ 4 \sigma^2} ; 1  \right\} \mathcal{I}(\pi_{\rm data} | \pi_{\infty})   \eqsp,
    \end{align*}
    where $E_3(\sched)$ is defined by \eqref{eq:err:discr}.
\end{lemma}

\begin{proof}
By Lemma \ref{score differential form},  with $Y_t \eqdef \nabla \log \tilde p_{T-t} (\ola X_t)$,
\begin{align*}
\rmd Y_t &= \frac{ \bar \sched(t)}{2 \sigma^2} Y_t \rmd t + \sqrt{\bar \sched(t)} Z_t \rmd B_t\eqsp. 
\end{align*}
By applying  Itô's lemma to the function $x \mapsto \|x\|^2$, we obtain
\begin{align*}
\rmd \|Y_t\|^2 = \left(\frac{\bar \sched (t)}{\sigma^2} \|Y_t\|^2 + \bar \sched (t) \|Z_t\|_{\mathrm{Fr}}^2 \right) \rmd t + \sqrt{\bar \sched (t)} Y_t^\top Z_t \rmd B_t\eqsp.
\end{align*}

Fix $\delta>0$.
From \citet[Theorem~7.3, p.193]{baldi}, we have that $\left( \int_0^t g(s) Y_s^\top Z_s \rmd B_s\right)_{t \in [0,T-\delta]}$ is a square integrable martingale if 
\begin{align*}
\mathbb{E} \left[ \int_0^{T-\delta} g^2(s) \left\| Y_s^\top Z_s \right\|^2 \rmd s \right] < \infty\eqsp .
\end{align*}
From the Cauchy-Schwarz inequality, we get that 
\begin{align*}
    \mathbb{E} \left[ \left\| Y_s^\top Z_s \right\|^2_2 \right]
     \leq
    \mathbb{E} \left[ \left\| Y_s \right\|_2^2   \left\| Z_s \right\|_{\mathrm{Fr}}^2 \right] 
    \leq
    \mathbb{E} \left[ \left\| Y_s \right\|_2^4 \right]^{1/2}  \mathbb{E} \left[ \left\| Z_s \right\|_{\mathrm{Fr}}^4 \right]^{1/2}\eqsp .
\end{align*}
Applying Lemma \ref{score moment bound } and \ref{second derivative moment bound}, we get that both $\mathbb{E} [ \| Y_s \|_2^4 ]$ and $\mathbb{E} [ \| Z_s \|_2^4 ]$ are bounded by a quantity depending on $\sigma_{T-t}^{-8}$. As the term $\sigma_{T-t}^{-8}$ is uniformly bounded in $[0,T-\delta]$ and by Fubini's theorem, $\mathbb{E} [ \int_0^T g^2(s) \| Y_s^\top Z_s \|^2 \rmd s ]  =  \int_0^T g^2(s) \mathbb{E} [ \| Y_s^\top Z_s \|^2]   \rmd s < \infty$.  
Therefore, $( \int_0^t g(s) Y_s^\top Z_s \rmd B_s )_{t \in [0,T-\delta]}$ is a square integrable martingale. Therefore, we have
\begin{align*}
\mathbb{E} \left[ \|Y_t\|^2 \right] - \mathbb{E} \left[ \|Y_{t_k}\|^2 \right] = \mathbb{E} \left[\int_{t_k}^t \frac{\bar \sched (s)}{\sigma^2} \|Y_s\|^2  \rmd s+ \int_{t_k}^t \bar \sched (s) \|Z_s\|_{\mathrm{Fr}}^2  \rmd s \right]\eqsp,
\end{align*}
and
\begin{align}
    \mathbb{E} \left[ \left\|  Y_t -  Y_{t_k} \right\|^2 \right] &=  \mathbb{E} \left[ \left\|  \int_{t_k}^t \frac{ \bar \sched(s)}{2 \sigma^2} Y_s \rmd s + \int_{t_k}^t \sqrt{\bar \sched(s)} Z_s \rmd B_s \right\|^2 \right]
    \notag\\
    & \leq 2 \mathbb{E} \left[ \left\|  \int_{t_k}^{t} \frac{ \bar \sched(s)}{2 \sigma^2} Y_s \rmd s \right\|^2  \right] +2 \mathbb{E} \left[ \int_{t_k}^{t} \left\| \sqrt{\bar \sched(s)} Z_s \rmd B_s \right\|^2 \right]
    \notag\\
    & \leq 2 \mathbb{E} \left[ \left\|  \frac{1}{2 \sigma} \int_{t_k}^{t} \sqrt{\bar \sched(s)} \frac{ \sqrt{\bar \sched(s)}}{ \sigma} Y_s \rmd s \right\|^2  \right] +2 \mathbb{E} \left[ \int_{t_k}^{t} \left\| \sqrt{\bar \sched(s)} Z_s \rmd B_s \right\|^2 \right]
    \notag\\
    & \leq \frac{1}{2 \sigma^2} \int_{t_k}^{t_{k +1}} \bar \sched (s) \rmd s \mathbb{E} \left[ \int_{t_k}^{t_{k +1}} \frac{\bar \sched(s)}{\sigma^2}  \left\|  Y_s \right\|^2 ds  \right] + 2 \mathbb{E} \left[  \int_{t_k}^{t_{k +1}} \bar \sched (s) \left\| Z_s \right\|^2_{\mathrm{Fr}} \rmd s \right] 
    \notag\\
    & \leq 2 \max \left\{ \frac{\int_{t_k}^{t_{k +1}} \bar \sched (s) \rmd s}{ 4 \sigma^2} , 1 \right\} \left( \mathbb{E} \left[ \|Y_{t_{k+1}}\|^2 \right] - \mathbb{E} \left[ \|Y_{t_k}\|^2 \right] \right)\eqsp. \label{eq:lemmaA3:ineq-telescopic}
\end{align}
Without loss of generality, we have that $t_{N-1} = T-\delta$.
Then, the discretization error can be bounded as follows
\begin{align*}
    &\sum_{k = 0}^{N-1} \int_{t_k}^{t_{k+1}} \bar \sched (t) \mathbb{E} \left[  \left\|   \nabla \log \tilde p_{T-t}\left(\ola X_t\right) - \nabla \log \tilde p_{T-t_k}\left(\ola X_{t_k}\right) \right\|^2 \right] \rmd t
    \\
    &\hspace{.1cm}= 
    \sum_{k = 0}^{N-1} \int_{t_k}^{t_{k+1}} \bar \sched (t) \mathbb{E} \left[  \left\|   Y_t - Y_{t_k} \right\|^2 \right] \rmd t
    \\
    &\hspace{.1cm} \leq
    2 \sum_{k = 0}^{N-1} \int_{t_k}^{t_{k+1}} \bar \sched (t)  \max \left\{ \frac{\int_{t_k}^{t_{k +1}} \bar \sched (s) \rmd s}{ 4 \sigma^2} , 1 \right\} \left( \mathbb{E} \left[ \|Y_{t_{k+1}}\|^2 \right] - \mathbb{E} \left[ \|Y_{t_k}\|^2 \right] \right) \rmd t
    \\
    &\hspace{.1cm} \leq
    2 \sum_{k = 0}^{N-1}  \max \left\{ \frac{\int_{t_k}^{t_{k +1}} \bar \sched (s) \rmd s}{4 \sigma^2} , 1 \right\} \left( \mathbb{E} \left[ \|Y_{t_{k+1}}\|^2 \right] - \mathbb{E} \left[ \|Y_{t_k}\|^2 \right] \right) \int_{t_k}^{t_{k+1}} \bar \sched (t) \rmd t
    \\
    &\hspace{.1cm} \leq
    2 \sum_{k = 0}^{N-1}  \max \left\{ \frac{ \left( \int_{t_k}^{t_{k +1}} \bar \sched (s) \rmd s \right)^2 }{ 4 \sigma^2} , \int_{t_k}^{t_{k +1}} \bar \sched (s) \rmd s \right\} \left( \mathbb{E} \left[ \|Y_{t_{k+1}}\|^2 \right] - \mathbb{E} \left[ \|Y_{t_k}\|^2 \right] \right)
    \\
    &\hspace{.1cm} \leq
    2 \max_{0 \leq k \leq N-1} \left\{ \max \left\{ \frac{ \left( \int_{t_k}^{t_{k +1}} \bar \sched (s) \rmd s \right)^2 }{ 4 \sigma^2} , \int_{t_k}^{t_{k +1}} \bar \sched (s) \rmd s \right\} \right\}  \mathbb{E} \left[ \left\|\nabla \log \tilde p_{T-t_{N-1}} \left(\ola X_{t_{N-1}} \right) \right\|^2 \right]\eqsp.
\end{align*}
By H\ref{hyp:sched}, $t \mapsto \beta(t)$ is increasing, so that $t \mapsto \bar \sched(t)$ is decreasing. 
Therefore, using that since $\ola X_0 \sim p_T$, $\ola X_{T - \delta}$ and $\ora X_{\delta}$ have the same distribution, yields,
\begin{align*}
     \sum_{k = 0}^{N-1} \int_{t_k}^{t_{k+1}} &\bar \sched (t) \mathbb{E} \left[  \left\|   \nabla \log \tilde p_{T-t}\left(\ola X_t\right) - \nabla \log \tilde p_{T-t_k}\left(\ola X_{t_k}\right) \right\|^2 \right] \rmd t
    \\
    & \leq
    2 \max_{0 \leq k \leq N-1} \left\{ \max \left\{ \frac{ \left( (t_{k +1} - t_k) \bar \sched (t_k) \right)^2 }{ 4 \sigma^2} , (t_{k+1}  - t_k) \bar \sched (t_k)  \right\} \right\}\\
    &\hspace{6cm}\times\mathbb{E} \left[ \left\|\nabla \log \tilde p_{T-t_{N-1}} \left(\ola X_{t_{N-1}} \right) \right\|^2 \right] 
    \\
    & \leq
    2 \max_{0 \leq k \leq N-1}
    \left\{\max \left\{ \frac{  h^2 \bar \sched^2 (t_k)  }{4 \sigma^2} , h \bar \sched(t_k)  \right\} \right\}  \mathcal{I}(p_{T} Q_{T-\delta} | \pi_{\infty})   \\
    & \leq 2 h \bar \sched (0) \max \left\{ \frac{  h \bar \sched (0)  }{ 4 \sigma^2} , 1  \right\} \mathcal{I}(p_{T} Q_{T-\delta} | \pi_{\infty})   \\
    & \leq 2 h  \beta (T) \max \left\{ \frac{  h  \beta (T)  }{ 4 \sigma^2} , 1  \right\} \mathcal{I}(p_T Q_{T-\delta} | \pi_{\infty}) \eqsp.
\end{align*}
Finally, following the steps of the proof of \citet[Lemma 2]{conforti2023}, we can consider the limit when $\delta$ goes to zero, under Assumption H\ref{hyp:fisher_info}, concluding the proof.

\end{proof}

\subsection{Technical results}

\begin{lemma}\label{stationary distribution forward}
    Assume that H\ref{hyp:sched} and H\ref{hyp:fisher_info} hold. Let $(\ora X_t)_{t\geq 0}$ be a weak solution to the forward process (\ref{eq:forward-SDE:beta}). Then, the stationary distribution of $(\ora X_t)_{t\geq 0}$ is Gaussian with mean 0 and variance $\sigma^2 \mathrm{I}_d$.
\end{lemma}

\begin{proof}
Consider the process
\begin{align*}
\bar X_t = \exp \left(\frac{1}{2 \sigma^2} \int_0^t \sched(s) \rmd s\right) \ora X_t\eqsp. 
\end{align*}
Itô's formula yields
\begin{align}\label{eq:stationary_dirst1}
    \ora X_t & =  \exp \left(-\frac{1}{2 \sigma^2} \int_0^t \sched(s) \rmd s\right) \left(\ora X_0 +  \int_0^t \sqrt{ \sched(s)} \exp \left( \int_0^s \sched(u)/(2\sigma^2) \rmd u\right) \rmd B_s\right).
\end{align}
First, we have that
\begin{align*}
\lim_{t\to\infty}\exp \left(-\frac{1}{2 \sigma^2} \int_0^t \sched(s) \rmd s\right) \ora X_0 = 0\eqsp.
\end{align*}
Secondly, we have that the second term in the r.h.s. of \eqref{eq:stationary_dirst1}, by property of the Wiener integral, is Gaussian with mean $0$ and variance $\sigma^2_t \mathrm{I}_d$, where
\begin{align*}
\sigma^2_t = \exp \left(-\frac{1}{ \sigma^2} \int_0^t \sched(s) \rmd s\right) \int_0^t \sched(s) \rme^{ \int_0^s \sched(u)/\sigma^2 \rmd u} \rmd s = \sigma^2 \left( 1 - \exp \left(-\frac{1}{ \sigma^2} \int_0^t \sched(s) \rmd s\right) \right).
\end{align*}
By H\ref{hyp:sched}, $\lim_{t \to \infty} \sigma_t^2  = \sigma^2$, which concludes the proof. 
\end{proof}

\begin{lemma} \label{tvboundgirsanovproof} 
Let  $T > 0$ and $b_1, b_2 : [0, T] \times C([0, T], \mathbb{R}^d) \rightarrow \mathbb{R}^d $ be measurable functions such that for $i\in \{1, 2\}$,
\begin{align}\label{eq:uniqueness_KL_SDE}
\rmd X_t^{(i)} = b_i\left(t, ( X_s^{(i)})_{s \in [0, T]}\right) \rmd t + \sqrt{\sched(T-t)} \rmd B_t
\end{align}
admits a unique strong solution with  $X_0^{(i)} \sim  \pi_0^{(i)}$. Suppose that $(b_i(t, (X_s^{(i)})_{s \in [0, t]}))_{t \in [0, T]} $  is progressively measurable, with Markov semi-group $(P_t^{(i)})_{t \ge 0}$. In addition, assume that 
\begin{equation}
\label{eq:novi}
  \mathbb{E} \left[ \exp\left\{  \frac{1}{2} \int_0^T \frac{1}{{\sched(T-s)}}  \left\|b_1\left(s, \left(X^{(1)}_u\right)_{ u \in [0,s]}\right) - b_2\left(s, \left(X^{(1)}_u\right)_{ u \in [0,s]}\right) \right\|^2 \rmd s \right\} \right] < \infty\eqsp.  
\end{equation}
Then,
\begin{multline}
\kl \left( \pi_0^{(1)} P_T^{(1)} \| \pi_0^{(2)} P_T^{(2)} \right) \leq \kl \left(\pi_0^{(1)} \| \pi_0^{(2)} \right)\\
+ \frac{1}{2} \int_0^T \frac{1}{\sched(T-t)} \mathbb{E}\left[\left\|b_1\left(s, \left(X^{(1)}_u\right)_{ u \in [0,s]}\right) - b_2\left(s, \left(X^{(1)}_u\right)_{ u \in [0,s]}\right) \right\|^2\right] \rmd t\eqsp. 
\end{multline}
\end{lemma}

\begin{proof}
Consider the probability space $(\Omega,(\mathcal{F}_t)_{0\leq t\leq T},\mathbb{P})$ and for $i \in\{1,2\}$, let \( \mu^{(i)} \) be the distribution of \( ( X_t^{(i)})_{t \in [0,T]} \) on the Wiener space \( (C([0, T]; \mathbb{R}^d), \mathcal{B}(C([0, T]; \mathbb{R}^d))) \) with \( X_0^{(i)} \sim \pi_0^{(i)} \). Define $u(t,\omega)$ as
\begin{align*}
u(t,\omega) \eqdef \sched(T-t)^{-1/2}\left( b_1\left(t, \left(X^{(1)}_u\right)_{ u \in [0,t]}\right) -  b_2\left(t, \left(X^{(1)}_u\right)_{ u \in [0,t]}\right)\right)\eqsp,
\end{align*}
and define $\rmd \mathbb{Q}/\rmd\mathbb{P}(\omega) = M_T (\omega)$ where, for $t \in[0,T]$,
\begin{align*}
    M_t(\omega)  = \exp\left\{ - \int_0^t u(s, \omega)^\top \rmd B_s - \frac{1}{2} \int_0^t \|u (s , \omega)\|^2 \rmd s\right\}\eqsp.
\end{align*}
From \eqref{eq:novi}, the Novikov's condition is satisfied \citep[][Chapter 3.5.D]{karatzas2012brownian}, thus the process $(M_t)_{0\leq t\leq T}$ is a martingale. 
Applying Girsanov theorem,  $\rmd \bar B_t = \rmd B_t + u(t, (X^{(1)}_s)_{s \in [0,t]}) \rmd t$ is a Brownian motion under the measure $\mathbb{Q}$. Therefore,
\begin{align*}
\rmd X_t^{(1)} = b_1\left(t, \left(X^{(1)}_u\right)_{ u \in [0,t]}\right) \rmd t + \sqrt{\sched(T-t)} \rmd B_t =
b_2\left(t, \left(X^{(1)}_u\right)_{ u \in [0,t]}\right) \rmd t + \sqrt{\sched(T-t)} \rmd \bar B_t\eqsp.
\end{align*}
Using the uniqueness in law of \eqref{eq:uniqueness_KL_SDE}, the law of $X^{(1)}$ under $\mathbb{P}$ is the same as the one of $\bar X^{(2)}$ under $\mathbb{Q}$, with $\bar X^{(2)}$ solution of \eqref{eq:uniqueness_KL_SDE} with $i=2$ and $\bar X^{(2)}_0=\pi^{(1)}_0$. Denote by $\bar\mu^{(2)}$ the law of $\bar X^{(2)}$.
Therefore,
\begin{eqnarray*}
\mu^{(1)}(A) = \mathbb{P}(X^{(1)}\in A) = \mathbb{Q}(\bar X^{(2)}\in A) = \int \1_A(\bar X^{(2)}(\omega))\mathbb{Q}(\rmd\omega) \eqsp,
\end{eqnarray*}
which implies that
\begin{align*}
\frac{\rmd \bar\mu^{(2)}}{\rmd \mu^{(1)}} =  M_T\eqsp.
\end{align*}
Hence, we obtain that
\begin{align*}
  \kl \left(\mu^{(1)} \middle\| \mu^{(2)}\right) 
  &= \kl \left( \pi_0^{(1)} \middle\| \pi_0^{(2)} \right) + \mathbb{E}  \left[\log\left(\frac{\rmd \mu^{(1)}}{\rmd \bar \mu^{(2)}}\right)\right]\\
  & = \kl \left( \pi_0^{(1)} \middle\| \pi_0^{(2)} \right) +  \mathbb{E}  \left[  \int_0^t u(s, \omega)^\top \rmd B_s + \frac{1}{2} \int_0^t \|u (s , \omega)\|^2 \rmd s \right] \\
  &= \kl \left( \pi_0^{(1)} \middle\| \pi_0^{(2)} \right) +  \frac{1}{2} \int_0^T \frac{1}{\sched(T-t)} \mathbb{E}\left[ \left\| b_1(t, (X_s^{(1)})_{s \in [0, t]}) - b_2(t, (X_s^{(1)})_{s \in [0, t]}) \right\|^2 \right] \rmd t\eqsp,
\end{align*}
which concludes the proof.
\end{proof}

\begin{lemma}
\label{stam-gross}
Let $p$ be a probability density function on $\rset^d$. For all $\sigma^2>0$,
    \begin{equation*}
\kl \left(p \| \gausspdf_{\sigma^2 }\right) = \int p(x) \log \frac{p(x)}{\gausspdf_{\sigma^2 }(x)} \, \rmd x \leq \frac{\sigma^2 }{2} \int \left\|\nabla \log \frac{p(x)}{\gausspdf_{\sigma^2 }(x)} \right\|^2 \, p(x) \, \rmd x.
\end{equation*}
\end{lemma}
\begin{proof}
 Define $f_{\sigma^2}:x\mapsto p(x)/\varphi_{\sigma^2}(x)$. Since $\nabla^2 \log \varphi_{\sigma^2}(x) = -\sigma^{-2} \mathrm{I}_d$, the  Bakry-Emery criterion is satisfied with constant ${\sigma^2}^{-1}$, see \cite{bakry2014analysis,villani2021topics,talagrand1996transportation}. By the classical  logarithmic Sobolev inequality,
\begin{align*}
\int f_{\sigma^2}(x)\log f_{\sigma^2}(x)\varphi_{\sigma^2}(x)\rmd x \leq \frac{{\sigma^2}}{2}\int \frac{\|\nabla f_{\sigma^2}(x)\|^2}{f_{\sigma^2}(x)}\varphi_{\sigma^2}(x)\rmd x\eqsp,
\end{align*}
which concludes the proof.
\end{proof}

\begin{lemma} \label{score differential form}
 Define $Y_t \eqdef \nabla \log \tilde p_{T-t} (\ola X_t)$ and  $Z_t \eqdef \nabla^2 \log \tilde p_{T-t} (\ola X_t)$, where $\{\ola X_t\}_{t\geq 0}$ is a weak solution to \eqref{eq:forward:general}. Then, 
\begin{align}\label{eq:SDE_for_Y}
    \rmd Y_t = \left( \frac{\bar g^2(t)}{\sigma^2} - \bar \alpha(t) \right) Y_t \rmd t - \frac{2}{\sigma^2} \left( \frac{\bar g^2(t)}{2 \sigma^2} - \bar \alpha(t) \right) \ola X_t \rmd t + \bar g(t) Z_t \rmd \bar B_t\eqsp.
\end{align}
\end{lemma}

\begin{proof}
The Fokker-Planck equation associated with the forward process \eqref{eq:forward:general} is
\begin{align}
\label{eq:fp}
    \partial_{t} p_{t}(x) = \alpha (t) \text{div} \left( x p_{t} (x) \right) + \frac{g^2 (t)}{2} \Delta p_{t} (x)\eqsp,
\end{align}
for $x \in \R^d$. First, we prove that $\tilde p_t$ satisfies the following PDE
\begin{equation}\label{eq:PDE_log_p_t}
    \begin{split}
        \partial_{t} \log \tilde p_{t}(x)  = d \left(  \bar \alpha(t) - \frac{\bar g^2(t)}{2 \sigma^2} \right) + \langle \nabla \log \tilde p_{t}(x),x \rangle  \left( \bar \alpha(t) - \frac{ \bar g^2(t)}{\sigma^2 } \right)  + \frac{\|x\|^2}{\sigma^2}  \left( \frac{\bar g^2(t)}{2 \sigma^2} - \bar \alpha(t) \right) \\+ \frac{\bar g^2(t)}{2} \frac{\Delta \tilde p_{t}(x)}{\tilde p_{t}(x)}\eqsp.
    \end{split}
\end{equation}

Using that $\nabla \log \gausspdf_{\sigma^2}(x) = -x/\sigma^2$, we have
\begin{align*}
 \text{div} ( x p_{t} (x) )
 &=d \;p_{t}(x) + p_{t} (x)  \;x^\top \nabla \log p_{t} (x) \\
 &= \gausspdf_{\sigma^2}(x)\left(d \;\tilde p_{t}(x)  + \tilde p_{t}(x)  \nabla \log \tilde p_{t}(x)^\top x - \frac{\|x\|}{\sigma^2} \right)\\
 &= \gausspdf_{\sigma^2}(x)\left(d \;\tilde p_{t}(x)  + \nabla \tilde p_{t}(x)^\top x - \frac{\|x\|}{\sigma^2} \tilde p_{t}(x)\right)\eqsp.
\end{align*}
Then, since $\Delta \gausspdf_{\sigma^2}(x) =(\gausspdf_{\sigma^2}(x)/\sigma^2) \left(\|x\|^2/\sigma^2 -d \right)$, we get
\begin{align*}
    \Delta p_{t}(x) &= \tilde p_{t}(x) \Delta \gausspdf_{\sigma^2}(x) + 2  \nabla \tilde p_{t}(x)^\top \nabla \gausspdf_{\sigma^2}(x) + \gausspdf_{\sigma^2}(x) \Delta \tilde p_{t}(x)
    \\
    &=\gausspdf_{\sigma^2}(x) \left(
    \frac{\tilde p_{t}(x)}{\sigma^2} \left( \frac{\|x\|^2}{\sigma^2} - d \right) - \frac{2}{\sigma^2}\nabla \tilde p_{t}(x)^\top x + \Delta \tilde p_{t}(x) \right)
    \eqsp.
\end{align*}
Combining these results with \eqref{eq:fp}, we obtain
\begin{align*}
     \partial_{t} \tilde p_{t}(x)&= d\; \tilde p_{t}(x) \left(  \alpha(t) - \frac{g^2(t)}{2 \sigma^2} \right) + \nabla \tilde p_{t}(x)^\top x \left( \alpha(t) - \frac{ g^2(t)}{\sigma^2 } \right)  \\
     &\hspace{3cm}+ \tilde p_{t}(x)\frac{\|x\|^2}{\sigma^2} \left( \frac{g^2(t)}{2 \sigma^2} - \alpha(t) \right) + \frac{g^2(t)}{2} \Delta \tilde p_{t}(x)\eqsp. 
\end{align*}
Hence, diving by $\tilde p_t$ yields \eqref{eq:PDE_log_p_t}.

The previous computation, together with the fact that $\Delta \tilde p_{t}/\tilde p_{t} = \Delta \log \tilde p_{t} + \|\nabla \log \tilde p_{t}\|^2$, yields that the function $\phi_t(x) \eqdef \log \tilde p_{T-t} (x)$ is a solution to the following PDE
\begin{align}\label{eq:PDE_phi_t}
\partial_t \phi_t(x) &= - d \left(  \bar \alpha(t) - \frac{\bar g^2(t)}{2 \sigma^2} \right) -  \nabla \phi_t (x)^\top x \left( \bar \alpha(t) - \frac{\bar g^2(t)}{\sigma^2 } \right)  \\
&\qquad- \frac{ \|x\|^2}{\sigma^2} \left( \frac{\bar g^2(t)}{2 \sigma^2} - \bar \alpha(t) \right) - \frac{\bar g^2(t)}{2} \left( \Delta \phi_t(x) + \| \nabla \phi_t(x)\|^2 \right)\eqsp. 
\end{align}
Following the lines of \citet[][Proposition 1]{conforti2023}, we get that, since $\alpha$ and $g$ are continuous and non-increasing, the map $p_t$, solution to \eqref{eq:fp}, belongs to $C^{1,2}((0,T]\times\R^d)$.
By \eqref{eq:backward:general}, as $Y_t = \nabla\phi_{t}(\ola X_t)$, we can apply Itô's formula and obtain, writing $\bar \gamma(t) = \bar \alpha(t) - \bar g(t)^2/\sigma^2$,
\begin{align*}
\rmd  Y_t &=
\left[
    \partial_t \nabla \phi_t\left(\ola X_t\right) + \nabla^2 \phi_t\left(\ola X_t\right)\left(    
        \bar\gamma(t)\ola X_t + \bar g^2 (t) \nabla\phi_t\left(\ola X_t\right)
    \right)+
    \frac{\bar g^2(t)}{2} \Delta \nabla \phi_t\left(\ola X_t\right)
\right]\rmd t \\
&\qquad\qquad\qquad\qquad\qquad\qquad
\qquad\qquad\qquad\qquad\qquad\qquad
\qquad\qquad+ \bar g(t) \nabla^2 \phi_t\left(\ola X_t\right) \rmd \bar B_t
\\
&=
\left[\nabla\left(
        \partial_t \phi_t\left(\ola X_t\right)+
        \frac{\bar g^2(t)}{2}\left(\Delta \phi_t\left(\ola X_t\right)+
        \left\|\nabla\phi_t\left(\ola X_t\right)\right\|^2\right)
    \right) + 
    \bar\gamma(t)
    \nabla^2 \phi_t\left(\ola X_t\right) \ola X_t
\right]\rmd t \\
&\qquad\qquad\qquad\qquad\qquad\qquad
\qquad\qquad\qquad\qquad\qquad\qquad
\qquad\qquad + \bar g(t) \nabla^2 \phi_t\left(\ola X_t\right) \rmd \bar B_t\eqsp,
\end{align*}
using that $ 2 \nabla^2 \phi_t(x)  \nabla \phi_t(x) = \nabla \| \nabla \phi_t(x) \|^2$. Using \eqref{eq:PDE_phi_t}, we get
\begin{align*}
\rmd  Y_t &=
\left[
    - \bar\gamma(t)\nabla \psi_t\left(\ola X_t\right)  
        + \frac{2}{\sigma^2}\left( \bar \alpha(t) - \frac{\bar g^2(t)}{2 \sigma^2} \right)  \ola X_t
    + \bar\gamma(t) \nabla^2 \phi_t\left(\ola X_t\right) \ola X_t
\right]\rmd t \\
&\qquad\qquad\qquad\qquad\qquad\qquad
\qquad\qquad\qquad\qquad\qquad
\qquad\qquad\qquad+ \bar g(t) \nabla^2 \phi_t\left(\ola X_t\right) \rmd \bar B_t\eqsp,
\end{align*}
with $\psi_t\left(x\right) \eqdef \nabla \phi_t(x)^\top x$. With the identity $\nabla \left( x^\top \nabla \phi_t(x) \right) = \nabla \phi_t (x) + \nabla^2 \phi_t(x) x$, we have
\begin{align*}
\rmd  Y_t &=
\left[
    \left( \frac{\bar g^2(t)}{\sigma^2 } - \bar \alpha(t) \right)\nabla \phi_t\left(\ola X_t\right)
        + \frac{2}{\sigma^2}\left( \bar \alpha(t) - \frac{\bar g^2(t)}{2 \sigma^2} \right)  \ola X_t
\right]\rmd t + \bar g(t) \nabla^2 \phi_t\left(\ola X_t\right) \rmd \bar B_t
\\
&=
\left[
    \left( \frac{\bar g^2(t)}{\sigma^2 } - \bar \alpha(t) \right)Y_t
        + \frac{2}{\sigma^2}\left( \bar \alpha(t) - \frac{\bar g^2(t)}{2 \sigma^2} \right)  \ola X_t
\right]\rmd t + \bar g(t) Z_t \rmd \bar B_t\eqsp,
\end{align*}
which concludes the proof.
\end{proof}

\begin{lemma} \label{score moment bound }
Let $Y_t \eqdef \nabla \log \tilde p_{T-t} ( \ola X_t )$, with $\ola X$ satisfying \eqref{eq:backward:general}. There exists a constant $C>0$ such that
\begin{align}\label{eq:moment4_Y}
\mathbb{E} \left[ \left\| Y_t \right\|^4 \right] \leq C \left(
    \sigma_{T-t}^{-4}
    \E\left[\left\| N \right\|^4\right] + \sigma^{-8} \E\left[\left\| \ora X_0 \right\|^4 \right]
     \right),
\end{align}
with $N\sim\Nc(0,\mathrm{I}_d)$ and $\sigma_t^2$ as in \eqref{eq:cond_variance}.
\end{lemma}

\begin{proof}
The transition density $q_t(y,x) $ associated with the semi-group of the process (\ref{eq:forward:general}) is given by 
\begin{align*}
q_t(y,x) =  \left(2 \pi \sigma_t^2 \right)^{-d/2} \exp\left( { \frac{-\left\| x - y \exp\left({- \int_0^t \alpha (s) \rmd s}\right) \right\|^2}{2 \sigma_t^2}} \right)\eqsp. 
\end{align*}
Therefore, we have 
\begin{align*}
    \nabla \log p_{T-t} (x) &= \frac{1}{p_{T-t}(x)} \int p_0(y)\nabla_x q_{T-t}(y,x) \rmd y \\
    &= \frac{1}{p_{T-t}(x)} \int p_0(y) \frac{y \exp\left({- \int_0^{T-t} \alpha (u) \rmd u}\right) - x }{\sigma^2_{T-t}}
    q_{T-t}(y,x) \rmd y\eqsp. 
\end{align*}
This, together with the definition of $\tilde p$, yields 
\begin{align*}
\nabla \log \tilde p_{T-t} \left( \ora X_{T-t} \right) = \sigma_{T-t}^{-2} \mathbb{E} \left[ \ora X_0 \rme^{- \int_0^{T-t} \alpha (u) \rmd u} - \ora X_{T-t} \middle| \ora X_{T-t}\right] + \sigma^{-2}\ora X_{T-t}\eqsp.
\end{align*}
Using Jensen's inequality for conditional expectation, there exists a constant $C>0$ (which may change from line to line) such that 
\begin{align*}
     \left\| \nabla \log \tilde p_{T-t} \left( \ora X_{T-t} \right) \right\|^4 
     & \leq
     C\left(\sigma_{T-t}^{-8} \left\| \mathbb{E} \left[  \ora X_{0} \rme^{- \int_0^{T-t} \alpha (s) \rmd s} - \ora X_{T-t} \middle| \ora X_{T-t}  \right] \right\|^4+
     \sigma^{-8} \left\|\ora X_{T-t}\right\|^4
     \right)\\
    & \leq C\left(\sigma_{T-t}^{-8}  \mathbb{E} \left[ \left\|  \ora X_{0} \rme^{- \int_0^{T-t} \alpha (s) \rmd s} - \ora X_{T-t} \right\|^4 \middle| \ora X_{T-t}  \right]+
     \sigma^{-8}\left\|\ora X_{T-t}\right\|^4
     \right)\eqsp.
\end{align*}
Note that $\ora X_t$ has the same law as $\exp(-\int_0^t \alpha(s) \rmd s) \ora X_0  + \sigma_t N$, with $N \sim \mathcal{N}(0, \mathrm{I}_d)$. This means that we have that
\begin{align*}
     \E\left[\left\| \nabla \log p_{T-t} \left( \ora X_{T-t} \right) \right\|^4\right] & \leq C\sigma_{T-t}^{-4} \left(    \E\left[\left\| N \right\|^4\right] + \E\left[\left\| \ora X_0 \right\|^4 \right]
     \right)\eqsp. 
\end{align*}

Finally,
\begin{align*}
    \mathbb{E}  \left[ \left\| Y_t \right\|^4  \right]      = \mathbb{E}  \left[ \left\| \nabla \log \tilde p_{T-t} \left( \ola X_t \right) \right\|^2  \right]
    &= \mathbb{E}  \left[ \left\| \nabla \log \tilde p_{T-t} \left( \ora X_{T-t} \right) \right\|^4  \right] \\
    &\leq \sigma_{T-t}^{-4}\mathbb{E}  \left[ \left\| N \right\|^4 \right]
    \\ &\leq C \left(
    \sigma_{T-t}^{-4}
    \E\left[\left\| N \right\|^4\right] + \sigma^{-8} \E\left[\left\| \ora X_0 \right\|^4 \right]
     \right)\eqsp,
\end{align*}
which concludes the proof.
\end{proof}

\begin{lemma} \label{second derivative moment bound}
Let $Z_t \eqdef \nabla^2 \log \tilde p_{T-t} ( \ola X_t )$, where $\{\ola X_t\}_{t\geq 0}$ is a weak solution to \eqref{eq:backward:general}. There exists a constant $C>0$ such that
\begin{align}\label{eq:moment4_Z}
\mathbb{E} \left[ \left\| Z_t \right\|^4 \right] \leq C \left(\sigma^{-8}_{T-t}+ \sigma^{-8}\right) \left( \mathbb{E} \left[ \left\|   Z \right\|^8_2 \right] + d^4 \right)
\eqsp,
\end{align}
with $Z\sim\Nc(0,\mathrm{I}_d)$ and $\sigma_t^2$ as in \eqref{eq:cond_variance}.
\end{lemma}

\begin{proof}

Let $q_t(y,x) $ be the transition density associated to the semi-group of the process (\ref{eq:forward:general}). Write 
\begin{align*}
    &\nabla^2 \log p_{T-t} (x)\\
    & \hspace{.3cm}=
    \nabla \left( \frac{1}{p_{T-t}(x)} \int p_0(y) \frac{y e^{- \int_0^{T-t} \alpha (s) \rmd s} - x }{\sigma^2_{T-t}}q_{T-t}(y,x) \rmd y \right) 
    \\
    & \hspace{.3cm}= -  \frac{\nabla p_{T-t} (x)}{p_{T-t}^2(x)} \left( \int p_0(y) \frac{y e^{- \int_0^{T-t} \alpha (s) \rmd s} - x }{\sigma^2_{T-t}} q_{T-t}(y,x)\rmd y \right)^\top
    \\
    &\hspace{.3cm}\quad
    +
    \frac{1}{p_{T-t}(x)} \nabla\int p_0(y) \frac{y e^{- \int_0^{T-t} \alpha (s) \rmd s} - x }{\sigma^2_{T-t}}
    q_{T-t}(y,x) \rmd y \\
    & \hspace{.3cm}= \frac{1}{\sigma^2_{T-t}\;p_{T-t}(x)}
    \Bigg(- \int \left( \frac{\nabla p_{T-t} (x)}{p_{T-t}(x)} \right) \left(  \frac{y e^{- \int_0^{T-t} \alpha (s) \rmd s} - x }{\sigma^2_{T-t}} \right)^\top 
    q_{T-t}(y,x) p_0(y) \rmd y
    \\
    &\hspace{.3cm}\quad 
    - \mathrm{I}_d + 
    \int \frac{1}{\sigma^2_{T-t}}\left(  y e^{- \int_0^{T-t} \alpha (s) \rmd s} - x  \right)
    \left(  
    y e^{- \int_0^{T-t} \alpha (s) \rmd s} - x  \right)^\top 
    q_{T-t}(y,x) p_0(y) \rmd y
    \Bigg)
    \eqsp.
\end{align*}

Therefore, 
\begin{align*}
     &\nabla^2 \log \tilde p_{T-t} \left( \ora X_{T-t} \right)\\ 
     & \hspace{.3cm}=
     - \frac{1}{\sigma_{T-t}^{2}}\left(  \mathbb{E} \left[ \left( \frac{\nabla p_{T-t} \left( \ora X_{T-t} \right)}{p_{T-t} \left( \ora X_{T-t} \right)} \right) \left(  \ora X_0 e^{- \int_0^{T-t} \alpha (s) \rmd s} - \ora X_{T-t} \right)^\top
     \middle| \ora X_{T-t}   \right] + \mathrm{I}_d\right)
     \\
    & \hspace{.3cm}\quad 
    + \sigma_{T-t}^{-4}  \mathbb{E} \left[ \left(\ora X_0 e^{- \int_0^{T-t} \alpha (s) \rmd s} - \ora X_{T-t} \right) \left(\ora X_0 e^{- \int_0^{T-t} \alpha (s) \rmd s} - \ora X_{T-t} \right)^\top \middle| \ora X_{T-t} \right] +
    \sigma^{-2}\mathrm{I}_d
    \\
    &\hspace{.3cm} = - \sigma_{T-t}^{-4}  \left(\mathbb{E} \left[ \ora X_0 e^{- \int_0^{T-t} \alpha (s) \rmd s} - \ora X_{T-t} \middle| \ora X_{T-t}\right] \right)\left( \mathbb{E} \left[    \ora X_0 e^{- \int_0^{T-t} \alpha (s) \rmd s} - \ora X_{T-t}  \middle| \ora X_{T-t}  \right] \right)^\top \\
    & \hspace{.5cm}\quad  + \left(\sigma^{-2}- \sigma_{T-t}^{-2}\right) \mathrm{I}_d 
    \\
   & \hspace{.3cm}\quad
   + \sigma_{T-t}^{-4}  \mathbb{E} \left[ \left(\ora X_0 e^{- \int_0^{T-t} \alpha (s) \rmd s} - \ora X_{T-t} \right) \left(\ora X_0 e^{- \int_0^{T-t} \alpha (s) \rmd s} - \ora X_{T-t} \right)^\top \middle| \ora X_{T-t} \right]\eqsp.
\end{align*}
There exists a constant $C>0$ (which may change from line to line) such that
\begin{align*}
    &\mathbb{E} \left[ \left\| \nabla^2 \log p_{T-t} \left( \ora X_{T-t} \right) \right\|^4_{\mathrm{Fr}} \right]\\
    & \leq \frac{C}{\sigma_{T-t}^{16}} \mathbb{E} \left[\left\| \mathbb{E} \left[ \ora X_0 e^{- \int_0^{T-t} \alpha (s) \rmd s} - \ora X_{T-t} \middle| \ora X_{T-t}\right] \mathbb{E} \left[    \ora X_0 e^{- \int_0^{T-t} \alpha (s) \rmd s} - \ora X_{T-t}  \middle| \ora X_{T-t}   \right]^\top \right\|^4_{\mathrm{Fr}} \right]
    \\
    & \quad + C \left(\sigma^{-8}_{T-t}+ \sigma^{-8}\right) d^4
    \\
    & \quad + \frac{C}{\sigma_{T-t}^{16}} \mathbb{E } \left[ \left\| \mathbb{E} \left[ \left(\ora X_0 e^{- \int_0^{T-t} \alpha (s) \rmd s} - \ora X_{T-t} \right) \left(\ora X_0 e^{- \int_0^{T-t} \alpha (s) \rmd s} - \ora X_{T-t} \right)^\top 
    \middle| \ora X_{T-t} \right]  \right\|^4_{\mathrm{Fr}} \right].
\end{align*}

As in the previous proof, we note that $\ora X_t$ has the same law as $\rme^{-\int_0^t \alpha(s) \rmd s} \ora X_0  + \sigma_t Z$, with $Z \sim \mathcal{N}(0, \mathrm{I}_d)$ independent of $\ora X_0$. Therefore, using Jensen's inequality,
\begin{align*}
    & \mathbb{E} \left[\left\| \mathbb{E} \left[ \ora X_0 e^{- \int_0^{T-t} \alpha (s) \rmd s} - \ora X_{T-t} \middle| \ora X_{T-t}\right] \mathbb{E} \left[    \ora X_0 e^{- \int_0^{T-t} \alpha (s) \rmd s} - \ora X_{T-t}  \middle| \ora X_{T-t}   \right]^\top \right\|^4_{\mathrm{Fr}} \right] \\
    & \leq \mathbb{E} \left[\left\| \mathbb{E} \left[ \ora X_0 e^{- \int_0^{T-t} \alpha (s) \rmd s} - \ora X_{T-t} \middle| \ora X_{T-t}\right] \right\|^4_2 \left\| \mathbb{E} \left[    \ora X_0 e^{- \int_0^{T-t} \alpha (s) \rmd s} - \ora X_{T-t}  \middle| \ora X_{T-t}   \right] \right\|^4_{2} \right] \\
    & \leq \mathbb{E} \left[ \mathbb{E} \left[ \left\|  \ora X_0 e^{- \int_0^{T-t} \alpha (s) \rmd s} - \ora X_{T-t} \right\|^8_2 \middle| \ora X_{T-t}\right] \right]  \\
    & \leq \sigma_t^8 \mathbb{E} \left[ \left\|   Z \right\|^8_2 \right]
\end{align*}
and
\begin{align*}
    & \mathbb{E } \left[ \left\| \mathbb{E} \left[ \left(\ora X_0 e^{- \int_0^{T-t} \alpha (s) \rmd s} - \ora X_{T-t} \right) \left(\ora X_0 e^{- \int_0^{T-t} \alpha (s) \rmd s} - \ora X_{T-t} \right)^\top \middle| \ora X_{T-t} \right]  \right\|^4_{\mathrm{Fr}} \right] \\
    & \leq \mathbb{E } \left[  \mathbb{E} \left[ \left\| \left(\ora X_0 e^{- \int_0^{T-t} \alpha (s) \rmd s} - \ora X_{T-t} \right)
    \left(\ora X_0 e^{- \int_0^{T-t} \alpha (s) \rmd s} - \ora X_{T-t} \right)^\top  \right\|^4_{\mathrm{Fr}}  \middle| \ora X_{T-t} \right]  \right] \\
    & =  \mathbb{E} \left[ \left\| \left(\ora X_0 e^{- \int_0^{T-t} \alpha (s) \rmd s} - \ora X_{T-t} \right) \right\|^8_{\rm 2}  \right]   \\
    & \leq \sigma_t^8 \mathbb{E} \left[ \left\|   Z \right\|^8_2 \right]\eqsp.
\end{align*}
Hence, we can conclude that
\begin{align*}
    \mathbb{E} \left[ \left\| Z_t \right\|^4_{\mathrm{Fr}} \right] = \mathbb{E} \left[ \left\| \nabla^2 \log \tilde p_{T-t} \left( \ora X_{T-t} \right) \right\|^4_{\mathrm{Fr}} \right] \leq C \left(\sigma^{-8}_{T-t}+ \sigma^{-8}\right) \left( \mathbb{E} \left[ \left\|   Z \right\|^8_2 \right] + d^4 \right)\eqsp.
\end{align*}
\end{proof}

\section{Proofs of Section~\ref{sec:wasserstein}}
\label{sec:proof:wasserstein}

\subsection{Gaussian case: proof of Lemma~\ref{lem:gaussian:contract}}
In the case where $\pi_{\rm data}$ is the Gaussian probability density with mean $\mu_0$ and variance $\Sigma_0$, we have 
\begin{align*}
\nabla\log\tilde p_t(x) = -\left(m_t^2\Sigma_0 + \sigma_t^2 \mathrm{I}_d\right)^{-1}(x-m_t\mu_0) + \sigma^{-2}x\eqsp,
\end{align*}
with $m_t = \exp\left(-\int_0^t\beta(s)\rmd s/(2\sigma^2)\right)$ and $\sigma_t =
\sigma^2(1-m^2_t) $. Let $\ora \Sigma_t = m_t^2\Sigma_0 + \sigma_t^2 \mathrm{I}_d$ be the covariance of the forward process $\ora X_t$ and $b_t = \ora \Sigma_t^{-1} m_t \mu_0$ so that 
\begin{equation} \label{eq:gaussian_score_linear}
\nabla\log\tilde p_t(x) = A_t x + b_t \quad \text{with} \quad A_t = - \left(\ora {\Sigma_t}^{-1} - \sigma^{-2} \mathrm{I}_d \right) \eqsp.
\end{equation}
Note that, if we denote by $\lambda_0^1\leq\dots\leq\lambda^d_0$ the eigenvalues of $\Sigma_0$, which are positive as $\Sigma_0$ is positive definite, we have that the eigenvalues of $A_t$ are
\begin{align*}
    \lambda^i_t\eqdef -\frac{1}{m_t^2\lambda^i_0 + \sigma_t^2} + \frac{1}{\sigma^2}\eqsp.
\end{align*}
It is straightforward to see that $\lambda^1_t\leq \dots \leq \lambda^d_t$. Moreover, we always have that in this case
\begin{align*}
    \left(\nabla \log \tilde p_t(x) - \nabla \log \tilde p_t(y)\right)^\top(x-y)&\leq \lambda^d_t \left\|x-y\right\|^2\eqsp,
    \\
    \left\|\nabla \log \tilde p_t(x) - \nabla \log \tilde p_t(y)\right\|&\leq
    \max \left\{\left|\lambda^1_t\right|, \left|\lambda^d_t\right|\right\} \left\|x-y\right\|\eqsp,
\end{align*}
which entails that we can define
\begin{align*}
L_t\eqdef\max \left\{\left|\lambda^1_t\right|, \left|\lambda^d_t\right|\right\}\eqsp,\qquad
C_t\eqdef-\lambda^d_t\eqsp,
\end{align*}
and apply Proposition \ref{prop:contractivity}.

The condition $\lambda^d_t \le 0$, or equivalently $\sigma^2 \ge \lambda_{\max}(\Sigma_0)$, yields a contraction in $2$--Wasserstein distance in the backward process as well in the forward process from Proposition \ref{prop:contractivity}. This shows that, in specific cases, with an appropriate calibration of the variance of the stationary law with respect to the initial law, we have a contraction both in the forward and in the backward flows. 

As a consequence, note that
{
\begin{equation*}
    \Wc_2\left(\pi_{\rm data}, \pi_{\infty} Q_T \right)^2 \le \Wc_2\left(p_T, \pi_{\infty}\right)^2 \exp\left( - \frac{1}{\sigma^2} \int_0^T \beta(t)(1 + 2 C_t \sigma^2) \rmd t \right).
\end{equation*}
Using Talagrand’s $T_2$ inequality for the Gaussian measure $\Wc_2\left(\mu, \pi_{\infty}\right)^2 \le 2 \sigma^2 \kl(\mu \| \pi_{\infty})$ and Lemma~\ref{lem:mix} we get 
\begin{equation*}
    \Wc_2\left(\pi_{\rm data}, \pi_{\infty} Q_T \right)^2 \le 2 \sigma^2 \kl \left(\pi_{\rm data} \| \pi_{\infty} \right) \exp\left( - \frac{2}{\sigma^2} \int_0^T \beta(t)(1 + 2C_t \sigma^2) \rmd t \right).
\end{equation*}
}

\begin{proposition}\label{prop:back_contractivity_kl}
Assume that $\pi_{\rm data}$ is a Gaussian distribution $\mathcal{N}(\mu_0, \Sigma_0)$ such that $\lambda_{\max}(\Sigma_0) \le \sigma^2$ where $\lambda_{\max}(\Sigma_0)$ denotes the largest eigenvalue of $\Sigma_0$. Then, 
\begin{equation*}
    \kl\left(\pi_{\rm data} \| \pi_{\infty} Q_T \right) \le \kl \left(\pi_{\rm data} \| \varphi_{\sigma^2} \right) \exp\left( - \frac{2}{\sigma^2} \int_0^T \beta(s) \rmd s \right).
\end{equation*}
\end{proposition}

\begin{proof}
In this Gaussian case, the backward process is linear (see~\eqref{eq:gaussian_score_linear}) and the associated infinitesimal generator writes, for $g \in \mathcal{C}^2$,  
\begin{equation*}
    \ola{\mathcal{L}}_t g(x) = \nabla g(x)^\top \left(- \frac{\bar \beta(t)}{2 \sigma^2} + \bar \beta(t) (\bar A_t x + \bar b_t) \right) + \frac{1}{2} \bar \beta(t) \Delta g(x),
\end{equation*}
where $\bar A_t = A_{T-t}$ and $\bar b_t = b_{T-t}$. 

Our objective is to monitor the evolution of the Kullback-Leibler divergence, $\kl(p_T Q_t \|\varphi_{\sigma^2} Q_t)$, for $t \in [0, T]$. We follow~\citet[Section 6]{del_moral_contraction_2003} \citep[see also][]{collet_logarithmic_2008}. Let $q_t = p_T Q_t$ and $\phi_t = \varphi_{\sigma^2} Q_t$ two densities 
that satisfy the Fokker-Planck equation, involving the dual operator $\ola{\mathcal{L}}^*_t$ of the infinitesimal generator $\ola{\mathcal{L}}$ 
\begin{align*}
    \partial_t q_t &= \ola{\mathcal{L}}^*_t q_t, \qquad q_0(x) = p_T(x) \\
    \partial_t \phi_t &= \ola{\mathcal{L}}^*_t \phi_t, \qquad \phi_0(x) = \varphi_{\sigma^2}(x).
\end{align*}

Let $f_t = q_t / \phi_t$.  By definition of $\kl(q_t \| \phi_t) = \int \ln\left(f_t(x)\right) q_t(x) \rmd x$ we have 
\begin{align*}
    \partial_t \kl \left( q_t \| \phi_t \right) &= 
    \int \ln \left(f_t(x)\right) \partial_t q_t(x) \rmd x
    + \int \partial_t \ln\left(f_t(x)\right) q_t(x) \rmd x \\
    &= \int \ln \left(f_t(x)\right) \partial_t q_t(x) \rmd x
    - \int f_t(x) \partial_t \phi_t(x) \rmd x\eqsp.
\end{align*}
By employing the Fokker-Planck equation and the adjoint relation, which states that $\int f(x) \ola{\mathcal{L}}^*_t(g)(x) \rmd x = \int \ola{\mathcal{L}}_t f(x) g(x) \rmd x$ we obtain 
\begin{align*}
    \partial_t \kl \left( q_t \| \phi_t \right) 
    &= \int \ola{\mathcal{L}} \ln \left(f_t \right) (x) q_t(x) \rmd x
    - \int \ola{\mathcal{L}} f_t(x) \phi_t(x) \rmd x \eqsp.
\end{align*}
The infinitesimal generator $\ola{\mathcal{L}}$ satisfies the change of variables formula \citep[see][]{bakry2014analysis} so that 
\begin{equation*}
    \ola{\mathcal{L}}_t(\ln(f)) = \frac{1}{f} \ola{\mathcal{L}}_t f - \frac{1}{2 f^2} \ola{\Gamma}_t(f, f) \eqsp,
\end{equation*}
where $\ola{\Gamma}_t$ is the ``carré du champ'' operator associated with $\ola{\mathcal{L}}_t$  defined by $\ola{\Gamma}_t(f, f)(x) = \beta(t) |\nabla f(x)|^2$. We then obtain 
\begin{align}
    \partial_t \kl \left( q_t \| \phi_t \right) 
    &= \int \ola{\mathcal{L}}f_t(x) \frac{q_t(x)}{f_t(x)} \rmd x 
    - \int \frac{\beta(t)}{2} \frac{|\nabla f_t(x)|^2}{f_t^2(x)} q_t(x) \rmd x
    - \int \ola{\mathcal{L}} f_t(x) \phi_t(x) \rmd x  \notag \\
    &= - \frac{\beta(t)}{2} \int \frac{|\nabla f_t(x)|^2}{f_t(x)} \phi_t(x) \rmd x \eqsp.
    \label{eq:kl_contract:dt}
\end{align}

To obtain a control of the Kullback-Leibler divergence we need a logarithmic Sobolev inequality for the distribution of density $\phi_t = \varphi_{\sigma^2} Q_t$.
In this Gaussian case, if $\ola X_0 \sim \mathcal{N}(0, \sigma^2)$ then for all $t \in [0,T]$ the law of $\ola X_t$ is a centered Gaussian with covariance matrix $\ola{\Sigma}_t$ given by 
\begin{equation*}
    \ola{\Sigma}_t = 
    \sigma^2 \exp\left(\int_0^t -\frac{\bar \beta(s)}{\sigma^2} + 2 \bar \beta_s \bar A_s \rmd s\right)
    + \int_0^t \beta(s)\exp\left(\int_s^t - \frac{\bar \beta(u)}{\sigma^2} 
    + 2 \bar \beta(u) \bar A_u \rmd u\right) \rmd s \eqsp,
\end{equation*}
where we use the matrix exponential. As mentioned before, if $\lambda_{\max}(\Sigma_0) \le \sigma^2$, the eigenvalues of $A_s$, for $s \in [0,T]$, are negative. We can easily deduce that $\lambda_{\max}(\ola \Sigma_t) \le \sigma^2$.
We recall the logarithmic Sobolev inequality for a normal distribution \citep[see][Corollary~9]{chafai_entropies_2004}  
\begin{equation*}
    \kl(q_t \| \phi_t) \le \frac{1}{2} \int \frac{1}{f_t(x)} \nabla f_t(x)^\top \ola{\Sigma}_t \nabla f_t(x) \phi_t(x) \rmd x 
    \le \frac{\lambda_{\max}(\ola{\Sigma_t})}{2} 
    \int \frac{|\nabla f_t(x)|^2}{f_t(x)} \phi_t(x) \rmd x \eqsp.
\end{equation*} 
Plugging this into~\eqref{eq:kl_contract:dt} we get 
\begin{equation*}
    \partial_t \kl(q_t \| \phi_t) \le - \frac{\beta(t)}{\sigma^2} \kl(q_t \| \phi_t) \eqsp.
\end{equation*}
Therefore, recalling that $q_0 = p_T$ and $\phi_0 = \varphi_{\sigma^2}$ 
\begin{equation*}
    \kl \left(q_T \| \varphi_{\sigma^2} Q_T \right)
    \le \kl(p_T \| \varphi_{\sigma^2}) \exp\left(-\int_0^T \frac{\beta(s)}{\sigma^2} \rmd s\right) \eqsp.
\end{equation*}
We conclude using Lemma~\ref{lem:mix}.
\end{proof}

\subsection{Proof of Theorem~\ref{thm:wasserstein_bound}}
\label{ap:sec:wasserstein}

\paragraph{EI scheme. }
Using the fact that 
\begin{align*}
    \int_{t_k}^t \rme^{-\int^{t}_{s} \bar \beta(v)/(2 \sigma^2) \rmd v} \bar \beta (s) \rmd s = 2\sigma^2 \left( 1 -   \rme^{-\int_{t_k}^{t} \bar \beta (v)/(2 \sigma^2) \rmd v} \right)\eqsp,
\end{align*}
the Exponential Integrator scheme that we consider consists in the following discretization, recursively given with respect to the index $k$,
\begin{multline}
    \label{eq:backward_EI}
    \ola{X}_{t}  = \rme^{-\int_{t_k}^{t} \bar \beta (s)/(2 \sigma^2)  \rmd s} \bar X_{t_k} + 2 \sigma^2 \left( 1 -   \rme^{-\int_{t_k}^{t} \bar \beta (s)/(2 \sigma^2) \rmd s} \right) \nabla \log \tilde p_{T-t_k} \left( \bar X_{t_k} \right) \\ + \sigma \sqrt{ \left( 1 - \rme^{-  \int_{t_k}^{t} \bar \beta(s)/\sigma^2 \rm d s} \right)} Z_k \eqsp,
\end{multline}
where $Z_k$ are i.i.d. Gaussian random vectors $\mathcal{N}\left( 0 , I_d \right)$. In particular, we have that
\begin{multline}
    \label{eq:backward_EI-theta}
    \bar{X}^\theta_{t}  = \rme^{-\int_{t_k}^{t} \bar \beta (s)/(2 \sigma^2)  \rmd s} \bar X^\theta_{t_k} + 2 \sigma^2 \left( 1 -   \rme^{-\int_{t_k}^{t} \bar \beta (s)/(2 \sigma^2) \rmd s} \right) s_\theta\left(T-t_k , \bar X^\theta_{t_k} \right)  \\+ \sigma \sqrt{ \left( 1 - \rme^{- \int_{t_k}^{t} \bar \beta(s)/\sigma^2 \rm d s} \right)} Z_k \eqsp,
\end{multline}
and $\bar{X}^\theta_{0}\sim \mathcal{N}\left( 0 , \sigma^2 I_d \right)$.  
Note that
    \begin{align}
    \label{eq:decomp:w2}
        \Wc_2 \left( \pi_{\rm data} ,\pihat \right) \leq \Wc_2 \left( \pi_{\rm data} ,\pi_\infty Q_T \right) + 
         \Wc_2 \left(\pi_\infty Q_T ,\pi_\infty Q^{N,\param}_T \right),
    \end{align}
where 
$$
\Wc_2 \left( \pi_{\rm data} ,\pi_\infty Q_T \right)
=
\Wc_2\left( p_T Q_T, \pi_\infty Q_T \right),
$$
which corresponds to the discrepancy between the same process  \eqref{eq:backward_SDE} with two different initializations. The first term of \eqref{eq:decomp:w2} is upper bounded by Proposition~\ref{prop:contractivity}.

\begin{proposition}\label{prop:contractivity}
    Assume that $\Wc_2\left(\pi_{\rm data},\pi_{\infty}\right)^2<+\infty$. The marginal distribution at the end of the forward phase satisfies 
    \begin{align}
    \label{eq:contraction_forward}
\Wc_2\left(p_T,\pi_{\infty}\right)^2 \leq
        \Wc_2\left(\pi_{\rm data},\pi_{\infty}\right)^2\exp\left(- 
        \int_0^T \frac{\beta(t)}{\sigma^2}\rmd t\right)\eqsp.
    \end{align}
    Assume that H\ref{hyp:score_regularity}\eqref{hyp:score_lipschitz} holds. Then,
    \begin{align}
    \notag
        \Wc_2\left(\pi_{\rm data}, \pi_{\infty} Q_T\right)^2 &\leq 
        \Wc_2\left(p_T,\pi_{\infty}\right)^2
        \exp\left(- 
        \int_0^T \frac{\beta(t)}{\sigma^2}\left(1- 2 L_t\sigma^2\right)\rmd t\right) \\
        &\leq \Wc_2\left(\pi_{\rm data},\pi_{\infty}\right)^2
        \exp\left(- 
        \int_0^T \frac{\beta(t)}{\sigma^2}\left(2- 2 L_t\sigma^2\right)\rmd t\right)
        \eqsp.
            \label{eq:contractivity}
    \end{align}
    Moreover, under Assumption H\ref{hyp:score_regularity}\eqref{hyp:strong-log-concavity}, we have
    \begin{align}
    \notag
    \Wc_2\left(\pi_{\rm data}, \pi_{\infty} Q_T\right)^2 &\leq \Wc_2\left(p_T,\pi_{\infty}\right)^2
    \exp\left(- 
    \int_0^T \frac{\beta(t)}{\sigma^2}\left(1+ 2 C_t\sigma^2\right)\rmd t\right) \\
    &\leq 
    \Wc_2\left(\pi_{\rm data},\pi_{\infty}\right)^2
    \exp\left(- 
    \int_0^T \frac{\beta(t)}{\sigma^2}\left(2+ 2 C_t\sigma^2\right)\rmd t\right) 
    \eqsp.
    \label{eq:contractivity_monotone_operator}
    \end{align}

\end{proposition}
\begin{proof}[Proof of Proposition \ref{prop:contractivity}]
Let $x\in\R^d$ (resp. $y\in\R^d$) and denote by $\ora{X}^x$ (resp. $\ora{X}^y$) the solution of \eqref{eq:forward-SDE:beta}, with initial condition $\ora{X}^x_0 = x$ (resp. $\ora{X}^x_0 = y$). Applying the chain rule, we get
\begin{align*}
    \left\|\ora{X}^x_t - \ora{X}^y_t\right\|^2&=
    \left\|x-y\right\|^2 + 
    2\int_0^t
       - \frac{\bar\beta(s)}{2\sigma^2}
       \left\|\ora{X}^x_s - \ora{X}^y_s\right\|^2\rmd s
   \eqsp.
\end{align*}
Therefore, applying Grönwall's lemma, we obtain
\begin{align*}
    \E\left[ \sup_{t\in[0,T]} \left\|\ora{X}^x_t - \ora{X}^y_t\right\|^2 \right] \leq {\rm exp}\left(- 
   \int_0^T\frac{\bar\beta(t)}{\sigma^2}\rmd t\right)
   \left\|x-y\right\|^2 \eqsp.
\end{align*}
From this, we can show contraction \eqref{eq:contraction_forward} in $2$--Wasserstein distance by taking the infimum over all couplings.

Now, let $x\in\R^d$ (resp. $y\in\R^d$) and denote by $\ola{X}^x$ (resp. $\ola{X}^y$) the solution of \eqref{eq:backward_SDE}, with initial condition $\ola{X}^x_0 = x$ (resp. $\ola{X}^x_0 = y$). Applying the chain rule and using Cauchy-Schwarz inequality, we get
\begin{align*}
    \left\|\ola{X}^x_t - \ola{X}^y_t\right\|^2&=
    \left\|x-y\right\|^2 + 
    2\int_0^t
       - \frac{\bar\beta(s)}{2\sigma^2}
       \left\|\ola{X}^x_s - \ola{X}^y_s\right\|^2\rmd s
       \\
       &\qquad
       +
       2\int_0^t\bar\beta(s)\left(
       \nabla\log\tilde p_{T-s}\left( \ola{X}^x_s \right) - \nabla\log\tilde p_{T-s}\left( \ola{X}^y_s \right)
       \right)^\top
    \left(\ola{X}^x_s - \ola{X}^y_s\right)
   \rmd s
   \\
   &\leq\left\|x-y\right\|^2 - 
   \int_0^t \frac{\bar\beta(s)}{\sigma^2}\left(1- 2\bar L_{s}\sigma^2\right)\left\|\ola{X}^x_s - \ola{X}^y_s\right\|^2\rmd s\eqsp.
\end{align*}
Therefore, applying Grönwall's lemma, we obtain
\begin{align*}
    \E\left[ \sup_{t\in[0,T]} \left\|\ola{X}^x_t - \ola{X}^y_t\right\|^2 \right] \leq {\rm exp}\left(- 
   \int_0^T\frac{\bar\beta(t)}{\sigma^2}\left(1- 2\bar L_{t}\sigma^2\right)\rmd t\right)
   \left\|x-y\right\|^2 \eqsp.
\end{align*}
From this, we can show contraction \eqref{eq:contractivity} in $2$--Wasserstein distance by taking the infimum over all couplings.

To establish \eqref{eq:contractivity_monotone_operator} note that, under Assumption H\ref{hyp:score_regularity}\eqref{hyp:strong-log-concavity}, we have
\begin{align*}
    \left\|\ola{X}^x_t - \ola{X}^y_t\right\|^2&=
    \left\|x-y\right\|^2 + 
    2\int_0^t
       - \frac{\bar\beta(s)}{2\sigma^2}
       \left\|\ola{X}^x_s - \ola{X}^y_s\right\|^2\rmd s
       \\
       &\qquad
       +
       2\int_0^t\bar\beta(s)\left(
       \nabla\log\tilde p_{T-s}\left( \ola{X}^x_s \right) - \nabla\log\tilde p_{T-s}\left( \ola{X}^y_s \right)
       \right)^\top
    \left(\ola{X}^x_s - \ola{X}^y_s\right)
   \rmd s
   \\
   &\leq\left\|x-y\right\|^2 - 
   \int_0^t \frac{\bar\beta(s)}{\sigma^2}\left(1+ 2\bar C_{s}\sigma^2\right)\left\|\ola{X}^x_s - \ola{X}^y_s\right\|^2\rmd s\eqsp.
\end{align*}
Therefore, applying Grönwall's lemma, we obtain
\begin{align*}
    \E\left[ \sup_{t\in[0,T]} \left\|\ola{X}^x_t - \ola{X}^y_t\right\|^2 \right] \leq {\rm exp}\left(- 
   \int_0^T \frac{\bar\beta(t)}{\sigma^2}\left(1+2\bar C_{t}\sigma^2\right)\rmd t\right)
   \left\|x-y\right\|^2 \eqsp.
\end{align*}
From this, we can show contraction \eqref{eq:contractivity_monotone_operator} in the $2$--Wasserstein distance by taking the infimum over all couplings.
\end{proof}

Note that a similar assumption as Assumption H\ref{hyp:score_regularity}\eqref{hyp:strong-log-concavity} is used in \citet[Proposition 10,11,12]{debortoli2021}, in particular to bound the conditional  moments of $\ola{X}_0$ given $\ola{X}_t$ for $t>0$. However, in this paper the authors also require additional assumptions, in particular that the score of $\pi_{\mathrm{data}}$ has a linear growth.

\paragraph{Second term. } The second term of \eqref{eq:decomposition_W2} can be handled as follows
\begin{align*}
\Wc_2 \left(\pi_\infty Q_T ,\pi_\infty Q^{N,\param}_T \right) &\leq
    \left\| \ola{X}^\infty_{T} - \bar{X}^\theta_{T} \right\|_{L_2}.
\end{align*}
 To upper bound $\| \ola{X}^\infty_{T} - \bar{X}^\theta_{T} \|_{L_2}$, 
we aim at controlling $\| \ola{X}^\infty_{t_{k+1}} - \bar{X}^\theta_{t_{k+1}} \|_{L_2}$ by 
$\| \ola{X}^\infty_{t_k} - \bar{X}^\theta_{t_k} \|_{L_2}$ to resort subsequently to a telescopic sum.

\begin{proposition} 
\label{prop:wasserstein_tk}
Assume that  H\ref{hyp:score_regularity}, H\ref{hyp:sup_approx} and H\ref{hyp:lipschitz_score_time} hold.  Consider the regular discretization $\{t_k, 0\leq k \leq N\}$ of $[0,T]$ of constant step size $h$ such that for all $t_k$ with $0\leq k\leq N-1$,
$$
h < \frac{ 2\bar C_{t} }{\bar \beta(t_k) \left(\max_{t_k\leq s \leq t_{k+1}}\bar L_{s}\right)\bar L_{t}} \frac{\widetilde m_{t_{k+1}}}{\widetilde m_{t_{k}}}\eqsp,
$$
where $\widetilde m_t := \exp(-\int_0^t\bar \beta(s)\rmd s/(2\sigma^2))$, $ m_t := \exp(-\int_0^t \beta(s)\rmd s/(2\sigma^2))$. Then, 
     \begin{multline*}
         \left\|  \ola{X}_{T}^{\infty} - \bar{X}_{T}^{\theta} \right\|_{L_2} 
        \leq \varepsilon T\beta(T) +  M h T \beta(T)\left(1+2B \right) \\ 
        + \sum_{k=0}^{N-1}
        \left(\int_{t_k}^{t_{k+1}}\bar L_{t} 
            \frac{\widetilde m_{t}}{\widetilde m_{t_{k}}} \bar \beta (t) \rmd t \right)
         \left(
        \frac{\sqrt{2h\beta(T)}}{\sigma} + 
        m_T \int_{t_k}^{t_{k+1}} \left( \frac{1}{2\sigma^2}+2\bar{L}_{t}\right) \bar{\beta} (t) \rmd t
    \right) B\eqsp,
    \end{multline*}
    where $M$ is defined in H\ref{hyp:lipschitz_score_time} and $B:= (  \mathbb{E} [ \| X_0 \|^2 ] + \sigma^2 d )^{1/2}$.
\end{proposition}

\begin{proof}
    Using \eqref{eq:backward_EI-theta} and the triangular inequality, we have
    \begin{align}
        \left\|  \ola{X}_{t_{k+1}}^{\infty} - \bar{X}_{t_{k+1}}^{\theta} \right\|_{L_2} & \nonumber\\
        &\hspace{-2.9cm}= \left\|
            \frac{\widetilde m_{t_{k+1}}}{\widetilde m_{t_{k}}} \ola{X}_{t_{k}}^{\infty} - \frac{\widetilde m_{t_{k+1}}}{\widetilde m_{t_{k}}}
            \bar{X}_{t_{k}}^{\theta} + 
            \int_{t_k}^{t_{k+1}} 
            \frac{\widetilde m_{t}}{\widetilde m_{t_{k}}} \bar \beta (t) \left( 
                \nabla \log \tilde p_{T-t} \left( \ola{X}_t^{\infty} \right)- \tilde s_{\theta} \left( T- t_k ,\bar{X}_{t_{k}}^{\theta}  \right) \right) \rmd t \right\|_{L_2}\notag \\
                & \hspace{-2.9cm} \leq \left\|
            \frac{\widetilde m_{t_{k+1}}}{\widetilde m_{t_{k}}} \ola{X}_{t_{k}}^{\infty} - \frac{\widetilde m_{t_{k+1}}}{\widetilde m_{t_{k}}}
            \bar{X}_{t_{k}}^{\theta} + 
            \int_{t_k}^{t_{k+1}} 
            \frac{\widetilde m_{t}}{\widetilde m_{t_{k}}} \bar \beta (t) \left( 
                \nabla \log \tilde p_{T-t} \left( \ola{X}_{t_k}^{\infty} \right)- \nabla \log \tilde p_{T-t} \left(\bar{X}_{t_{k}}^{\theta}  \right) \right) \rmd t \right\|_{L_2} \notag \\
        & + 
            \left\|  \int_{t_k}^{t_{k+1}} 
            \frac{\widetilde m_{t}}{\widetilde m_{t_{k}}} \bar \beta (t) \left( 
                \nabla \log \tilde p_{T-t} \left( \ola{X}_{t}^{\infty} \right)- \nabla \log \tilde p_{T-t} \left(\ola{X}_{t_{k}}^{\infty}  \right) \right) \rmd t \right\|_{L_2}
            \label{eq:decomposition_W2}\\
        & + 
        \left\| \int_{t_k}^{t_{k+1}} 
            \frac{\widetilde m_{t}}{\widetilde m_{t_{k}}} \bar \beta (t) \left( 
                \nabla \log \tilde p_{T-t} \left( \bar{X}_{t_k}^{\theta} \right)- \tilde s_{\theta} \left( T- t_k ,\bar{X}_{t_{k}}^{\theta}  \right) \right) \rmd t \right\|_{L_2} \eqsp.\notag
    \end{align}
    Using the strong concavity and Lipschitz properties of the modified score function, we have that the first term of r.h.s. of \eqref{eq:decomposition_W2} can be bounded as follows
    \begin{align*}
        & \left\|
            \frac{\widetilde m_{t_{k+1}}}{\widetilde m_{t_{k}}} \ola{X}_{t_{k}}^{\infty} - \frac{\widetilde m_{t_{k+1}}}{\widetilde m_{t_{k}}}
            \bar{X}_{t_{k}}^{\theta} + 
            \int_{t_k}^{t_{k+1}} 
            \frac{\widetilde m_{t}}{\widetilde m_{t_{k}}} \bar \beta (t) \left( 
                \nabla \log \tilde p_{T-t} \left( \ola{X}_{t_k}^{\infty} \right)- \nabla \log \tilde p_{T-t} \left(\bar{X}_{t_{k}}^{\theta}  \right) \right) \rmd t \right\|^2 \\
        &= \frac{\widetilde m^2_{t_{k+1}}}{\widetilde m^2_{t_{k}}} \left\|  \ola{X}_{t_{k}}^{\infty} - \bar{X}_{t_{k}}^{\theta} \right\|^2 +
        \left\| \int_{t_k}^{t_{k+1}} 
            \frac{\widetilde m_{t}}{\widetilde m_{t_{k}}}
            \bar \beta (t) \left( 
                \nabla \log \tilde p_{T-t} \left( \ola{X}_{t_k}^{\infty} \right)- \nabla \log \tilde p_{T-t} \left(\bar{X}_{t_{k}}^{\theta}  \right) \right) \rmd t\right\|^2 \\
        & \quad \quad \quad +
        \frac{\widetilde m_{t_{k+1}}}{\widetilde m_{t_{k}}}
        2 \int_{t_k}^{t_{k+1}} \frac{\widetilde m_{t}}{\widetilde m_{t_{k}}} \bar \beta (t) \eqsp
        \left[ \ola{X}_{t_{k}}^{\infty} - \bar{X}_{t_{k}}^{\theta}\right]^\top \left[ \nabla \log \tilde p_{T-t} \left(  \ola{X}_{t_k}^{\infty} \right)- \nabla \log \tilde p_{T-t} \left(\bar{X}_{t_{k}}^{\theta} \right)\rmd t \right] 
        \\
        & \leq \left\|  \ola{X}_{t_{k}}^{\infty} - \bar{X}_{t_{k}}^{\theta} \right\|^2 \left( \frac{\widetilde m^2_{t_{k+1}}}{\widetilde m^2_{t_{k}}} + 
        \left(  \int_{t_k}^{t_{k+1}} \bar L_{t}
        \frac{\widetilde m_{t}}{\widetilde m_{t_{k}}}
        \bar \beta (t)  \rmd t \right)^2
        - 2 \frac{\widetilde m_{t_{k+1}}}{\widetilde m_{t_{k}}}\int_{t_k}^{t_{k+1}} \bar C_{t}
        \frac{\widetilde m_{t}}{\widetilde m_{t_{k}}}
        \bar \beta (t)  \rmd t \right) \eqsp.
    \end{align*}
    Using the Lipschitz property of the modified score and Proposition \ref{prop:X_inf_sup_t}, the second term of the r.h.s. of \eqref{eq:decomposition_W2} can be controlled as follows
    \begin{align*}
        &\left\|  \int_{t_k}^{t_{k+1}} 
            \frac{\widetilde m_{t}}{\widetilde m_{t_{k}}} \bar \beta (t) \left( 
                \nabla \log \tilde p_{T-t} \left( \ola{X}_{t}^{\infty} \right)- \nabla \log \tilde p_{T-t} \left(\ola{X}_{t_{k}}^{\infty}  \right) \right) \rmd t \right\|_{L_2}\\
        &\hspace{.8cm}\leq \left(\int_{t_k}^{t_{k+1}}L_{T-t} 
            \frac{\widetilde m_{t}}{\widetilde m_{t_{k}}} \bar \beta (t) \rmd t \right) \sup_{t_k\leq t \leq t_{k+1}} \left\| \ola{X}_{t}^{\infty} - \ola{X}_{t_{k}}^{\infty}\right\|_{L_2}\\
        &\hspace{.8cm}\leq \left(\int_{t_k}^{t_{k+1}}\bar L_{t} 
            \frac{\widetilde m_{t}}{\widetilde m_{t_{k}}} \bar \beta (t) \rmd t \right)
         2\left(
            \frac{1}{\sigma}\sqrt{h\beta(T)} + {\rm exp}\left(- 
            \int_0^{t_k}  \frac{\bar\beta(s)}{ \sigma^2}\left(1+ \bar C_{s} \sigma^2\right)\rmd s\right)
        \right) B\eqsp.
    \end{align*}
Using Assumption H\ref{hyp:sup_approx}, we can control the third term of the r.h.s. of \eqref{eq:decomposition_W2} as follows
    \begin{align*}
    & \left\| \int_{t_k}^{t_{k+1}} 
            \frac{\widetilde m_{t}}{\widetilde m_{t_{k}}} \bar \beta (t) \left( 
                \nabla \log \tilde p_{T-t} \left( \bar{X}_{t_k}^{\theta} \right)- \tilde s_{\theta} \left( T- t_k ,\bar{X}_{t_{k}}^{\theta}  \right) \right) \rmd t \right\|_{L_2}\\
    & \leq \left\| \int_{t_k}^{t_{k+1}} 
            \frac{\widetilde m_{t}}{\widetilde m_{t_{k}}} \bar \beta (t) \left( 
                \nabla \log \tilde p_{T-t_k} \left( \bar{X}_{t_k}^{\theta} \right)- \tilde s_{\theta} \left( T- t_k ,\bar{X}_{t_{k}}^{\theta}  \right) \right) \rmd t \right\|_{L_2}\\
    & \quad +\left\| \int_{t_k}^{t_{k+1}} 
            \frac{\widetilde m_{t}}{\widetilde m_{t_{k}}} \bar \beta (t) \left( 
                \nabla \log \tilde p_{T-t} \left( \bar{X}_{t_k}^{\theta} \right) -  \nabla \log \tilde p_{T-t_k} \left( \bar{X}_{t_k}^{\theta} \right) \right) \rmd t \right\|_{L_2}\\
    &\leq \varepsilon \int_{t_k}^{t_{k+1}} 
            \frac{\widetilde m_{t}}{\widetilde m_{t_{k}}} \bar \beta (t) \rmd t + 
    \int_{t_k}^{t_{k+1}} 
            \frac{\widetilde m_{t}}{\widetilde m_{t_{k}}} \bar \beta (t) 
            \left\|
                \nabla \log \tilde p_{T-t} \left( \bar{X}_{t_k}^{\theta} \right) - \nabla \log \tilde p_{T-t_k} \left( \bar{X}_{t_k}^{\theta} \right)
            \right\|
            \rmd t \\
         &\leq \varepsilon \int_{t_k}^{t_{k+1}} 
                \frac{\widetilde m_{t}}{\widetilde m_{t_{k}}} \bar \beta (t)
                \rmd t + h M \left(1+\left\| \bar{X}_{t_k}^{\theta} \right\|_{L_2}\right)
         \int_{t_k}^{t_{k+1}} 
                \frac{\widetilde m_{t}}{\widetilde m_{t_{k}}} \bar \beta (t)
                \rmd t\eqsp .        
    \end{align*}
    Note that $
        \ora X_t$ has the same law as $m_t X_0 + \sigma \sqrt{( 1 - m_t^2)} G$, with $G$ a standard Gaussian random variable independent of $X_0$. We have that $\ola{X}^\infty_0\sim \mathcal{N}(0,\sigma^2 I_d)$. Define $(\ola{X}_t)_{t\in[0,T]}$ satisfying \eqref{eq:backward_SDE} but initialized at
    \begin{align*}
        \ola{X}_0 & = m_T X_0 + \sqrt{ \left( 1 - m_T^2\right)} \ola{X}^\infty_0 \eqsp,
    \end{align*}
    with $X_0\sim \pi_{\rm data}$.
    Employing Proposition \ref{prop:true_backward_bound} and \eqref{eq:X_infty_X_ola:X_inf_sup_t}, we obtain
    \begin{align*}
        \left\| \bar{X}_{t_k}^{\theta} \right\|_{L_2} \leq
        \left\| \bar{X}_{t_k}^{\theta} - \ola{X}_{t_k}^{\infty} \right\|_{L_2} +
        \left\| \ola{X}_{t_k}^{\infty} - \ola{X}_{t_k} \right\|_{L_2} +
        \left\| \ola{X}_{t_k} \right\|_{L_2} \leq  \left\| \bar{X}_{t_k}^{\theta} - \ola{X}_{t_k}^{\infty} \right\|_{L_2} + 2B \eqsp.
    \end{align*}
    Therefore, combining the previous bounds, together with \eqref{eq:decomposition_W2}, we obtain
    \begin{align*}
         & \left\|  \ola{X}_{t_{k+1}}^{\infty} - \bar{X}_{t_{k+1}}^{\theta} \right\|_{L_2} \\
         & \leq \left\|  \ola{X}_{t_{k}}^{\infty} - \bar{X}_{t_{k}}^{\theta} \right\| \left( \frac{\widetilde m^2_{t_{k+1}}}{\widetilde m^2_{t_{k}}} + 
        \left(  \int_{t_k}^{t_{k+1}} \bar L_{t}
        \frac{\widetilde m_{t}}{\widetilde m_{t_{k}}}
        \bar \beta (t)  \rmd t \right)^2
        - 2 \frac{\widetilde m_{t_{k+1}}}{\widetilde m_{t_{k}}}\int_{t_k}^{t_{k+1}} \bar C_{t}
        \frac{\widetilde m_{t}}{\widetilde m_{t_{k}}}
        \bar \beta (t)  \rmd t \right)^{1/2}
        \\
        & \quad \quad \quad + \left(\int_{t_k}^{t_{k+1}}\bar L_{t} 
            \frac{\widetilde m_{t}}{\widetilde m_{t_{k}}} \bar \beta (t) \rmd t \right)
         \left(
        \frac{1}{\sigma}\sqrt{2h\beta(T)} + 
        m_T \int_{t_k}^{t_{k+1}} \left( \frac{1}{2\sigma^2}+2\bar L_{t}\right) \bar{\beta} (t) \rmd t
    \right) B
        \\
        & \quad \quad \quad + \varepsilon \int_{t_k}^{t_{k+1}} 
                \frac{\widetilde m_{t}}{\widetilde m_{t_{k}}} \bar \beta (t)
                \rmd t + h M \left(1+\left\| \bar{X}_{t_k}^{\theta} - \ola{X}_{t_k}^{\infty} \right\|_{L_2} + 2B \right)
         \int_{t_k}^{t_{k+1}} 
                \frac{\widetilde m_{t}}{\widetilde m_{t_{k}}} \bar \beta (t)
                \rmd t \eqsp.
    \end{align*}
    By the assumption on $h$ and Proposition \ref{prop:step_size_choice}, 
    $$
    0< 1 + 
        \frac{\widetilde m^2_{t_{k}}}{\widetilde m^2_{t_{k+1}}}
        \left(  \int_{t_k}^{t_{k+1}} \bar L_{t}
        \frac{\widetilde m_{t}}{\widetilde m_{t_{k}}}
        \bar \beta (t)  \rmd t \right)^2
        - 2
        \frac{\widetilde m_{t_{k}}}{\widetilde m_{t_{k+1}}}
        \int_{t_k}^{t_{k+1}} \bar C_{t}
        \frac{\widetilde m_{t}}{\widetilde m_{t_{k}}}
        \bar \beta (t)  \rmd t
        <1\eqsp,
    $$
    and, using that $\sqrt{1-x}\leq 1-x/2$ for $x\in[0,1]$, we conclude that 
    \begin{align*}
         & \left\|  \ola{X}_{t_{k+1}}^{\infty} - \bar{X}_{t_{k+1}}^{\theta} \right\|_{L_2} \notag\\
         \notag
         & \leq \left\|  \ola{X}_{t_{k}}^{\infty} - \bar{X}_{t_{k}}^{\theta} \right\|  \frac{\widetilde m^2_{t_{k+1}}}{\widetilde m^2_{t_{k}}} \\
         & \qquad   \times  \left(
            1 + 
            \frac{1}{2}\frac{\widetilde m^2_{t_{k}}}{\widetilde m^2_{t_{k+1}}}
            \left(  \int_{t_k}^{t_{k+1}} \bar L_{t}
            \frac{\widetilde m_{t}}{\widetilde m_{t_{k}}}
            \bar \beta (t)  \rmd t \right)^2
            - \frac{\widetilde m_{t_{k}}}{\widetilde m_{t_{k+1}}}
            \int_{t_k}^{t_{k+1}} \bar C_{t}
            \frac{\widetilde m_{t}}{\widetilde m_{t_{k}}}
            \bar \beta (t) \rmd t + \right.\\
            &\left. \hspace{9cm}\frac{\widetilde m^2_{t_{k}}}{\widetilde m^2_{t_{k+1}}} M h \int_{t_k}^{t_{k+1}} 
                \frac{\widetilde m_{t}}{\widetilde m_{t_{k}}} \bar \beta (t)
                \rmd t 
        \right) \notag
        \\
        & \qquad   + \left(\int_{t_k}^{t_{k+1}}\bar L_{t} 
            \frac{\widetilde m_{t}}{\widetilde m_{t_{k}}} \bar \beta (t) \rmd t \right)
         \left(
        \frac{1}{\sigma}\sqrt{2h\beta(T)} + 
        m_T \int_{t_k}^{t_{k+1}} \left( \frac{1}{2\sigma^2}+2\bar L_{t}\right) \bar{\beta} (t) \rmd t
    \right) B 
        \\
        & \qquad   +\varepsilon \int_{t_k}^{t_{k+1}} 
                \frac{\widetilde m_{t}}{\widetilde m_{t_{k}}} \bar \beta (t)
                \rmd t + h M \left(1+ 2B \right)
         \int_{t_k}^{t_{k+1}} 
                \frac{\widetilde m_{t}}{\widetilde m_{t_{k}}} \bar \beta (t)
                \rmd t \eqsp.\notag
    \end{align*}
    Define
    \begin{align*}
    \delta_k &:=
        \frac{\widetilde m^2_{t_{k+1}}}{\widetilde m^2_{t_{k}}} \left(
            1 + 
            \frac{1}{2}\frac{\widetilde m^2_{t_{k}}}{\widetilde m^2_{t_{k+1}}}
            \left(  \int_{t_k}^{t_{k+1}} \bar L_{t}
            \frac{\widetilde m_{t}}{\widetilde m_{t_{k}}}
            \bar \beta (t)  \rmd t \right)^2
            - \frac{\widetilde m_{t_{k}}}{\widetilde m_{t_{k+1}}}
            \int_{t_k}^{t_{k+1}} \bar C_{t}
            \frac{\widetilde m_{t}}{\widetilde m_{t_{k}}}
            \bar \beta (t) \rmd t \right.\\
            &\left.\hspace{9cm}+ \frac{\widetilde m^2_{t_{k}}}{\widetilde m^2_{t_{k+1}}} M h \int_{t_k}^{t_{k+1}} 
                \frac{\widetilde m_{t}}{\widetilde m_{t_{k}}} \bar \beta (t)
                \rmd t 
        \right) \\
        &\leq \left(
            1 + 
            \frac{1}{2}\frac{\widetilde m^2_{t_{k}}}{\widetilde m^2_{t_{k+1}}}
            \left(  \int_{t_k}^{t_{k+1}} \bar L_{t}
            \frac{\widetilde m_{t}}{\widetilde m_{t_{k}}}
            \bar \beta (t)  \rmd t \right)^2
            - \frac{\widetilde m_{t_{k}}}{\widetilde m_{t_{k+1}}}
            \int_{t_k}^{t_{k+1}} \bar C_{t}
            \frac{\widetilde m_{t}}{\widetilde m_{t_{k}}}
            \bar \beta (t) \rmd t \right.\\
            &\left. \hspace{9cm}+ \frac{\widetilde m^2_{t_{k}}}{\widetilde m^2_{t_{k+1}}} M h \int_{t_k}^{t_{k+1}} 
                \frac{\widetilde m_{t}}{\widetilde m_{t_{k}}} \bar \beta (t)
                \rmd t 
        \right)\eqsp.
    \end{align*}
    By Proposition \ref{prop:step_size_choice},  $\delta_k \leq 1$ for any $0\leq k\leq N-1$ , which yields
    \begin{align*}
        \left\|  \ola{X}_{T}^{\infty} - \bar{X}_{T}^{\theta} \right\|_{L_2} 
        &\leq \prod_{k=0}^{N-1} \delta_k  \left\|  \ola{X}_{0}^{\infty} - \bar{X}_{0}^{\theta} \right\|_{L_2} +\left(\varepsilon h\beta(T) +  M h^2\beta(T)\left(1+2B \right)\right)\sum_{k=0}^{N-1}\prod_{\ell=k}^{N-1} \delta_\ell\\
        &\hspace{-2cm}+\sum_{k=0}^{N-1}
        \left(\int_{t_k}^{t_{k+1}}\bar L_{t} 
            \frac{\widetilde m_{t}}{\widetilde m_{t_{k}}} \bar \beta (t) \rmd t \right)
         \left(
        \frac{1}{\sigma}\sqrt{2h\beta(T)} + 
        m_T \int_{t_k}^{t_{k+1}} \left( \frac{1}{2\sigma^2}+2\bar L_{t}\right) \bar{\beta} (t) \rmd t
    \right) B \prod_{\ell=k}^{N-1} \delta_\ell
        \\
        &\leq \varepsilon T\beta(T) +  M h T \beta(T)\left(1+2B \right)
        \\
        & \hspace{-2cm}+ \sum_{k=0}^{N-1}
        \left(\int_{t_k}^{t_{k+1}}\bar L_{t} 
            \frac{\widetilde m_{t}}{\widetilde m_{t_{k}}} \bar \beta (t) \rmd t \right)
         \left(
        \frac{1}{\sigma}\sqrt{2h\beta(T)} + 
        m_T \int_{t_k}^{t_{k+1}} \left( \frac{1}{2\sigma^2}+2\bar L_{t}\right) \bar{\beta} (t) \rmd t
    \right) B\eqsp.
    \end{align*}

    \end{proof}

    \paragraph{Final bound. } Finally, combining the results of Proposition \ref{prop:contractivity} and Proposition \ref{prop:wasserstein_tk}, we conclude that
    \begin{align*}
        \Wc_2 \left( \pi_{\rm data} ,\pihat \right)&\leq 
        \Wc_2\left(\pi_{\rm data},\pi_{\infty}\right)
    \exp\left(- 
    \int_0^T \frac{\beta(t)}{\sigma^2}\left(1+  C_t\sigma^2\right)\rmd t\right) \\
    &\hspace{-2cm} + \sum_{k=0}^{N-1}
        \left(\int_{t_k}^{t_{k+1}}\bar L_{t} 
            \frac{\widetilde m_{t}}{\widetilde m_{t_{k}}} \bar \beta (t) \rmd t \right)
         \left(
        \frac{1}{\sigma}\sqrt{2h\beta(T)} + 
        m_T \int_{t_k}^{t_{k+1}} \left( \frac{1}{2\sigma^2}+2\bar L_{t}\right) \bar{\beta} (t) \rmd t
    \right) B
        \\
        &\hspace{-2cm}  + \varepsilon T\beta(T) +  M h T \beta(T)\left(1+2B \right)\eqsp.
    \end{align*}

\subsection{Technical results for Wasserstein upper bound}
    \begin{proposition}
    \label{prop:step_size_choice}
        Assume that  H\ref{hyp:score_regularity} and H\ref{hyp:lipschitz_score_time} hold. Consider the regular discretization $\{t_k, 0\leq k \leq N\}$ of $[0,T]$ of constant step size $h$.
        Assume that $h>0$ is such that for all $t_k$ with $0\leq k\leq N-1$,
        \begin{align}
    \label{eq:condition_step_size1}
            h &< \frac{ 2\bar C_{t} }{\bar{\beta}(t_k) \left(\max_{t_k\leq s \leq t_{k+1}}\bar L_{s}\right)\bar L_{t}} \frac{\widetilde m_{t_{k+1}}}{\widetilde m_{t_{k}}} \eqsp ,  
        \end{align}
        where $\widetilde m_t := \exp(-\int_0^t\bar \beta(s)\rmd s/(2\sigma^2))$, $ m_t := \exp(-\int_0^t \beta(s)\rmd s/(2\sigma^2))$. Then, for all $0\leq k\leq N-1$,
        \begin{align*}
            0< 1 + 
        \frac{\widetilde m^2_{t_{k}}}{\widetilde m^2_{t_{k+1}}}
        \left(  \int_{t_k}^{t_{k+1}} \bar{L}_{t}
        \frac{\widetilde m_{t}}{\widetilde m_{t_{k}}}
        \bar \beta (t)  \rmd t \right)^2
        - 2
        \frac{\widetilde m_{t_{k}}}{\widetilde m_{t_{k+1}}}
        \int_{t_k}^{t_{k+1}} \bar{C}_{t}
        \frac{\widetilde m_{t}}{\widetilde m_{t_{k}}}
        \bar \beta (t)  \rmd t
            <1\eqsp.
        \end{align*}
        In addition, if
        \begin{align}
        \label{eq:condition_step_size2}
            h &< \frac{ 2\bar{C}_{t} }{M + \bar{\beta}(t_k) \left(\max_{t_k\leq s \leq t_{k+1}}\bar{L}_{s}\right)\bar{L}_{t}} \frac{\widetilde m_{t_{k+1}}}{\widetilde m_{t_{k}}} \eqsp ,   
        \end{align}
        then, for all $0\leq k\leq N-1$,
        \begin{multline*}
            0< 
            1 + 
            \frac{1}{2}\frac{\widetilde m^2_{t_{k}}}{\widetilde m^2_{t_{k+1}}}
            \left(  \int_{t_k}^{t_{k+1}} \bar{L}_{t}
            \frac{\widetilde m_{t}}{\widetilde m_{t_{k}}}
            \bar \beta (t)  \rmd t \right)^2
            - \frac{\widetilde m_{t_{k}}}{\widetilde m_{t_{k+1}}}
            \int_{t_k}^{t_{k+1}} \bar{C}_{t}
            \frac{\widetilde m_{t}}{\widetilde m_{t_{k}}}
            \bar \beta (t) \rmd t \\+ \frac{\widetilde m^2_{t_{k}}}{\widetilde m^2_{t_{k+1}}} M h \int_{t_k}^{t_{k+1}} 
                \frac{\widetilde m_{t}}{\widetilde m_{t_{k}}} \bar \beta (t)
                \rmd t 
        <1 \eqsp . 
        \end{multline*}
    \end{proposition}
    \begin{proof}
    Denote $\epsilon_1$ and $\epsilon_2$ the following quantities
    \begin{align}
    \label{eq:step_size_choice:eps1}
        \epsilon_1&= 1 + 
            \frac{\widetilde m^2_{t_{k}}}{\widetilde m^2_{t_{k+1}}}
            \left(  \int_{t_k}^{t_{k+1}} \bar{L}_{t}
            \frac{\widetilde m_{t}}{\widetilde m_{t_{k}}}
            \bar \beta (t)  \rmd t \right)^2
            - 2
            \frac{\widetilde m_{t_{k}}}{\widetilde m_{t_{k+1}}}
            \int_{t_k}^{t_{k+1}} \bar{C}_{t}
            \frac{\widetilde m_{t}}{\widetilde m_{t_{k}}}
            \bar \beta (t)  \rmd t\eqsp,\\
        \epsilon_2&= 1 + 
            \frac{\widetilde m^2_{t_{k}}}{\widetilde m^2_{t_{k+1}}}
            \left(  \int_{t_k}^{t_{k+1}} \bar{L}_{t}
            \frac{\widetilde m_{t}}{\widetilde m_{t_{k}}}
            \bar \beta (t)  \rmd t \right)^2
            - 2
            \frac{\widetilde m_{t_{k}}}{\widetilde m_{t_{k+1}}}
            \int_{t_k}^{t_{k+1}} \bar{C}_{t}
            \frac{\widetilde m_{t}}{\widetilde m_{t_{k}}}
            \bar \beta (t)  \rmd t\nonumber\\
            &\hspace{7cm}
            +
            \frac{\widetilde m^2_{t_{k}}}{\widetilde m^2_{t_{k+1}}} M h \int_{t_k}^{t_{k+1}} 
                \frac{\widetilde m_{t}}{\widetilde m_{t_{k}}} \bar \beta (t)
                \rmd t \eqsp.
                \label{eq:step_size_choice:eps2}
    \end{align}

        First, we prove that $\epsilon_1$ is positive. Completing the square, we obtain
        \begin{align*}
            \epsilon_1&
            =\left(1 - 
            \frac{\widetilde m_{t_{k}}}{\widetilde m_{t_{k+1}}}
             \int_{t_k}^{t_{k+1}} \bar{L}_{t}
            \frac{\widetilde m_{t}}{\widetilde m_{t_{k}}}
            \bar \beta (t)  \rmd t \right)^2 +
            2
            \frac{\widetilde m_{t_{k}}}{\widetilde m_{t_{k+1}}}
            \int_{t_k}^{t_{k+1}} \bar{L}_{t}
            \frac{\widetilde m_{t}}{\widetilde m_{t_{k}}}
            \bar \beta (t)  \rmd t\\
            &\hspace{7cm}
            - 2
            \frac{\widetilde m_{t_{k}}}{\widetilde m_{t_{k+1}}}
            \int_{t_k}^{t_{k+1}} \bar{C}_{t}
            \frac{\widetilde m_{t}}{\widetilde m_{t_{k}}}
            \bar \beta (t)  \rmd t\\
            &\eqsp=\left(1 - 
            \frac{\widetilde m_{t_{k}}}{\widetilde m_{t_{k+1}}}
             \int_{t_k}^{t_{k+1}} \bar{L}_{t}
            \frac{\widetilde m_{t}}{\widetilde m_{t_{k}}}
            \bar \beta (t)  \rmd t \right)^2 +
            2
            \frac{\widetilde m_{t_{k}}}{\widetilde m_{t_{k+1}}}
                \int_{t_k}^{t_{k+1}} \left(\bar{L}_{t}
                - \bar{C}_{t}\right)
            \frac{\widetilde m_{t}}{\widetilde m_{t_{k}}}
            \bar \beta (t)  \rmd t.
        \end{align*}
        The first term if the r.h.s. of the previous equality is a square, therefore always positive. The second term is always positive as well, as $\bar{L}_{t}\geq  \bar{C}_{t}$ for any $t$, as the Lipschitz constant and the log-concavity coefficient of the score function respectively. Moreover, the previous is always strictly positive as
        \begin{align*}
            \frac{\widetilde m_{t_{k}}}{\widetilde m_{t_{k+1}}}
         \int_{t_k}^{t_{k+1}} \bar{L}_{t}
        \frac{\widetilde m_{t}}{\widetilde m_{t_{k}}}
        \bar \beta (t)  \rmd t >0\eqsp.
        \end{align*}

        Secondly, proving that the previous quantity is smaller than $1$ is equivalent to show that
        \begin{align*}
            \frac{\widetilde m^2_{t_{k}}}{\widetilde m^2_{t_{k+1}}}
            \left(  \int_{t_k}^{t_{k+1}} \bar{L}_{t}
            \frac{\widetilde m_{t}}{\widetilde m_{t_{k}}}
            \bar \beta (t)  \rmd t \right)^2
            - 2
            \frac{\widetilde m_{t_{k}}}{\widetilde m_{t_{k+1}}}
            \int_{t_k}^{t_{k+1}} \bar{C}_{t}
            \frac{\widetilde m_{t}}{\widetilde m_{t_{k}}}
            \bar \beta (t)  \rmd t <0\eqsp.
        \end{align*}
        As $\bar\beta(t)$ is a decreasing function, we obtain the following bound
        \begin{align*}
            &\frac{\widetilde m^2_{t_{k}}}{\widetilde m^2_{t_{k+1}}}
            \left(  \int_{t_k}^{t_{k+1}} \bar{L}_{t}
            \frac{\widetilde m_{t}}{\widetilde m_{t_{k}}}
            \bar \beta (t)  \rmd t \right)^2
            - 2
            \frac{\widetilde m_{t_{k}}}{\widetilde m_{t_{k+1}}}
            \int_{t_k}^{t_{k+1}} \bar{C}_{t}
            \frac{\widetilde m_{t}}{\widetilde m_{t_{k}}}
            \bar \beta (t)  \rmd t
            \\
            &\leq\left(\frac{\widetilde m_{t_{k}}}{\widetilde m_{t_{k+1}}} \max_{t_k\leq s \leq t_{k+1}}\bar{L}_{s} \bar\beta(t_k) h\right) \frac{\widetilde m_{t_{k}}}{\widetilde m_{t_{k+1}}}
            \int_{t_k}^{t_{k+1}} \bar{L}_{t}
            \frac{\widetilde m_{t}}{\widetilde m_{t_{k}}}
            \bar \beta (t)  \rmd t - 2
            \frac{\widetilde m_{t_{k}}}{\widetilde m_{t_{k+1}}}
            \int_{t_k}^{t_{k+1}} \bar{C}_{t}
            \frac{\widetilde m_{t}}{\widetilde m_{t_{k}}}
            \bar \beta (t)  \rmd t\\
            \\
            &= \frac{\widetilde m_{t_{k}}}{\widetilde m_{t_{k+1}}}
            \int_{t_k}^{t_{k+1}} \left(\left(\frac{\widetilde m_{t_{k}}}{\widetilde m_{t_{k+1}}} \max_{t_k\leq s \leq t_{k+1}}\bar{L}_{s} \bar\beta(t_k) h\right)\bar{L}_{t}- 2 \bar{C}_{t}\right)
            \frac{\widetilde m_{t}}{\widetilde m_{t_{k}}}
            \bar \beta (t)  \rmd t\eqsp.
        \end{align*}
        This means that, if we have
        \begin{align*}
            \frac{\widetilde m_{t_{k}}}{\widetilde m_{t_{k+1}}}\left( \max_{t_k\leq s \leq t_{k+1}}\bar{L}_{s} \right)\bar\beta(t_k) h\bar{L}_{t}- 2 \bar{C}_{t} < 0
        \end{align*}
        for $t_k\leq t\leq t_{k+1}$, we have $\epsilon_1<1$.
        Isolating $h$ in the previous inequality, we obtain that it is equivalent to the condition \eqref{eq:condition_step_size1}.
    
        Now we focus on $\epsilon_2$. This quantity is clearly positive as the $\epsilon_2\geq \epsilon_1$. Moreover, following the same lines as to prove that $\epsilon_1<1$, we have
        \begin{align*}
            &\epsilon_2 - 1
            \leq \frac{\widetilde m_{t_{k}}}{\widetilde m_{t_{k+1}}}
            \int_{t_k}^{t_{k+1}} \left(\frac{\widetilde m_{t_{k}}}{\widetilde m_{t_{k+1}}}\left( \max_{t_k\leq s \leq t_{k+1}}\bar{L}_{s} \right)\bar\beta(t_k) h\bar{L}_{t} + \frac{\widetilde m_{t_{k}}}{\widetilde m_{t_{k+1}}}M h- 2 \bar{C}_{t}\right)
            \frac{\widetilde m_{t}}{\widetilde m_{t_{k}}}
            \bar \beta (t)  \rmd t\eqsp.
        \end{align*}
        This means that, if we have
        \begin{align*}
            \frac{\widetilde m_{t_{k}}}{\widetilde m_{t_{k+1}}}\left( \max_{t_k\leq s \leq t_{k+1}}\bar{L}_{s} \right)\bar\beta(t_k) h\bar{L}_{t} + \frac{\widetilde m_{t_{k}}}{\widetilde m_{t_{k+1}}}M h- 2 \bar{C}_{t} < 0
        \end{align*}
        for $t_k\leq t\leq t_{k+1}$, we have $\epsilon_2<1$.
        Isolating $h$ in the previous inequality, we obtain that it is equivalent to the condition \eqref{eq:condition_step_size2}.
    
    \end{proof}

\begin{proposition} \label{prop:true_backward_bound}
    Assume that H\ref{hyp:fisher_info} holds. For all $t \geq 0$,
    \begin{align*}
        \sup_{0 \leq t \leq T} \left\| \ola X_t \right\|_{L_2} \leq \sup_{0 \leq t \leq T} \left( m_t^2   \mathbb{E} \left[ \left\| X_0 \right\|^2 \right]  +(1 - m_t^2) \sigma^2 d \right)^{1/2} \leq  \left(  \mathbb{E} \left[ \left\| X_0 \right\|^2 \right] + \sigma^2 d \right)^{1/2}\eqsp,
    \end{align*}
    where $m_t=\exp(-\int_{0}^t \beta(s) \rmd s /2\sigma^2)$.
\end{proposition}

\begin{proof}
      Recall the following equality in law
     \begin{align*}
        \ora X_t & = m_t X_0 + \sigma \sqrt{ \left( 1 - m_t^2\right)} G \eqsp.
    \end{align*}
    with $X_0 \sim \pi_{\rm data}$ and $G \sim \mathcal{N}\left( 0 , I_d \right)$.

    Therefore, for any $t \in [0,T]$
    \begin{align*}    
     \mathbb{E} \left[  \left\| \ola X_{T-t} \right\|^2 \right] = \mathbb{E} \left[  \left\| \ora X_t \right\|^2 \right] & \leq  m_t^2  \mathbb{E} \left[ \left\| X_0 \right\|^2 \right]  +  \sigma^2  \left( 1 - m_t^2 \right) \mathbb{E} \left[ \left\| G \right\|^2 \right] \\
     &  \leq  m_t^2  \mathbb{E} \left[ \left\| X_0 \right\|^2 \right]  +  \sigma^2  \left( 1 - m_t^2 \right)d \eqsp.
     \end{align*}

\end{proof}

\begin{proposition}
    \label{prop:X_inf_sup_t}
    Assume that H\ref{hyp:fisher_info} holds. For all $t_k \leq t \leq t_{k +1}$, 
    \begin{align}
    \label{eq:X_inf_sup_t}
     \sup_{t_k \leq t \leq t_{k+1}} \left\|   \ola{X}_{t}^{\infty} -  \ola{X}_{t_k}^{\infty} \right\|_{L_2}
    &\leq
     \left(
        \frac{1}{\sigma}\sqrt{2h\beta(T)} + 
        m_T \int_{t_k}^{t_{k+1}} \left( \frac{1}{2\sigma^2}+2\bar{L}_{t}\right) \bar{\beta} (t) \rmd t
    \right) B\eqsp,
    \\
        \sup_{0 \leq t \leq T}\left\| \ola{X}_{t}^{\infty} - \ola X_{t}  \right\|_{L_2}
        &\leq \left(\left[ \left\| X_0 \right\|^2\right]+ \sigma^2 d
        \right)^{1/2}{\rm exp}\left(- 
        \int_0^{T} \frac{\bar\beta(s)}{ 2\sigma^2}\rmd s\right)
        \eqsp,
        \label{eq:X_infty_X_ola:X_inf_sup_t}
    \end{align}
where $m_t=\exp(-\int_{0}^t \beta(s) \rmd s /2\sigma^2)$ and $B=(  \mathbb{E} [ \| X_0 \|^2 ] + \sigma^2 d )^{1/2}$.
\end{proposition}

\begin{proof}
    Note that $\ora X_t$ has the same distribution as $m_t X_0 + \sigma \sqrt{ \left( 1 - m_t^2\right)} G$ where $G \sim \mathcal{N}(0,I_d)$ is independent of $X_0$.
    We have that $\ola{X}^\infty_0=G\sim \mathcal{N}(0,\sigma^2 I_d)$. Define $(\ola{X}_t)_{t\in[0,T]}$ satisfying \eqref{eq:backward_SDE} but initialized at
    \begin{align}
    \label{eq:initialization:X_inf_sup_t}
        \ola{X}_0 & = m_T Y + \sqrt{ \left( 1 - m_T^2\right)} G \eqsp,
    \end{align}
    with $Y\sim \pi_{\rm data}$ independent of $G$ ($G$ being shared by $\ola{X}_0$ and $\ola{X}^\infty_0$).

        On the one hand, following the same proof as in Proposition \ref{prop:contractivity}, we have that
    \begin{align*}
        \notag 
        \left\| \ola{X}_{t}^{\infty} - \ola X_{t}  \right\|_{L_2} 
        &\leq \left\| \ola{X}_0^{\infty} - \ola X_0 \right\|_{L_2} {\rm exp}\left(- 
       \int_0^{t} \frac{\bar\beta(s)}{2 \sigma^2}\left(1+ 2\bar{C}_{s} \sigma^2\right)\rmd s\right)\\
        &\leq \left(\left[ \left\| Y \right\|^2\right]+ \sigma^2 d
        \right)^{1/2} m_T
        \eqsp,
    \end{align*}
    where we have used \eqref{eq:initialization:X_inf_sup_t} as well as the fact that
    \begin{align*}
        \left\|X_0 - G\right\|_{L_2}
        = 
        \left(\left[ \left\| Y \right\|^2\right]+
        \left[ \left\| G \right\|^2\right]
        \right)^{1/2}
        = B\eqsp.
    \end{align*}
    Therefore,
    \begin{align*}
        \sup_{0 \leq t \leq T}\left\| \ola{X}_{t}^{\infty} - \ola X_{t}  \right\|_{L_2}
        &\leq \left(\left[ \left\| Y \right\|^2\right]+ \sigma^2 d
        \right)^{1/2}{\rm exp}\left(- 
        \int_0^{T} \frac{\bar\beta(s)}{ 2\sigma^2}\rmd s\right),
    \end{align*}
    corresponding to \eqref{eq:X_infty_X_ola:X_inf_sup_t}.

    On the other hand, we have that
    \begin{align*}
        \left\| \ola{X}_{t}^{\infty} - \ola{X}_{t_k}^{\infty}  \right\|_{L_2}
        \leq
        \left\| \ola{X}_{t} - \ola{X}_{t_k} \right\|_{L_2}+
        \left\| \left(\ola{X}_{t}^{\infty} - \ola{X}_{t}\right) - 
         \left(\ola{X}_{t_k}^{\infty} - \ola{X}_{t_k}\right) \right\|_{L_2}
        \eqsp.
    \end{align*}

    The process $(\ola{X}_{t}^{\infty} - \ola{X}_{t})_{t\geq 0}$ is determined by the following ODE:
    $$
    d\left(\ola{X}_{t}^{\infty} - \ola{X}_{t}\right)
    = \left( -\frac{\bar{\beta}(t)}{2\sigma^2}\left(\ola{X}_{t}^{\infty} - \ola{X}_{t}\right) + 2 \bar{\beta}(t)\left( \nabla \log \tilde{p}_{T-t} \left(\ola{X}_{t}^{\infty}\right) - \nabla \log \tilde{p}_{T-t} \left(\ola{X}_{t}\right) \right) \right)\rmd t \eqsp .
    $$
    Then,
    \begin{align*}
        &\left\| \left(\ola{X}_{t}^{\infty} - \ola{X}_{t}\right) - 
         \left(\ola{X}_{t_k}^{\infty} - \ola{X}_{t_k}\right) \right\|_{L_2} \\
         &\hspace{.5cm}= \left\|\int_{t_k}^{t} \left( - \frac{\bar{\beta}(s)}{2\sigma^2}
         \left(\ola{X}_{s}^{\infty} - \ola{X}_{s}\right) + 2\bar{\beta}(s)\left( \nabla \log \tilde{p}_{T-s} \left(\ola{X}_{s}^{\infty}\right) - \nabla \log \tilde{p}_{T-s} \left(\ola{X}_{s}\right) \right) \right)\rmd s \right\|_{L_2}
         \\
         &\hspace{.5cm}\leq \sup_{t_k\leq t\leq t_{k+1}}\left\|   \ola{X}_{t}^{\infty} - \ola{X}_{t}
         \right\|_{L_2} \int_{t_k}^{t_{k+1}} \left( \frac{1}{2\sigma^2}+2\bar{L}_{t}\right) \bar{\beta} (t) \rmd t \\
         &\hspace{.5cm}\leq B m_T \int_{t_k}^{t_{k+1}} \left( \frac{1}{2\sigma^2}+2\bar{L}_{t}\right) \bar{\beta} (t) \rmd t\eqsp.
    \end{align*}

  Write $(\ora{X}_t)_{t\in[0,T]}$ the time reversal of $(\ola{X}_t)_{t\in[0,T]}$, which clearly satisfies \eqref{eq:forward-SDE:beta}. Using the following equality in law
    \begin{align*}
        \ora{X}_{T-t_k} = \frac{m_{T-t_k}}{m_{T-t}}\ora{X}_{T-t}+ \left(1-\left(\frac{m_{T-t_k}}{m_{T-t}}\right)^2\right)^{1/2}\sigma G\eqsp,
    \end{align*}
    with $G\sim \mathcal{N}(0,I_d)$, we get
    \begin{align*}
        \left\| \ola{X}_{t} - \ola{X}_{t_k} \right\|_{L_2} =
        \left\| \ora{X}_{T-t_k} - \ora{X}_{T-t} \right\|_{L_2} &= 
       \left(1-\left(\frac{m_{T-t_k}}{m_{T-t}}\right)^2\right)^{1/2}
       \left(\left[ \left\| \ora{X}_{T-t} \right\|^2\right]+ \sigma^2 d
        \right)^{1/2}\\
        &\leq  \left(1-\left(\frac{m_{T-t_k}}{m_{T-t}}\right)^2\right)^{1/2} \sqrt{2} B\eqsp,
    \end{align*}
    where we have applied Proposition \ref{prop:true_backward_bound} in the last inequality.
    Since
    \begin{align*}
        1-\left(\frac{m_{T-t_k}}{m_{T-t}}\right)^2 &= 1-{\rm exp}\left(-\frac{1}{\sigma^2}\int_{T-t}^{T-t_k} \beta(s)\rmd s\right)\\
        &= \frac{1}{\sigma^2}
        \int_{T-t}^{T-t_k}{\rm exp}\left(-\frac{1}{\sigma^2}\int_{T-t}^{T-u} \beta(s)\rmd s\right)\beta(u)\rmd u\\
        &\leq \frac{1}{\sigma^2} h \beta(T) 
        \eqsp,
    \end{align*}
    which concludes the proof of  \eqref{eq:X_inf_sup_t}.    
\end{proof}

\section{Discussion on the hypotheses}

\begin{proposition}\label{prop:from_C_0_to_C_t}
Assume that $\log \pi_{\rm data}$ is $\Cpdata$-strongly concave and that $\Cpdata> 1/ \sigma^2$. Then, the modified score function $\log \tilde p_t (x)$ is, for any $t \in ( 0,T ]$, $C_t$-strongly concave, with 
\begin{align*}
    m_t &=  \exp\left(-\frac{1}{2 \sigma^2}\int_0^t \beta(s) \rm d s\right) \eqsp,
    \\
     C_t &= \frac{1}{ m_t^{2}/\Cpdata + \sigma^{2}\left(1-m_t^{2}\right)} - \frac{1}{\sigma^{2}}\eqsp.
\end{align*}
Moreover, we have that $C_t \leq \Cpdata - 1/\sigma^2$ for any $t\geq0$.
\end{proposition}

\begin{proof}
This result is also proved in \cite{saremi2023chain}. We provide an alternative proof here for completeness.
    For all $1\leq t\leq T$, $\ora X_t$ has the same law has $m_t X_0 + \sigma\sqrt{1 - m_t^2}Z$ where  $X_0 \sim \pi_{\rm data}$ and $Z \sim \mathcal{N}( 0 , I_d )$ are independent.
    Therefore, writing $p_0 = \pi_{\rm data}$,
\begin{align}\label{eq:def_p_t:convolution}
        p_t (y) & = \int_{\mathbb{R}^d} (2 \pi \sigma^2\left(1-m^2_t\right))^{-d/2} \exp \left\{ \frac{-\left\|y - x_0 m_t \right\|^2}{2 \sigma^2\left(1-m^2_t\right)} \right\} p_0(x_0) \rmd x_0 \eqsp.
    \end{align}
    This implies that 
    \begin{align*}
        \log p_t(y) & = - \frac{d}{2} \log \left(2 \pi \sigma^2\left(1-m^2_t\right) \right) + \log \left( \int_{\mathbb{R}^d} \exp \left\{ -\frac{\left\|y - x_0 m_t \right\|^2}{2 \sigma^2\left(1-m^2_t\right)} \right\} p_0(x_0) \rmd x_0 \right) \\
        & = - \frac{d}{2} \log \left(2 \pi \sigma^2\left(1-m^2_t\right) \right) + \log \left( \int_{\mathbb{R}^d} \exp \left\{ -\frac{\left\|y - u \right\|^2}{2 \sigma^2\left(1-m^2_t\right)} \right\} p_0 \left(\frac{u}{m_t} \right) \rmd u \right)\\
        &\hspace{9cm}+   \frac{d}{2 \sigma^2}\int_0^t \beta (s) \rmd s \eqsp.
    \end{align*}
    Since $\log p_0$ is $\Cpdata$-strongly concave, the function $x \mapsto p_0 \left(u/m_t \right)$ is $\Cpdata/ m_t^2$-strongly log-concave. Moreover, we have that the function
    $y\mapsto \exp \{ -\|y\|^2/(2 \sigma^2(1-m^2_t))\}$ is $(\sigma^2(1-m^2_t)^{-1}$-strongly log-concave. Applying \citet[Proposition 7.1][]{saumard2014log}, since $p_t$ is a convolution of the previous two functions up to terms independent in space, we have that $\log p_t$ is $\left(m_t^2/\Cpdata + \sigma^2\left(1-m^2_t\right)\right)^{-1}$-strongly concave. Note that if $\Cpdata\geq 1/\sigma^2$,  
    \begin{align*}
        \frac{\Cpdata}{m_t^2 + \sigma^2\Cpdata\left(1-m^2_t\right)} \geq \frac{1}{\sigma^2}.
    \end{align*}
    This entails that $\log \tilde p_t$ is $C_t$-strongly concave, with 
    \begin{align*}
        C_t &= \frac{1}{ m_t^{2}/\Cpdata + \sigma^{2}\left(1-m_t^{2}\right)} - \frac{1}{\sigma^{2}}\eqsp.
    \end{align*}
    Finally, finding the maximum $t\mapsto C_t$, is equivalent to find the maximum of the following function on $[0,1]$:
    \begin{align*}
        \psi:z\mapsto \frac{\Cpdata}{z + \sigma^2\Cpdata(1-z)} - \frac{1}{\sigma^2}\eqsp.
    \end{align*}
    We have that $\psi(0)= \Cpdata -1/\sigma^2$, $\psi(1)= 0$ and for all $z\in [0,1]$, 
    $$
    \psi'(z)= \frac{\sigma^2- 1/\Cpdata}{\left(z/\Cpdata + \sigma^2(1-z)\right)^2}\eqsp,
    $$ 
    which is negative since $\Cpdata\geq 1\leq1/\sigma^2$. Therefore, we  get $0\leq C_t\leq \Cpdata -1/\sigma^2$.
\end{proof}

\begin{proposition}\label{prop:from_L_0_to_L_t}
    If $\log \pi_{\rm data}$ is $\Lpdata$-smooth, then for all $0\leq t\leq T$, $\nabla \log \tilde p_{t}$ is $L_t$-Lipschitz in the space variable with 
    \begin{align*}
        L_t = \min \left\{ \frac{1}{\sigma^2\left(1-m^2_t\right)} ;  \frac{\Lpdata}{m^2_t} \right\} + \frac{1}{\sigma^2} \eqsp.
    \end{align*}
    Moreover, if $\Lpdata>1/\sigma^2$, we can choose $L_t$ as follows:
    \begin{align*}
        L_t = \min \left\{ \frac{1}{\sigma^2\left(1-m^2_t\right)} ;  \frac{\Lpdata}{m^2_t} \right\} - \frac{1}{\sigma^2} \eqsp.
    \end{align*}
    Moreover, in this case, we have that $L_t \leq \Lpdata$ for any $t\geq0$.
\end{proposition}

\begin{proof}
    In the proof of Proposition \ref{prop:from_C_0_to_C_t}, we proved that, if $\log \pi_{\rm data}$ is $\Cpdata$-strongly concave, $\log p_t$ is $\left(m_t^2/\Cpdata + \sigma^2\left(1-m^2_t\right)\right)^{-1}$-strongly concave i.e.,
    \begin{equation*}
        \nabla^2 \left( - \log p_t \right) (x) \succcurlyeq \frac{1}{m_t^2/\Cpdata + \sigma^2\left(1-m^2_t\right)} I_d \eqsp.
    \end{equation*}
    For $p_0 := \pi_{\rm data}$, we have that $p_t$ is given by  \eqref{eq:def_p_t:convolution}. This means that $p_t$ is the density of the sum of two independent random variables $X_1 + X_0$ of density respectively $q_0$ and $q_1$, such that
    \begin{align*}
        q_0(x) &\eqdef \frac{1}{m^d_t} p_0 \left(\frac{u}{m^d_t} \right) = e^{- \phi_0 (x)}\eqsp, \\
        q_1(x) &\eqdef \frac{1}{\left(2\pi \sigma^2 \left(1-m^2_t\right)\right)^{d/2}} \exp \left\{ -\frac{\|y\|^2}{2 \sigma^2\left(1-m^2_t\right)} \right\} = e^{- \phi_1 (x)} \eqsp,
    \end{align*}
    for two functions $\phi_0$ and $\phi_1$.
    Therefore, as in the proof of \citet[Proposition 7.1][]{saumard2014log}, we get
    \begin{align*}
         \nabla^2 \left( - \log p_t \right) (x) & = -\text{Var}(\nabla\phi_0(X_0) | X_0 + X_1 = x) + \mathbb{E}[\nabla^2\phi_0(X_0) | X_0 + X_1 = x] \\
         & = -\text{Var}(\nabla\phi_1(X_1) | X_0 + X_1 = x) + \mathbb{E}[\nabla^2\phi_1(X_1) | X_0 + X_1 = x] \eqsp.
    \end{align*}
    Since $\nabla \log p_0$ is $\Lpdata$-Lipschitz and from the definition of $q_1$,
    \begin{equation*}
        \nabla^2 \phi_0 \preccurlyeq \frac{\Lpdata}{m^2_t} I_d\eqsp,\quad 
        \nabla^2 \phi_1 \preccurlyeq  \frac{1}{\sigma^2\left(1-m^2_t\right)} I_d  \eqsp.
    \end{equation*}
    Hence, 
    \begin{equation*}
        \nabla^2 \left( - \log p_t \right) (x) \preccurlyeq \min \left\{ \frac{1}{\sigma^2\left(1-m^2_t\right)} ;   \frac{L_0}{m^2_t} \right\} I_d \eqsp.
    \end{equation*}
    Therefore, since the difference between $\nabla \log p_t$ and $\nabla \log \tilde p_t$ is a linear function, we can choose $L_t$ as follows:
    \begin{align*}
        L_t = \min \left\{ \frac{1}{\sigma^2\left(1-m^2_t\right)} ;   \frac{\Lpdata}{m^2_t} \right\} + \frac{1}{\sigma^2} \eqsp.
    \end{align*}
Clearly we have that $0\leq m_t^2\leq 1$, therefore $1/ m^2_t \geq 1$ and $1/ \left( 1- m^2_t\right) \geq 1$. This means that, if $\Lpdata\geq 1/\sigma^2$, 
    \begin{align*}
        \min \left\{ \frac{1}{\sigma^2\left(1-m^2_t\right)} ;   \frac{\Lpdata}{m^2_t} \right\} \geq \frac{1}{\sigma^2}.
    \end{align*}
    Thus, we can choose $L_t$ to be
    \begin{align*}
        L_t = \min \left\{ \frac{1}{\sigma^2\left(1-m^2_t\right)} ;   \frac{\Lpdata}{m^2_t} \right\} - \frac{1}{\sigma^2} \eqsp.
    \end{align*}
    Finally, since $m_0=1$, we have that $L_0 = \Lpdata -1/\sigma^2$. This function increases up to the point where $\Lpdata/m^2_t = (\sigma^2 (1-m^2_t)^{-1}$, achieved for $m^2_{t^*} = (\sigma^2 \Lpdata)/(\sigma^2 \Lpdata + 1)$. At this point, we have that $L_{t^*} =\Lpdata$. After this point the Lipschitz constant decreases to $0$, as $m_t\to 0$ for $t\to \infty$. This means that for any $t$, $L_t$ is bounded by $\Lpdata$.
\end{proof}

   \begin{proposition}
    \label{prop:step_size_choice:case_C_0}
        Assume that $\log \pi_{\rm data}$ is $\Lpdata$-smooth and $\Cpdata$-strongly concave. Consider the regular discretization $\{t_k, 0\leq k \leq N\}$ of $[0,T]$ of constant step size $h$.
        By choosing $h>0$ such that for all $t_k$ with $0\leq k\leq N-1$,
        \begin{align}
        \label{eq:condition_step_size1:case_C_0}
            h\leq \min\left\{
                \frac{\log(2) 2\sigma^2}{\beta(T)} ;
                \frac{\sigma^2 \Cpdata -1}{\sigma^2\Cpdata\left(\sigma^2 \Lpdata +1\right)\Lpdata\beta(T)} ;
                \frac{\sigma^2 \Cpdata -1}{\left(\sigma^2 \Lpdata -1\right)\Lpdata\beta(T)}
            \right\}\eqsp,
        \end{align}
        then, for all $0\leq k\leq N-1$,
        \begin{align*}
            0< 1 + 
        \frac{\widetilde m^2_{t_{k}}}{\widetilde m^2_{t_{k+1}}}
        \left(  \int_{t_k}^{t_{k+1}} \bar{L}_{t}
        \frac{\widetilde m_{t}}{\widetilde m_{t_{k}}}
        \bar \beta (t)  \rmd t \right)^2
        - 2
        \frac{\widetilde m_{t_{k}}}{\widetilde m_{t_{k+1}}}
        \int_{t_k}^{t_{k+1}} \bar{C}_{t}
        \frac{\widetilde m_{t}}{\widetilde m_{t_{k}}}
        \bar \beta (t)  \rmd t
            <1\eqsp.
        \end{align*}
        In addition, if
        \begin{multline}
        \label{eq:condition_step_size2:case_C_0}
            h\leq \min\bigg\{
                \frac{\log(2) 2\sigma^2}{\beta(T)} ;
                \frac{\sigma^2 \Cpdata -1}{\sigma^2 M+ \beta(T)\Lpdata\left(\sigma^2 \Lpdata -1\right)};
                \\
                \frac{\sigma^2 \Cpdata -1}{\sigma^2 \Cpdata}
                \frac{m_T^{2} \left(1-m_T^{2}\right)}{\sigma^2M \left(1-m_T^{2}\right)+ \beta(T)\Lpdata m_T^{2}}
                ;
                \left.
                \frac{ \left(\sigma^2 \Cpdata-1\right)\Lpdata}{\sigma^2 \Cpdata\left(\sigma^2 \Lpdata+1\right) \left(M+ \beta(T)\Lpdata^2\right)}
            \right\}\eqsp,
        \end{multline}
        then, for all $0\leq k\leq N-1$,
        \begin{multline*}
            0< 
            1 + 
            \frac{1}{2}\frac{\widetilde m^2_{t_{k}}}{\widetilde m^2_{t_{k+1}}}
            \left(  \int_{t_k}^{t_{k+1}} \bar{L}_{t}
            \frac{\widetilde m_{t}}{\widetilde m_{t_{k}}}
            \bar \beta (t)  \rmd t \right)^2
            - \frac{\widetilde m_{t_{k}}}{\widetilde m_{t_{k+1}}}
            \int_{t_k}^{t_{k+1}} \bar{C}_{t}
            \frac{\widetilde m_{t}}{\widetilde m_{t_{k}}}
            \bar \beta (t) \rmd t\\
            + \frac{\widetilde m^2_{t_{k}}}{\widetilde m^2_{t_{k+1}}} M h \int_{t_k}^{t_{k+1}} 
                \frac{\widetilde m_{t}}{\widetilde m_{t_{k}}} \bar \beta (t)
                \rmd t 
        <1 \eqsp . 
        \end{multline*}
    \end{proposition}
    \begin{proof}
    Define $\epsilon_1$ and $\epsilon_2$ as in \eqref{eq:step_size_choice:eps1}-\eqref{eq:step_size_choice:eps2}.
        From Proposition \ref{prop:step_size_choice}, we have that $\epsilon_i\in (0,1)$, for $i=1,2$, if we have \eqref{eq:condition_step_size1}-\eqref{eq:condition_step_size2}.
        
        First, we prove that \eqref{eq:condition_step_size1:case_C_0} implies \eqref{eq:condition_step_size1}. From Proposition \ref{prop:from_L_0_to_L_t}, we have that $L_t$ is bounded by $\Lpdata$ everywhere. Moreover, since $\widetilde m_{t_{k+1}}/\widetilde m_{t_{k}} = \exp\left(-\int_{t_{k}}^{t_{k+1}} \bar{\beta}(s)/2\sigma ^2 \rmd s\right)$, we can find $h$ small enough such that $2\widetilde m_{t_{k}}/\widetilde m_{t_{k+1}} \geq 1$. This is equivalent to $\int_{t_{k}}^{t_{k+1}} \bar{\beta}(s)/2\sigma ^2 \rmd s \leq \log(2)$ and it is implied by
        \begin{align*}
            h\leq \frac{\log(2) 2\sigma^2}{\beta(T)}\eqsp.
        \end{align*}
        Now, we study the function $t\mapsto C_t/L_t$. From the proof of the Proposition \ref{prop:from_C_0_to_C_t}, we have that
        \begin{align*}
            C_t &= \frac{1}{ m_t^{2}/\Cpdata + \sigma^{2}\left(1+m_t^{2}\right)} - \frac{1}{\sigma^{2}}\eqsp,
        \end{align*}
        which is a decreasing function. Moreover, from the proof of the Proposition \ref{prop:from_L_0_to_L_t}, we have that
        \begin{align*}
            L_t = \min \left\{ \frac{1}{\sigma^2\left(1-m^2_t\right)} ;   \frac{\Lpdata}{m^2_t} \right\} - \frac{1}{\sigma^2}\eqsp,
        \end{align*}
        which is an increasing function from $0$ up to $t^*$, such that $m^2_{t^*} = \frac{\sigma^2 \Lpdata}{\sigma^2 \Lpdata + 1}$ and decreasing for $t\geq t^*$. On the one hand, this means that for $t\in[0,t^*]$, the function $t\mapsto C_t/  L_t$ is decreasing, therefore reaching its minimum $\left(\sigma^2 \Cpdata -1\right)/ \left(\sigma^2 \Lpdata -1\right)$ in $0$, which is a positive quantity. On the other hand, for $t\geq t^*$, we have that
        \begin{align*}
            \frac{C_t}{L_t} &=
            \frac{1}{m_t^{2}/\Cpdata + \sigma^{2}\left(1+m_t^{2}\right)} \eqsp \frac{1}{\sigma^{2}}
            \frac{\left(\sigma^2 \Cpdata -1 \right)m^2_t}{\Cpdata}\eqsp \frac{\sigma^2\left(1-m_t^{2}\right)}{m_t^{2}}\\
            &\geq \frac{1}{\sigma^2}\frac{\sigma^2 \Cpdata -1}{\Cpdata} \left(1-m_t^{2}\right)\\
            &\geq \frac{1}{\sigma^2}\frac{\sigma^2 \Cpdata -1}{\Cpdata} \left(1-m_{t^*}^{2}\right) =
            \frac{\sigma^2 \Cpdata -1}{\sigma^2\Cpdata\left(\sigma^2 \Lpdata +1\right)}\eqsp.
        \end{align*}

        Therefore, combining the previous inequalities, we have that condition \eqref{eq:condition_step_size1:case_C_0} implies \eqref{eq:condition_step_size1}.

        Secondly, we prove \eqref{eq:condition_step_size2:case_C_0} implies \eqref{eq:condition_step_size2}. Take $h$ to satisfy
        \begin{align*}
            h\leq \frac{\log(2) 2\sigma^2}{\beta(T)}\eqsp.
        \end{align*}

        We now need to study the function $t\mapsto \frac{C_t}{M+ \beta(T)\Lpdata\bar{L}_t}$.
        On the one hand, this function is decreasing for $t\in[0,t^*]$, therefore reaching its minimum $\frac{\sigma^2 \Cpdata -1}{\sigma^2 M+ \beta(T)\Lpdata\left(\sigma^2 \Lpdata -1\right)}$ in $0$, which is a positive quantity. On the other hand, for $t\geq t^*$, we have that
        \begin{align*}
            \frac{C_t}{M+ \beta(T)\Lpdata L_t}
            &=
            \frac{1}{m_t^{2}/\Cpdata + \sigma^{2}\left(1+m_t^{2}\right)} \eqsp \frac{1}{\sigma^{2}}
            \frac{\left(\sigma^2 \Cpdata -1 \right)m^2_t}{\Cpdata}\eqsp \frac{\sigma^2\left(1-m_t^{2}\right)}{\sigma^2M \left(1-m_t^{2}\right)+ \beta(T)\Lpdata m_t^{2}}
            \\
            &\geq \frac{1}{\sigma^2}\frac{\sigma^2 \Cpdata -1}{\Cpdata} 
            \frac{m_t^{2}\left(1-m_t^{2}\right)}{\sigma^2M \left(1-m_t^{2}\right)+ \beta(T)\Lpdata m_t^{2}}\eqsp.
        \end{align*}
        Controlling from below the previous quantity, boils down to control from below the function $\psi(y)= \frac{y(1-y)}{\sigma^2M (1-y) +\beta(T)\Lpdata y}$ for $y\in [m^2_{t^*},m_T^2]$. We see that $\psi$ in this interval can be bounded by $\min\{\psi\left(m^2_{t^*}\right), \psi\left(m^2_{T}\right)\}$. Therefore, we get
         \begin{align*}
            \frac{C_t}{M+ \beta(T)\Lpdata L_t}
            &\geq \frac{\sigma^2 \Cpdata -1}{\sigma^2 \Cpdata}
            \min\left\{
                \frac{m_T^{2} \left(1-m_T^{2}\right)}{\sigma^2M \left(1-m_T^{2}\right)+ \beta(T)\Lpdata m_T^{2}}
                ;
                \frac{\Lpdata}{\left(\sigma^2 L_0+1\right) \left(M+ \beta(T)\Lpdata^2\right)}
            \right\}
            \eqsp.
        \end{align*}
        Therefore, combining the previous inequalities, we have that condition \eqref{eq:condition_step_size2:case_C_0} implies \eqref{eq:condition_step_size2}.
        
    \end{proof}

\section{Details on numerical experiments}

This section is divided into two parts. The first part is dedicated to providing  detailed implementation choices for the numerical experiments presented in Section \ref{sec:exp}. The second part displays additional experiments and more details about the experiments of Section \ref{sec:exp}. All experiments were conducted on a local computer CPU equipped with an Apple M3 processor (8GB of unified memory). This setup is sufficient to replicate the experiments of this paper.

\subsection{Implementation choices}

\subsubsection{Exact score and metrics in the Gaussian case}
\label{subsec:formulas_gaussian_setting}

\begin{lemma} \label{lem:exactscore}
Assume that the forward process defined in (\ref{eq:forward-SDE:beta}) : 
\begin{equation*} 
    \rmd  \overrightarrow{X}_t = - \frac{\sched(t)}{2 \sigma^2}  \overrightarrow{X}_t \rmd t + \sqrt{ \sched(t)}  \rmd B_t, \quad \ora{X}_0 \sim \pi_0,
    \end{equation*} 
is initialised with $\pi_0$ the Gaussian probability density function with mean $\mu_0$ and variance $\Sigma_0$. Then, the score function of (\ref{eq:forward-SDE:beta}) is:
\begin{equation*}
  \nabla \log p_t: x\mapsto -(m_t^2 \Sigma_0 + \sigma_t^2 \mathrm{I}_d)^{-1} (x - m_t \mu_0)\eqsp,  
\end{equation*}
where $p_t$ is the probability density function of $\overrightarrow{X}_t$,  $m_t = \exp\{-\int_0^t \beta(s) \rm d s/(2 \sigma^2)\}$ and $ \sigma^2_t = \sigma^2 (1-m_t^2)$.
\end{lemma}

\begin{proof}
Note that $\overrightarrow X_t $ has the same law as  $m_t X_0 + \sigma_tZ$ where $Z \sim \mathcal{N}(0,\mathrm{I}_d)$ is independent of $X_0$. Therefore $\overrightarrow X_t \sim \mathcal{N}(m_t \mu_0, \ora{\Sigma_t})$ with $\ora{\Sigma_t}= m_t^2 \Sigma_0 + \sigma_t^2 \mathrm{I}_d$ which concludes the proof.

\end{proof}

\begin{lemma} \label{KL divergence Gaussian rvs}
Let $\mu_1, \mu_2$ in $\rset^d$ and $\Sigma_1$ and $\Sigma_2$ be two definite positive matrices in $\rset^{d\times d}$. Then, 
\begin{equation}
\kl (\gausspdf_{\mu_1,\Sigma_1}\|\gausspdf_{\mu_2,\Sigma_2}) =\frac{1}{2} \left(  \log \frac{| \Sigma_2|}{|\Sigma_1|} - d + \operatorname{Tr}\left(\Sigma_2^{-1} \Sigma_1\right) + \left(\mu_2 - \mu_1\right)^\top \Sigma_2^{-1} \left(\mu_2 - \mu_1\right) \right)\eqsp.
\end{equation}
\end{lemma}

\begin{lemma} \label{W2 Gaussian rvs}
Let $\mu_1, \mu_2$ in $\rset^d$ and $\Sigma_1$ and $\Sigma_2$ be two definite positive matrices in $\rset^{d\times d}$. Then, 
\begin{equation}
\mathcal{W}_2^2 (\gausspdf_{\mu_1,\Sigma_1},\gausspdf_{\mu_2,\Sigma_2}) =\left\| \mu_2 - \mu_1 \right\|^2 + \operatorname{Tr} \left( \Sigma_1 + \Sigma_2 - 2 \left( \Sigma_2^{1/2} \Sigma_2 \Sigma_1^{1/2} \right)^{1/2} \right) \eqsp.
\end{equation}
\end{lemma}

\begin{lemma} \label{Info Fisher Gaussian rvs}
The relative Fisher information between $X_0 \sim \mathcal{N}(\mu_0, \Sigma_0)$ and $X_{\infty} \sim \mathcal{N}(0, \sigma^2 \mathrm{I}_d)$ is given by: 
\begin{align*}
\mathcal{I} \left(\gausspdf_{\mu_0 , \Sigma_0} \| \gausspdf_{\sigma^2} \right) =  \frac{1}{\sigma^4} \left( \operatorname{Tr} \left( \Sigma_{0} \right) + \left\|\mu_{0}\right\|^2 \right)  - \frac{2 d}{\sigma^2} + \operatorname{Tr} \left( \Sigma_0^{-1} \right)\eqsp.
\end{align*}
\end{lemma}

\begin{proof}
The relative Fisher information between $X_0$ and $X_{\infty}$ is given by
\begin{align*}
\mathcal{I} \left( \gausspdf_{\mu_0 , \Sigma_0} \| \gausspdf_{\sigma^2} \right) = \int \left\| \nabla \log \left(  \frac{ \gausspdf_{\mu_0 , \Sigma_0} (x) }{ \gausspdf_{\sigma^2}(x)}\right) \right\|^2  \gausspdf_{\mu_0 , \Sigma_0} (x) \rmd x\eqsp.
\end{align*}
Write
\begin{align*}
\nabla \log \frac{\gausspdf_{\mu_0 , \Sigma_0} (x)}{ \gausspdf_{\sigma^2} (x) } 
    = \frac{x}{\sigma^2} - \Sigma_0^{-1} (x - \mu_0)\eqsp,
\end{align*}
so that,
\begin{align*}
\left\| \nabla \log \frac{\gausspdf_{\mu_0 , \Sigma_0} (x)}{ \gausspdf_{\sigma^2} (x) } \right\|^2 & = \left\|  \frac{x}{\sigma^2} - \Sigma_0^{-1} (x - \mu_0) \right\|^2 \\
& = \left(\frac{x}{\sigma^2} - \Sigma_0^{-1} (x - \mu_0)\right)^\top \left(\frac{x}{\sigma^2} - \Sigma_0^{-1} (x - \mu_0)\right) \\
& = \frac{ \| x\|^2}{\sigma^4} - \frac{2}{ \sigma^2} x^\top \Sigma_0^{-1} (x - \mu_0) + (x - \mu_0)^\top \Sigma_0^{-2} (x - \mu_0)\eqsp.
 \end{align*}
First,
\begin{align*}
\mathbb{E} \left[ \frac{\| X_0\|^2}{\sigma^4}\right]  = \frac{1}{\sigma^4} \left( \text{Tr} \left( \Sigma_{0} \right) + \|\mu_{0}\|^2 \right)\eqsp. 
\end{align*}
Then,
\begin{equation*}
\mathbb{E}  \left[ \frac{2}{ \sigma^2} X_0^T \Sigma_0^{-1} (X_0 - \mu_0) \right] 
 = \frac{2}{ \sigma^2} \left(  \text{Tr} \left( \Sigma_0^{-1} \mathbb{E} \left[X_0 X_0^\top  \right] \right) -   \mu_0^\top \Sigma_0^{-1}  \mu_0  \right)\eqsp.
\end{equation*}

Using that $\mathbb{E} \left[X_0  X_0^\top  \right]  = \Sigma_0 + \mu_0 \mu_0^\top$
yields
\begin{align*}
\mathbb{E} \left[ \frac{2}{ \sigma^2} X_0^\top \Sigma_0^{-1} (X_0 - \mu_0) \right] & = \frac{2}{ \sigma^2} \left(  \text{Tr} \left( \Sigma_0^{-1} \left( \Sigma_0 + \mu_0 \mu_0^\top \right) \right)   - \mu_0^\top \Sigma_0^{-1}  \mu_0  \right) \\
& = \frac{2}{ \sigma^2} \left( d + \text{Tr} \left( \Sigma_0^{-1}  \mu_0 \mu_0^\top  \right)   - \mu_0^\top \Sigma_0^{-1}  \mu_0  \right) \\
& = \frac{2 d}{\sigma^2}\eqsp.
\end{align*}
Finally,
\begin{align*}
\mathbb{E} \left[ (X_0 - \mu_0)^\top \Sigma_0^{-2} (X_0 - \mu_0) \right] &= \mathbb{E} \left[ \operatorname{Tr} \left( (X_0 - \mu_0)^\top \Sigma_0^{-2} (X_0 - \mu_0) \right) \right] \\
&= \mathbb{E} \left[ \operatorname{Tr} \left( \Sigma_0^{-2} (X_0 - \mu_0) (X_0 - \mu_0)^\top\right) \right] \\
&= \text{Tr} \left( \Sigma_0^{-2} \mathbb{E} \left[   (X_0 - \mu_0) (X_0 - \mu_0)^\top  \right] \right) \\
& = \operatorname{Tr} \left( \Sigma_0^{-2} \Sigma_0 \right) \\
& = \operatorname{Tr} \left( \Sigma_0^{-1} \right)\eqsp,
\end{align*}
which concludes the proof.
\end{proof}

\begin{proposition} \label{prop:bar_l_t}
    Under the same assumptions as in Lemma \ref{lem:exactscore}, the Euclidean norm of the score function admits the following upper bound for $t_1 \leq t_2$:

\begin{align*}
    \sup_{ t_1 \leq t \leq t_2} \left\| \nabla \log p_{t_1}(x) -  \nabla \log p_{t}(x) \right\|  & \leq \left(t_2 - t_1 \right) \max \left\{  \left\| \mu_0 \right\| \kappa_2 ; \kappa_1 \right\}  \left( 1 + \left\| x \right\| \right) \eqsp,
\end{align*}
with 

\begin{align*}
\kappa_1 \eqdef \frac{m_{t_1}^2  \frac{\beta(t_2)}{\sigma^2} \left| \lambda_{\rm min} - \sigma^2 \right| }{ \left| \left( \sigma^2 + m_{t_1}^2 \left( \lambda_{\rm min} - \sigma^2 \right) \right) \left( \sigma^2 + m_{t_2}^2 \left( \lambda_{\rm min} - \sigma^2 \right) \right) \right| } \eqsp,
\end{align*}

and

\begin{align*}
     \kappa_2 \eqdef \frac{ m_{t_1} \frac{\beta(t_2)}{2 \sigma^2} \left| m_{t_1} m_{t_2} \left(  \lambda_{\rm min} - \sigma^2 \right) - \sigma^2 \right| }{ \left| \left( \sigma^2 + m_{t_1}^2 \left( \lambda_{\rm min} - \sigma^2 \right) \right) \left( \sigma^2 + m_{t_2}^2 \left( \lambda_{\rm min} - \sigma^2 \right) \right) \right|} \eqsp,
\end{align*}

where $\lambda_{\min}$ is the smallest eigenvalue of $\Sigma_0$.
\end{proposition}

\begin{proof}
        Let $t_1 \leq t_2$,

    \begin{align*}
        \left\| \nabla \log p_{t_1}(x) -  \nabla \log p_{t_2}(x)\right\| & = \left\| - (m_{t_1}^2 \Sigma_0 + \sigma_{t_1}^2 \mathrm{I}_d)^{-1} (x - m_{t_1} \mu_0) + (m_{t_2}^2 \Sigma_0 + \sigma_{t_2}^2 \mathrm{I}_d)^{-1} (x - m_{t_2} \mu_0) \right\| \\
        & \leq \left\| \left( m_{t_1} \left(m_{t_1}^2 \Sigma_0 + \sigma_{t_1}^2 \mathrm{I}_d \right)^{-1} - m_{t_2}\left( m_{t_2}^2 \Sigma_0 + \sigma_{t_2}^2 \mathrm{I}_d \right)^{-1} \right) \mu_0 \right\| \\
        & + \left\| \left( \left(m_{t_1}^2 \Sigma_0 + \sigma_{t_1}^2 \mathrm{I}_d \right)^{-1} - \left(m_{t_2}^2 \Sigma_0 + \sigma_{t_2}^2 \mathrm{I}_d \right)^{-1} \right) x  \right\| \eqsp.
    \end{align*}

Writing $
    M_t =  \left(m_{t}^2 \Sigma_0 + \sigma_{t}^2 \mathrm{I}_d \right)^{-1}$ we have, for $t_1 \leq t_2$,
\begin{align*}
    \left\| M_{t_1} - M_{t_2} \right\| & \leq \left|\frac{1}{m_{t_1}^2 \lambda_{\rm min} + \sigma_{t_1}^2} - \frac{1}{m_{t_2}^2 \lambda_{\rm min} + \sigma_{t_2}^2}\right| \\
    &\leq  \left| \frac{\left( m_{t_2}^2 - m_{t_1}^2 \right) \left( \lambda_{\rm min} - \sigma^2 \right) }{ \left( \sigma^2 + m_{t_1}^2 \left( \lambda_{\rm min} - \sigma^2 \right) \right) \left( \sigma^2 + m_{t_2}^2 \left( \lambda_{\rm min} - \sigma^2 \right) \right)} \right| \\
    &\leq \left(t_2 - t_1 \right)  \underbrace{\frac{m_{t_1}^2  \frac{\beta(t_2)}{\sigma^2} \left| \lambda_{\rm min} - \sigma^2 \right| }{ \left| \left( \sigma^2 + m_{t_1}^2 \left( \lambda_{\rm min} - \sigma^2 \right) \right) \left( \sigma^2 + m_{t_2}^2 \left( \lambda_{\rm min} - \sigma^2 \right) \right) \right| }}_{\kappa_1} \eqsp.
\end{align*}
Moreover, 
for $t_1 \leq t_2$,
\begin{align*}
    \left\| m_{t_1} M_{t_1} - m_{t_2} M_{t_2} \right\| & \leq \left|\frac{m_{t_1}}{m_{t_1}^2 \lambda_{\rm min} + \sigma_{t_1}^2} - \frac{m_{t_2}}{m_{t_2}^2 \lambda_{\rm min} + \sigma_{t_2}^2}\right| \\
    &\leq  \left| \frac{\left( m_{t_1} m_{t_2}^2 - m_{t_2} m_{t_1}^2 \right) \left( \lambda_{\rm min} - \sigma^2 \right) + \sigma^2 \left( m_{t_1} - m_{t_2} \right) }{ \left( \sigma^2 + m_{t_1}^2 \left( \lambda_{\rm min} - \sigma^2 \right) \right) \left( \sigma^2 + m_{t_2}^2 \left( \lambda_{\rm min} - \sigma^2 \right) \right)} \right| \\
    &\leq \frac{ \left| m_{t_2} - m_{t_1} \right| \left| m_{t_1} m_{t_2} \left( \lambda_{\rm min} - \sigma_2 \right) - \sigma^2 \right| }{ \left| \left( \sigma^2 + m_{t_1}^2 \left( \lambda_{\rm min} - \sigma^2 \right) \right) \left( \sigma^2 + m_{t_2}^2 \left( \lambda_{\rm min} - \sigma^2 \right) \right) \right|} \\
    &\leq   \left( t_2 - t_1 \right) \underbrace{\frac{ m_{t_1} \frac{\beta(t_2)}{2 \sigma^2} \left| m_{t_1} m_{t_2} \left(  \lambda_{\rm min} - \sigma^2 \right) - \sigma^2 \right| }{ \left| \left( \sigma^2 + m_{t_1}^2 \left( \lambda_{\rm min} - \sigma^2 \right) \right) \left( \sigma^2 + m_{t_2}^2 \left( \lambda_{\rm min} - \sigma^2 \right) \right) \right|}}_{\kappa_2}   \eqsp.
\end{align*}
Finally, 
    \begin{align*}
        \left\| \nabla \log p_{t_1}(x) -  \nabla \log p_{t_2}(x)\right\|  & \leq \left(t_2 - t_1 \right) 
         \left\| \mu_0 \right\| \kappa_2 + \left(t_2 - t_1 \right) \kappa_1 \left\| x \right\| \\
        &\leq  \left(t_2 - t_1 \right) \max \left\{  \left\| \mu_0 \right\| \kappa_2 ; \kappa_1 \right\}  \left( 1 + \left\| x \right\| \right) \eqsp. 
    \end{align*}

\end{proof}

\subsubsection{Stochastic differential equation exact simulation}
\label{sec:exact_simulation_forward}

In certain cases, exact simulation of stochastic differential equations is possible. In particular, due to the linear nature of the drift the forward process (\ref{eq:forward-SDE:beta}) can be simulated exactly. Indeed, the marginal distribution of (\ref{eq:forward-SDE:beta}) at time $t$ writes as 
\begin{equation*}
 \overrightarrow X_t = m_t X_0 + \sigma_t Z\eqsp,
\end{equation*}
with 
$Z \sim \mathcal{N} \left( 0, \mathrm{I}_d \right)$ independent of $X_0$, $X_0 \sim \pi_0$, $m_t = \exp\{-\int_0^t \beta(s) \rm d s/(2 \sigma^2)\}$ and $ \sigma^2_t = \sigma^2 ( 1 - \exp\{- \int_0^t \beta(s)/\sigma^2 \rm d s\}) $. Therefore, sampling from the forward process only necessitates access to samples from $\pi_0$ and $\mathcal{N}(0,\mathrm{I}_d)$. 

\subsubsection{Noise schedules}
\label{sec:noise_sched}

\paragraph{Linear and parametric noise schedules.}

In Section \ref{sec:exp}, we introduced parametric noise schedules of the form 
\begin{align*}
\beta_a(t) &\propto (\rme^{at} -1)/(\rme^{aT} -1)\eqsp,
\end{align*}
with $a \in \mathbb{R}$ ranging from $-10$ to $10$ (see Figure \ref{fig:noise_schedules_appendix}). For all $a$, with a time horizon of $T = 1$, the initial and final values have been set to match exactly the schedule prescribed by \citet{song2021score} (i.e. $\beta_a(0) = 0.1$ and $\beta_a(1) = 20$) when $a = 0$ (linear schedule). 

\begin{figure}[h]
    \centering
    \includegraphics[width=0.6\linewidth]{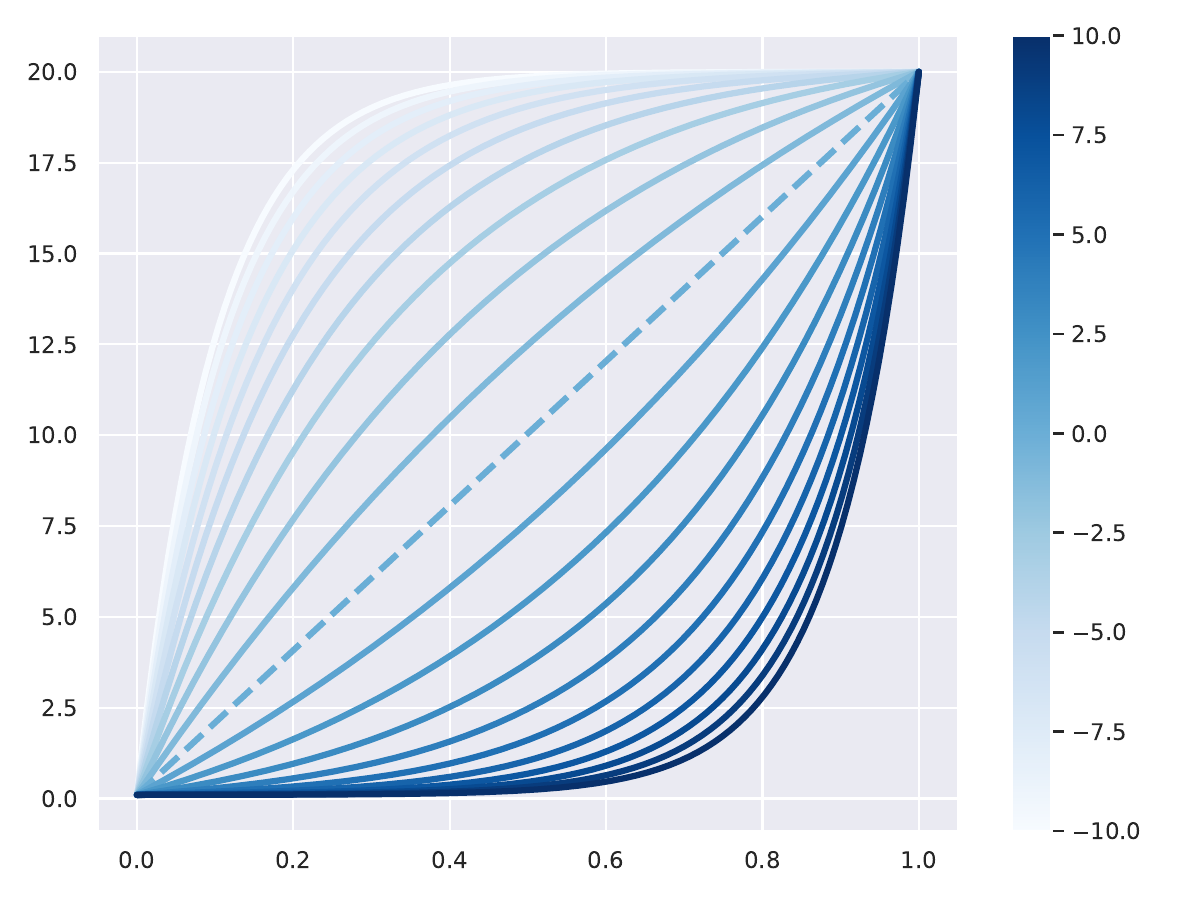}
    \caption{Evolution of noise schedules $\beta_a$ w.r.t.\ time, for different values of parameter between $-10$ to $10$. The linear case $a=0$ \citep{song2019generative, song2021score} is dashed.}
    \label{fig:noise_schedules_appendix}
\end{figure}

As shown in Section \ref{sec:exact_simulation_forward} $m_t$ and $\sigma_t$ are the two quantities of interest in the calibration of the noising procedure of the forward proces. Their values for different choices of $a$ are displayed in Figure \ref{fig:convergence_forward}.

\begin{figure}[h]
    \centering
    \begin{tabular}{cc}
    \includegraphics[width=0.45\linewidth]{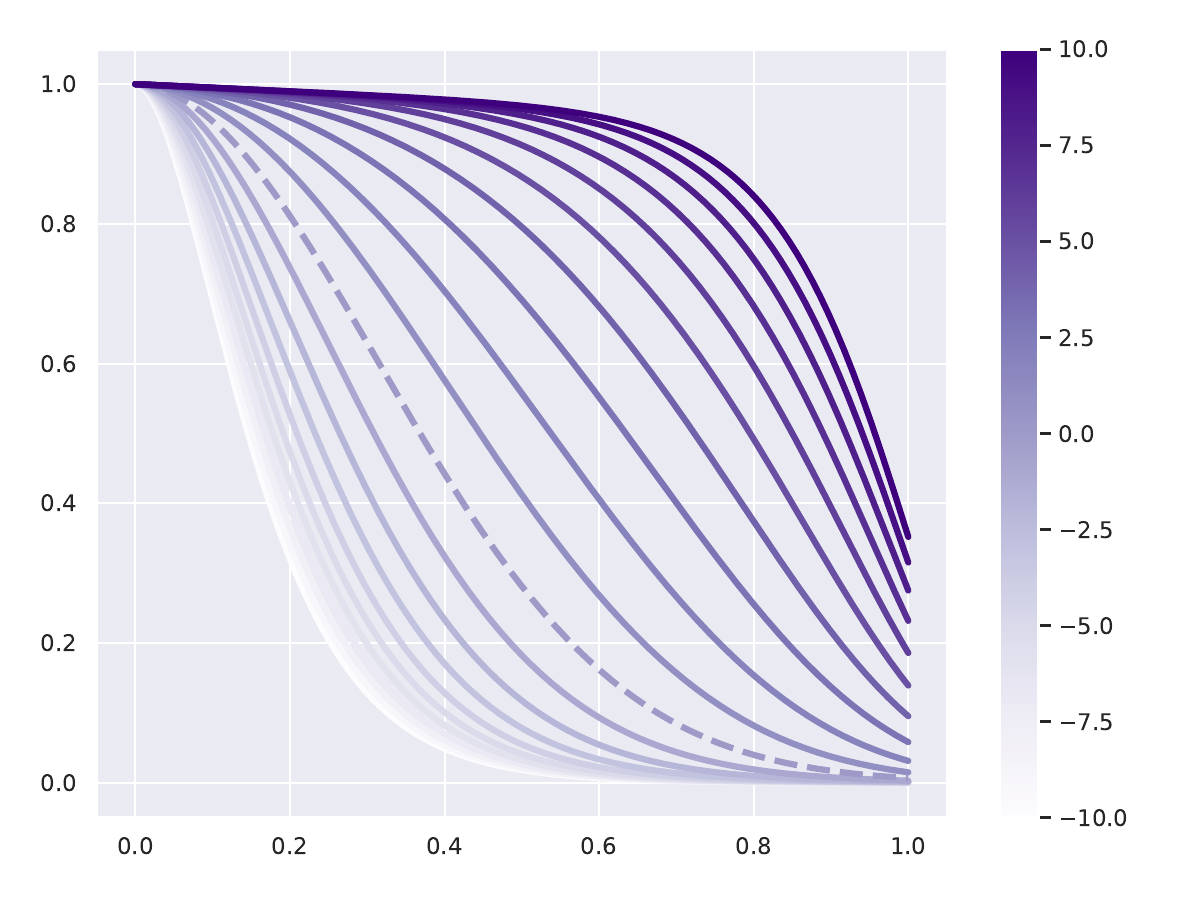} &
    \includegraphics[width=0.45\linewidth]{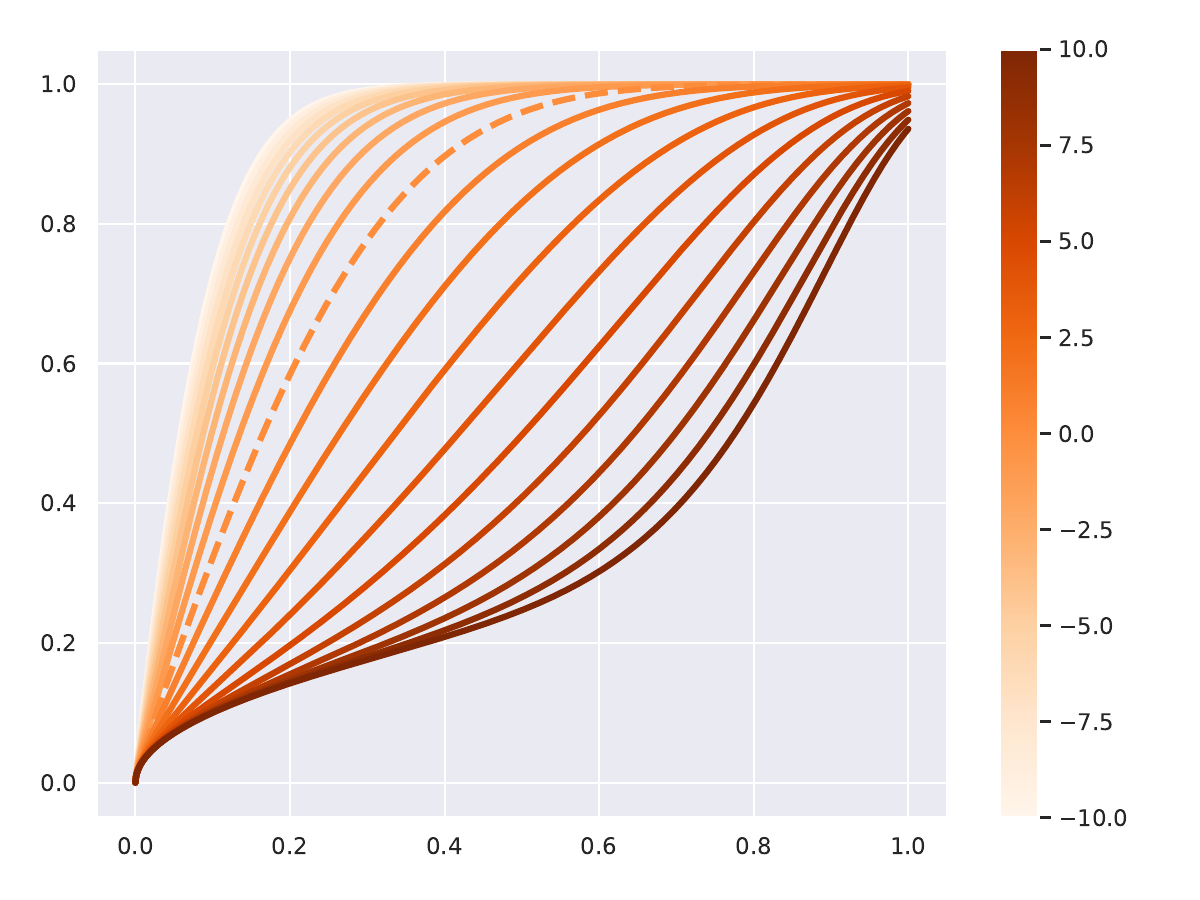} \\
    $m_t$ & $\sigma_t$
    \end{tabular}
    \caption{Evolution of  $m_t$ and  $\sigma_t$ over time, for different choices of $a$ in the noise schedule $\beta_a$ used in see Section \ref{sec:exp}. The stationary distribution of the forward process $\sigma^2$ is set to 1. 
    The range for $a$ spans from -10 to 10, with the dashed line representing the linear schedule as proposed originally in the VPSDE models \citep{song2021score}.}
    \label{fig:convergence_forward}
\end{figure}

\paragraph{Cosine noise schedule. }
We consider the cosine schedule introduced in \cite{nichol2021improved} for which the forward process is defined for $t \in \{ 1 , ..., T\}$ as
$$ X_t \eqdef \sqrt{ \bar \alpha_t} X_0 + \sqrt{ 1 - \bar \alpha_t} Z \eqsp, $$
with $X_0 \sim \pi_{\rm data}$, $Z \sim \mathcal{N} \left( 0 , I_d \right)$ and with
$$ 
\bar \alpha_t = \frac{f(t)}{f(0)}; \eqsp f(t) = \cos \left( \frac{t/T + s}{1 + s} \frac{\pi }{2} \right)^2 \eqsp. 
$$
To use this noise schedule in the SDE setting we notice that the forward process writes, for $t \in [0,T]$,
$$ 
\ora X_t = m_t X_0 + \sigma_t Z \eqsp,
$$
with $m_t = \exp\{-\int_0^t \beta(s) \rm d s/(2 \sigma^2)\}$, $ \sigma^2_t = \sigma^2 (1-m_t^2)$ and $Z \sim \mathcal{N}\left( 0, I_d \right)$. Therefore, we can simply identify $\beta_{\cos}$ by solving 
\begin{align*}
    - \int_0^t \frac{\beta_{cos}(s)}{2 \sigma^2} \rmd s = \log \left(\bar \alpha_t \right) \eqsp, 
\end{align*}
which yields the following noising function:
\begin{align}
\label{eq:cosine_sched}
    \beta_{\rm cos}(t) \eqdef \sigma^2 \frac{\pi}{T \left( s + 1 \right)}  \tan \left( \frac{\pi \left( s + t/T \right)}{2 \left( s + 1 \right) } \right) \eqsp.
\end{align}
Finally, to ensure fair comparison with the linear schedule and the parametric schedules defined in Section \ref{sec:exp}, we set in all our experiments $s=0.021122$ so that $\beta_{\cos} (0) \approx \beta_a(0) = 0.1$ for any $a$.

\begin{figure}[h]
    \centering
    \includegraphics[width=0.6\linewidth]{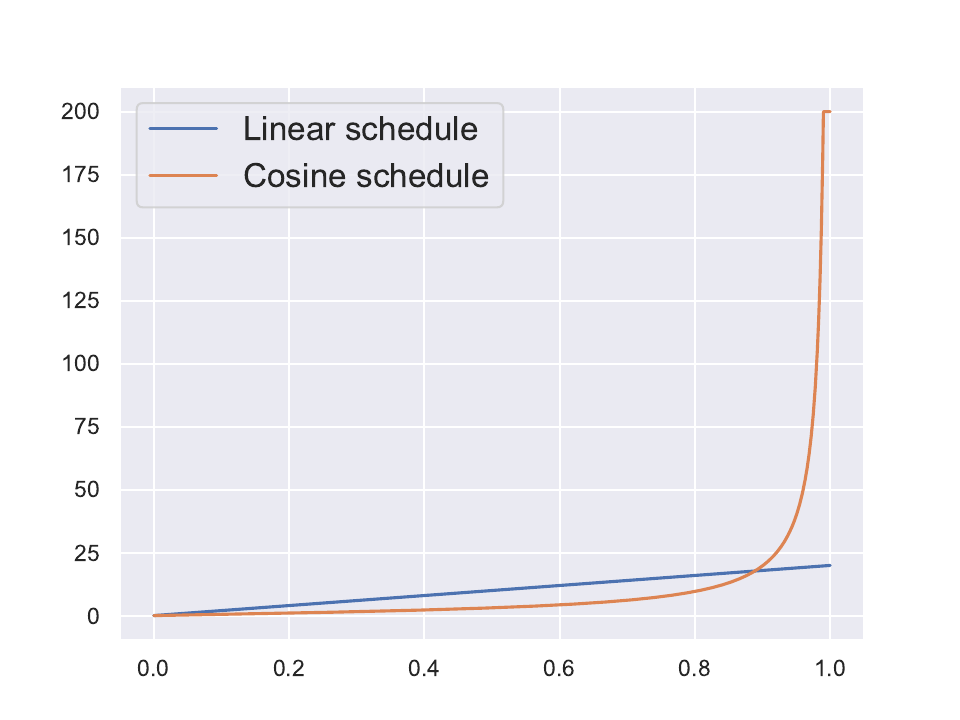}
\caption{Evolution of noising functions under the cosine schedule (orange, $\beta_{\rm cos}$) compared to the linear schedule ($\beta_0$, blue) over time with $\sigma^2 = 1$ and $s = 0.021122$. Additionally, since $\beta_0$ increases unboundedly near $T$, we clip its value to 200 for better visualization.}

    \label{fig:noise_schedule_cosine}
\end{figure}

\begin{figure}[h]
    \centering
    \begin{tabular}{cc}
    \includegraphics[width=0.45\linewidth]{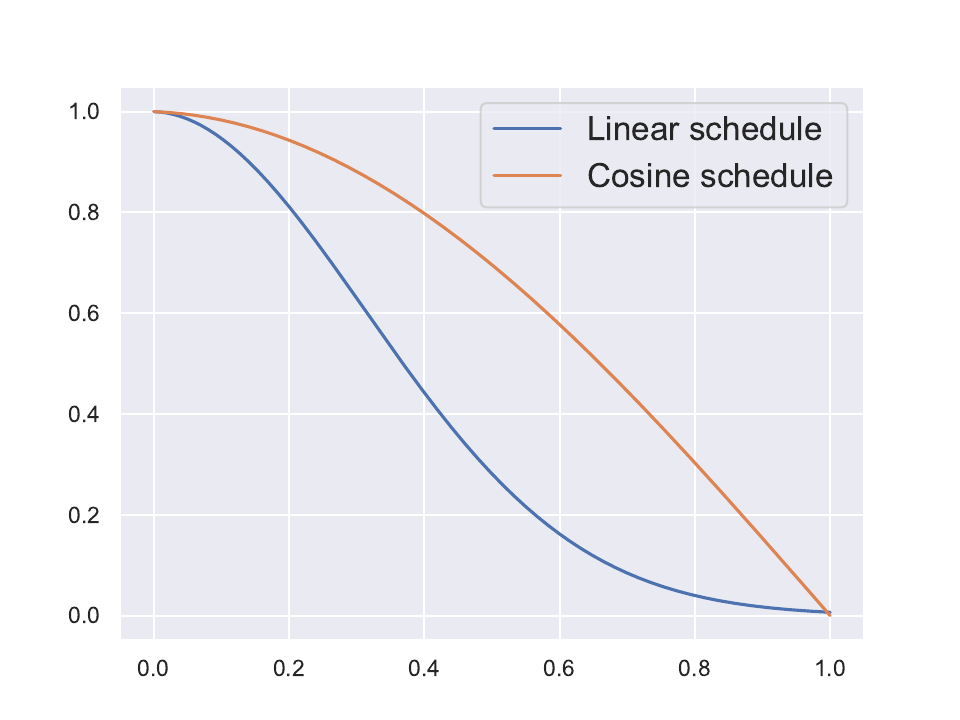} &
    \includegraphics[width=0.45\linewidth]{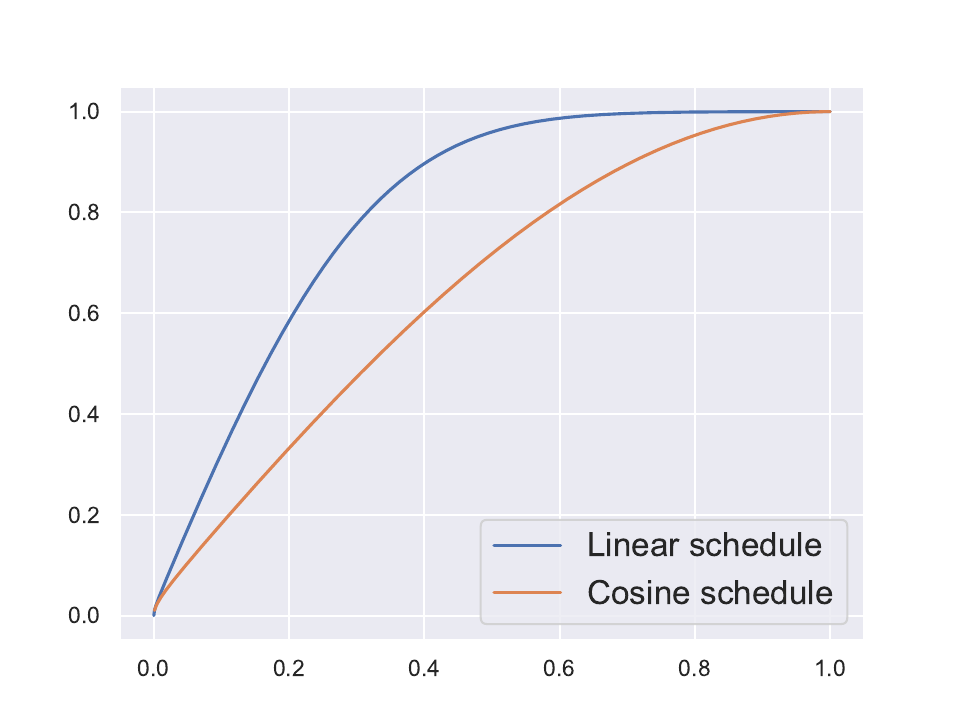} \\
    $m_t$ & $\sigma_t$
    \end{tabular}
    \caption{Evolution of  $m_t$ and  $\sigma_t$ for both the cosine schedule (orange) and the linear schedule (blue) w.r.t.\ time, with $s = 0.021122$ and $\sigma^2 = 1$. We clip the value of $\beta_{\rm cos}$ by 200 for better visualization.}
    \label{fig:convergence_forward_cosine}
\end{figure}

\subsubsection{Discretization details of the diffusion SDE}

In contrast to the forward process, described in Equation (\ref{eq:forward-SDE:beta}), which is simulated exactly, the backward process needs to be discretized. 
Recall that the backward process of (\ref{eq:forward-SDE:beta}) is given by: 
\begin{align*}
\rmd \overleftarrow{X}_t =   - \frac{\bar \sched (t)}{2 \sigma^2} \overleftarrow{X}_t + \bar \sched (t) \nabla \log   p_{T-t}\left( \ola{X}_t\right) \rmd t  +  \sqrt{\bar \sched(t)} \rmd B_t, \quad \ola X_0 \sim \pi_{\infty}.
\end{align*}
Consider time intervals $0 \leq t_k \leq t \leq t_{k+1} \leq T$, with $t_k = \sum_{\ell=1}^{k} \gamma_\ell$ and $T = \sum_{k=1}^N \gamma_k$.

In our theoretical analysis, we have considered the Exponential Integrator discretization,  defined recursively for $t \in  [ t_k , t_{k+1} ]$ by
\begin{equation*}
    \rmd \overleftarrow{X}^{EI}_t  =  \bar \sched (t)\left( - \frac{1}{2 \sigma^2} \overleftarrow{X}^{EI}_t +  \nabla \log p_{T-t_k} \left( T-{t_k}, \ola{X}_{t_k}^{EI}    \right) \right)  \rmd t +  \sqrt{\bar \sched(t)} \rmd B_t, \quad \ola X_0^{EI} \sim \pi_{\infty}\eqsp.
\end{equation*}

In the numerical experiments, we have given priority to the Euler-Maruyama discretization, which is widely used, and defined recursively for $t \in  [ t_k , t_{k+1} ]$ by
\begin{align}
\label{def_Euler_Maruyama}
\rmd \overleftarrow{X}_t^{EM} =   - \frac{\bar \sched (t_k)}{2 \sigma^2} \overleftarrow{X}_{t_k}^{EM} + \bar \sched (t_k) \nabla \log   p_{T-t_k}\left( \ola{X}_{t_k}^{EM}\right) \rmd t  +  \sqrt{\bar \sched(t_k)} \rmd B_t, \quad \ola X_0^{EM} \sim \pi_{\infty}\eqsp.
\end{align}

To ensure transparency, the graphs presented in Figure \ref{fig:gaussian_xp} of Section \ref{sec:exp_gaussian} are reproduced in Figure \ref{fig:fig2_using_EI}  using an Exponential Integrator discretization scheme. As expected for fine discretization steps (here 500 steps were used) the two schemes produce nearly identical results.

    \begin{figure}[h!]
    \centering
    \begin{tabular}{ccc}
    \hspace{-1cm}
        \includegraphics[width=0.32\linewidth]{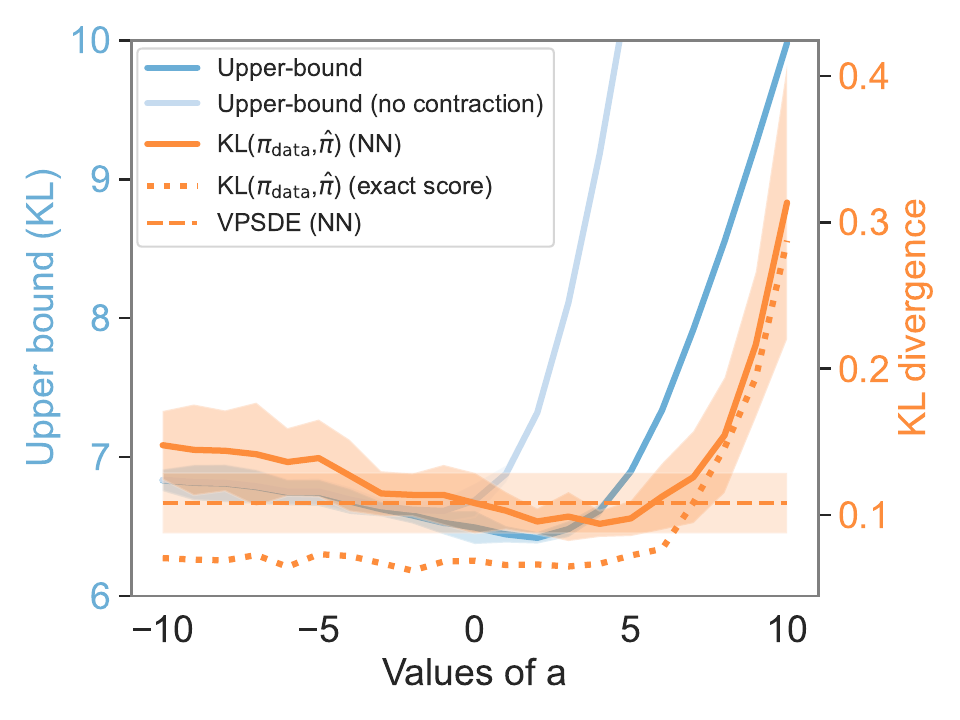}&
        \includegraphics[width=0.32\linewidth]{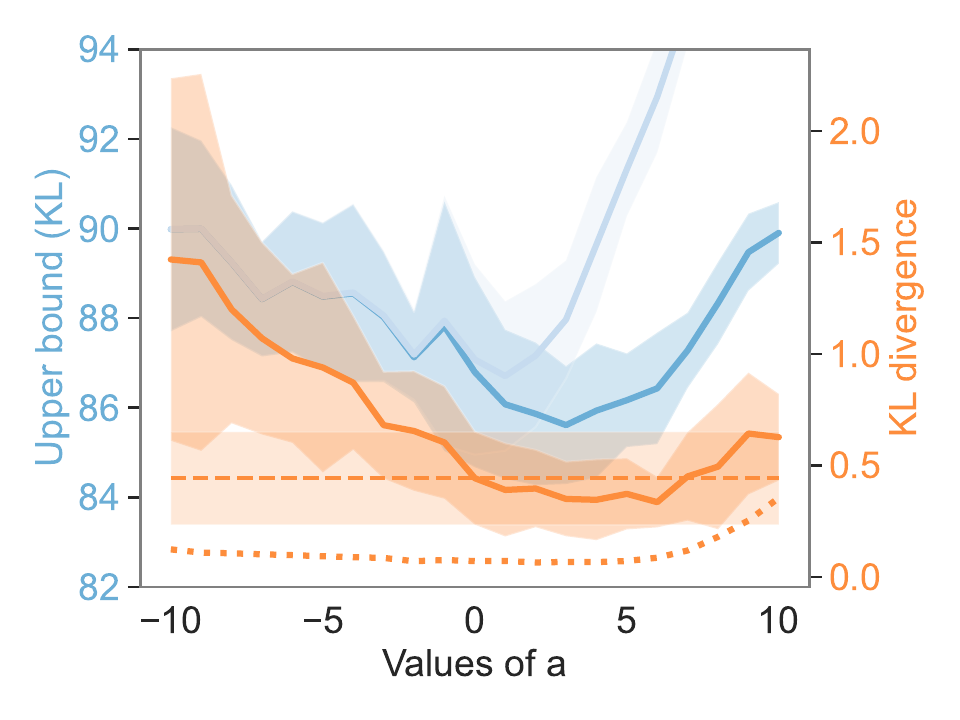}&
        \includegraphics[width=0.32\linewidth]{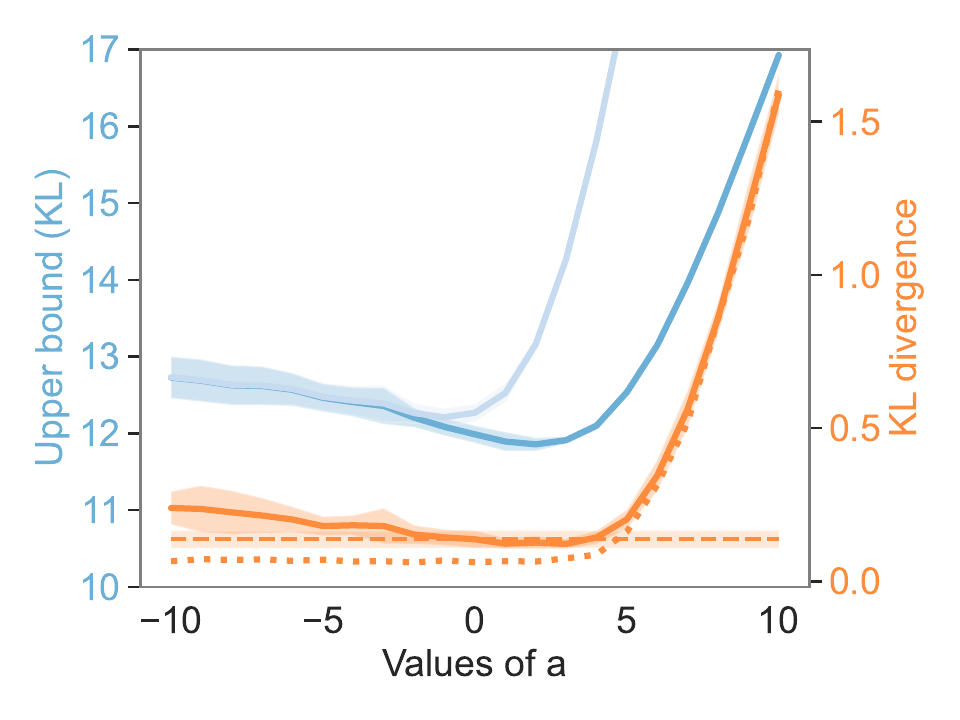} \\
        \hspace{-1cm}  
        \includegraphics[width=0.32\linewidth]{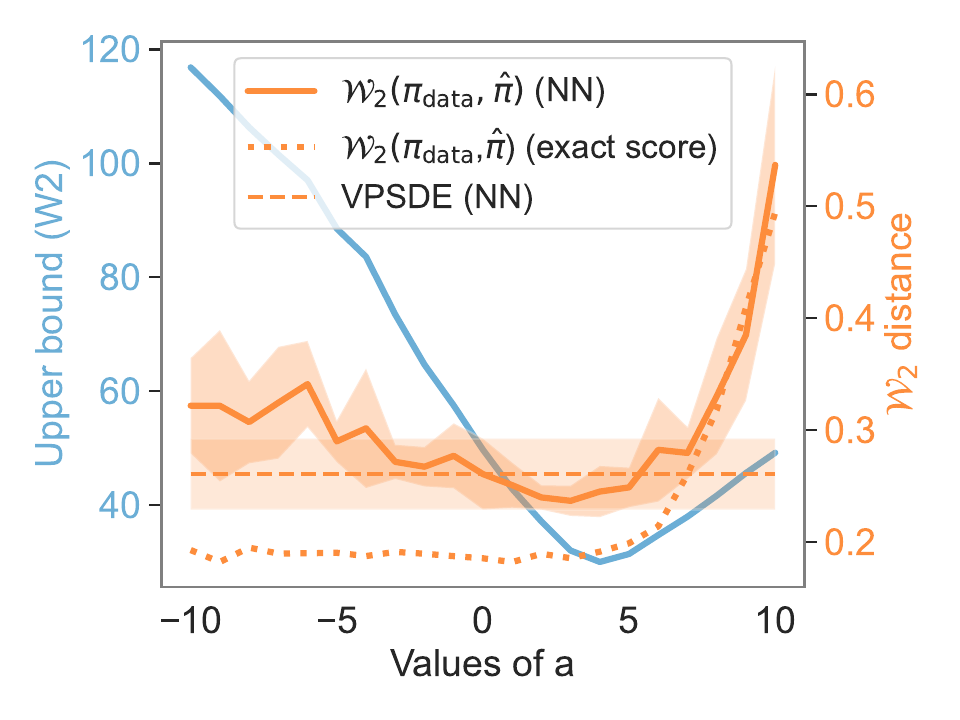}&
        \includegraphics[width=0.32\linewidth]{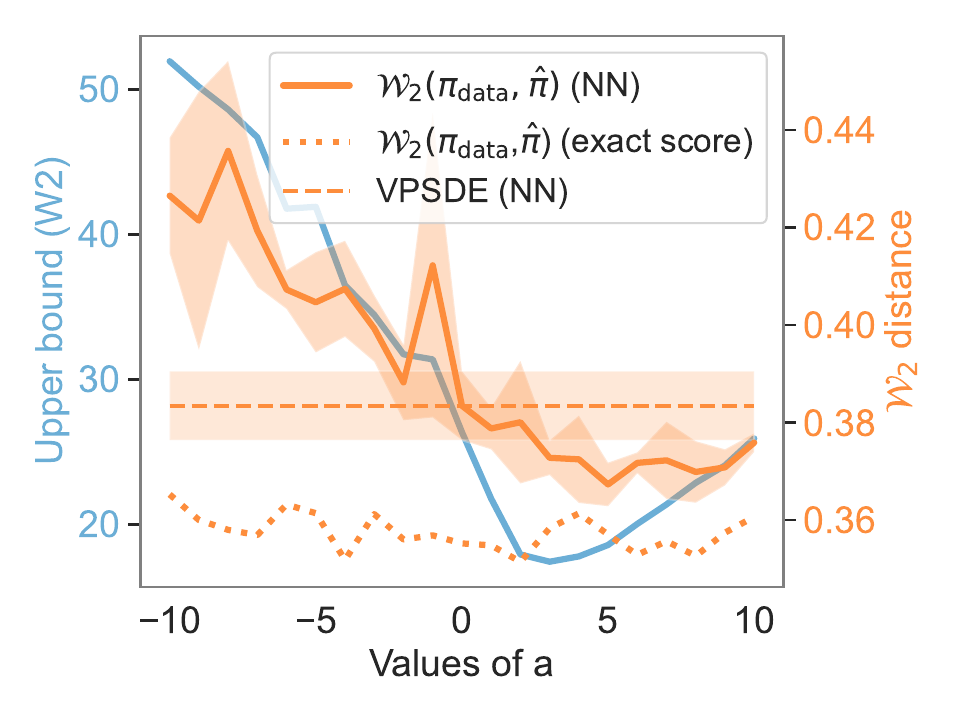}&
        \includegraphics[width=0.32\linewidth]{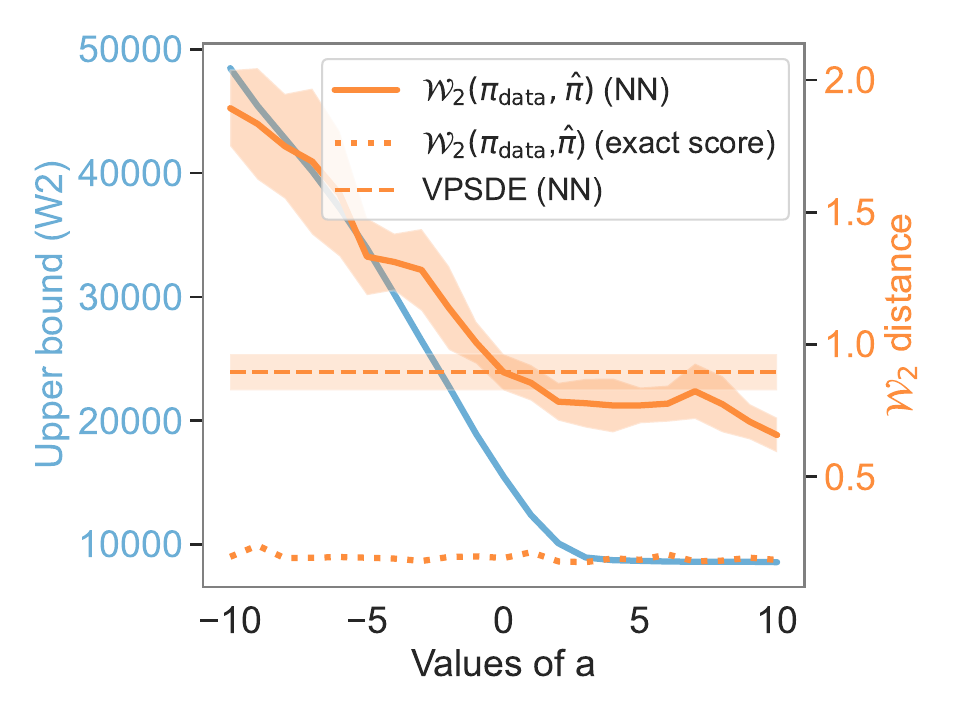} \\
        \hspace{-1cm} {(a) Isotropic setting 
        } & {(b) Heteroscedastic setting } & {(c) Correlated setting }
    \end{tabular}
    \caption{ \label{fig:fig2_using_EI}Comparison of the empirical KL divergence (mean ± std over 30 runs) (top) and $\mathcal{W}_2$  distance (mean ± std over 10 runs) (bottom)  between $\pi_{\mathrm{data}}$ and $\pihat$ (orange) and the related upper bounds (blue) from Theorem \ref{th:main} and Theorem \ref{thm:wasserstein_bound}  across parameter $a$ for noise schedule $\beta_a$, $d=50$ with \textbf{Exponential Integrator} discretization scheme. We also show the metrics for the linear VPSDE model (dashed line) and our model (dotted line) with exact score evaluation.}
    \end{figure}

\subsubsection{Implementation of the score approximation in the Gaussian setting} \label{NN_approx_model}

Although the score function is explicit when $\pi_{\rm data}$ is Gaussian (see Lemma~\ref{lem:exactscore}),  we implement SGMs as done in applications, i.e., we train a deep neural network to witness the effect of the noising function on the approximation error. We train a neural network architecture $s_{\param}(t,x) \in [0,T] \times \mathbb{R}^d \mapsto \mathbb{R}^d$ using the actual score function as a target: 
\begin{align*}
\mathcal{L}_{\rm explicit} (\theta )& = \mathbb{E} \left[ \left\|  s_{\theta} \left(\tau, \ora X_{\tau}\right) - \nabla \log p_{\tau} \left(\ora X_{\tau}\right)  \right\|^2 \right] \\
& =  \mathbb{E} \left[ \left\|  s_{\theta} \left(\tau, \ora X_{\tau}\right) -(m_{\tau}^2 \Sigma_0 + \sigma_{\tau}^2 \mathrm{I}_d)^{-1} (\ora X_{\tau} + m_{\tau} \mu_0)  \right\|^2 \right]\eqsp,
\end{align*}
where $t \to m_t$ and $t \to \sigma_t$ are defined in Lemma~\ref{lem:exactscore} and $\tau \sim \mathcal{U}(0,T)$ is independent of $\ora X$. The neural network architecture chosen for this task is described in Figure \ref{fig:nn_architecture}. The width of each dense layer $\texttt{mid\_features}$ is set to 256 throughout the experiments. 

\begin{figure}[h]
    \centering
    \includegraphics[width=\linewidth]{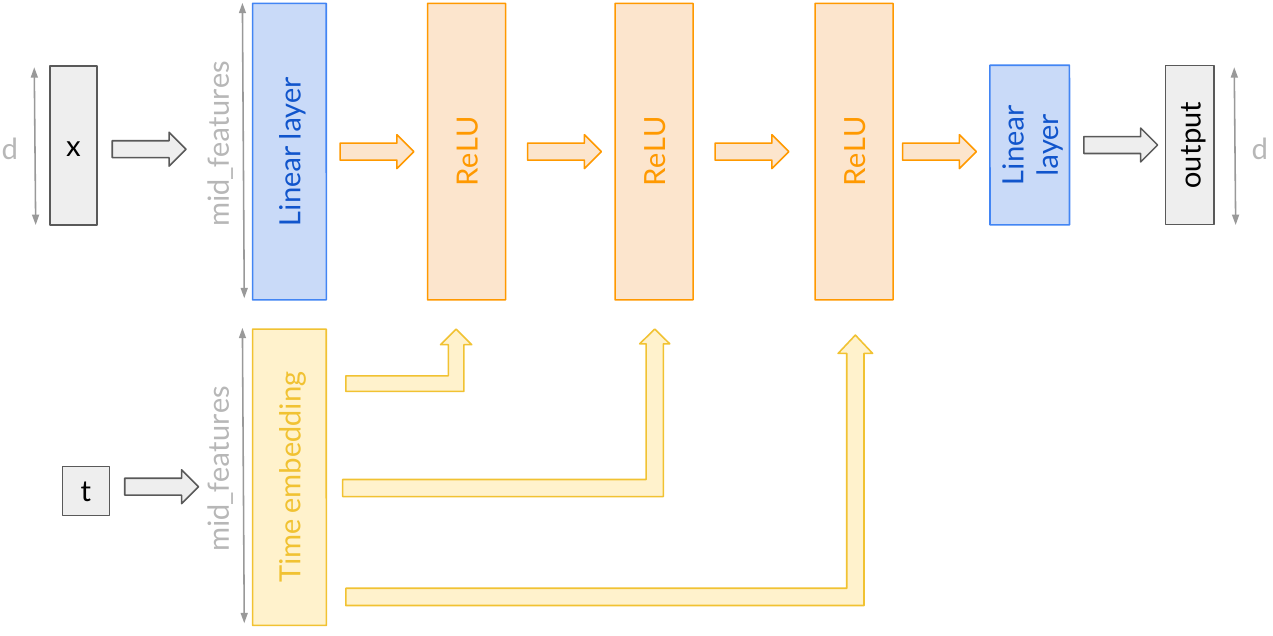}
    \caption{Neural network architecture. The input layer is composed of a vector $x$ in dimension $d$ and the time $t$. Both are respectively embedded using a linear transformation or a sine/cosine transformation \citep{nichol2021improved} of width $\texttt{mid\_features}$. Then, 3 dense layers of constant width $\texttt{mid\_features}$ followed by ReLu activations and skip connections regarding the time embedding. The output layer is linear resulting in a vector of dimension $d$. }
    \label{fig:nn_architecture}
\end{figure}

\subsection{Details on the experiments and additional results}

\subsubsection{Illustration of the KL bound in the Gaussian setting}
\label{sec:appendix_KL_bound}

\paragraph{Target distributions.}
We investigate the relevancy of the upper bound from Theorem \ref{th:main} for different noise schedules in the Gaussian setting. We use as a training sample $10^4$ samples with distribution $\mathcal{N} \left( {\bf 1}_d, \Sigma \right) $ for $d$ the dimension of the target distribution with different choices of covariance structure.

\begin{enumerate}
    \item (Isotropic) $\Sigma^{\mathrm{iso}} = 0.5 \mathrm{I}_d$.
    \item (Heteroscedastic) $\Sigma^{\mathrm{heterosc}}\in\mathbb{R}^{d\times d}$ is a diagonal matrix such that $\Sigma^{\mathrm{heterosc}}_{jj}=1$ for $1\leq j \leq d$, and $\Sigma^{\mathrm{heterosc}}_{jj}=0.01$ otherwise.
    \item (Correlated) $\Sigma^{\mathrm{corr}}\in\mathbb{R}^{d\times d}$ is a full matrix whose diagonal entries are equal to one and the off-diagonal terms are given by  $\Sigma^{\mathrm{corr}}_{jj'}=1/\sqrt{|j-j'|}$ for $1\leq j\neq j' \leq d$.
\end{enumerate}
The resulting data distributions are respectively denoted by $\pi_{\rm data}^{\mathrm{(iso)}}$, $\pi_{\rm data}^{\mathrm{(heterosc)}}$ and $\pi_{\rm data}^{\mathrm{(corr)}}$.

\paragraph{Upper bound evaluation.}We leverage the Gaussian nature of the target distribution to compute explicitly all the terms in the bound. On the one hand, the relative entropy in $\mathcal{E}^{\rm KL}_1$, $\kl \left( \pi_{\rm data} \middle\| \pi_{\infty} \right) $ is computed using the analytical formula for $\kl$-divergence between two random Gaussian variable (Lemma \ref{KL divergence Gaussian rvs}). On the other, the relative Fisher information in $\mathcal{E}^{\rm KL}_3$,  $\mathcal{I}(\pi_{\rm data} | \pi_{\infty} )$, is computed using Lemma (\ref{Info Fisher Gaussian rvs}). Moreover, as the noise schedule function $\beta_a$ and its primitive are analytically known, every occurrences of either of them are explicitly computed. Finally, it remains to estimate the expectation in $\mathcal{E}^{\rm KL}_2(\param,\sched)$. This is done via Monte Carlo estimation on 500 samples from the forward process (see Section \ref{sec:exact_simulation_forward})  for every step forward:
\begin{align*}
    \frac{1}{500} \sum_{k = 0}^{N-1} \sum_{i = 1}^{500} \left\| \nabla \log \tilde p_{T-t_k} \left( \ora X_{T-t_k}^{(i)} \right) - \tilde s_{\param}\left(T-t_k,  \ora X_{T-t_k}^{(i)}\right) \right\|^2 
    \int^{T-t_k}_{T-t_{k+1}} \sched_a (t) \rmd t\eqsp. 
\end{align*}

\paragraph{SGM data generation in dimension 50.} In Figures \ref{fig:gaussian_xp} (top) of the main paper, we represent the following quantities in the same graph, in dimension $d =50$, for different values of $a$.
\begin{itemize}
    \item In blue the upper bound from Theorem \ref{th:main}. The dark blue color is used to refer to the upper bound with the contraction argument in equation \eqref{eq:KL_refined_mixing_error} from Proposition \ref{prop:back_contractivity_kl} while the ligther blue bound is the same bound without the contraction argument.
    \item In orange (dotted line) the KL divergence between the target distribution $\pi_{\rm data}$ and the empirical mean and covariance of the data generated using the true score function from Lemma \ref{lem:exactscore}.
    \item In orange (plain line) we represent $\mathrm{KL}(\pi_{\rm data}\| \widehat{\pi}_N^{(\beta_a, \theta)})$ for $a \in \{ -10,-9, -8,\ldots ,10 \}$. That is, the KL divergence between the target distribution $\pi_{\rm data}$ and the empirical mean and covariance of the data generated using the neural network architecture described in Figure \ref{fig:nn_architecture} to approximate the score function.
    \item In orange (dashed line)  we represent $\mathrm{KL}(\pi_{\rm data}\| \widehat{\pi}_N^{(\beta_0, \theta)})$. That is, the KL divergence between the target data $\pi_{\rm data}$ and the empirical mean and covariance of the data generated by the linear schedule VPSDE presented in \citet{song2021score} with the neural network architecture described in Figure \ref{fig:nn_architecture}.
\end{itemize}

We generate 10 000 samples. The batch size is set to 64 and neural networks are optimized with Adam. All the KL divergences written above are computed using Lemma \ref{KL divergence Gaussian rvs}. Due to the stochastic nature of our experiments, they are repeated \textcolor{blue}{30} times so that the corresponding mean value and standard deviations of these results are respectively depicted using plain and fill-in-between plots. 

To disentangle the effect of each error term it is possible to plot the mixing time error $\mathcal{E}_1^{\rm KL}(\sched)$, the approximation error $\mathcal{E}_2^{\rm KL}(\param,\sched)$ and the discretization error $\mathcal{E}_3^{\rm KL}(\sched)$ on a same graph for different values of $a$. However, for the schedule choice presented in Figure \ref{fig:noise_schedules} as $\beta(T)$ is set to be $20$ for every $a$ values it is pointless to display $\mathcal{E}_3^{\rm KL}(\sched)$ as it would not vary for different choices of schedule from $\beta_a$. The three error terms for Theorem \ref{th:main}, corresponding to the example in Figure \ref{fig:gaussian_xp} (top) are provided below in Figure \ref{fig:error_decomposition}.

     \begin{figure}[h!]
    \centering
    \begin{tabular}{ccc}
    \hspace{-1cm}
        \includegraphics[width=0.32\linewidth]{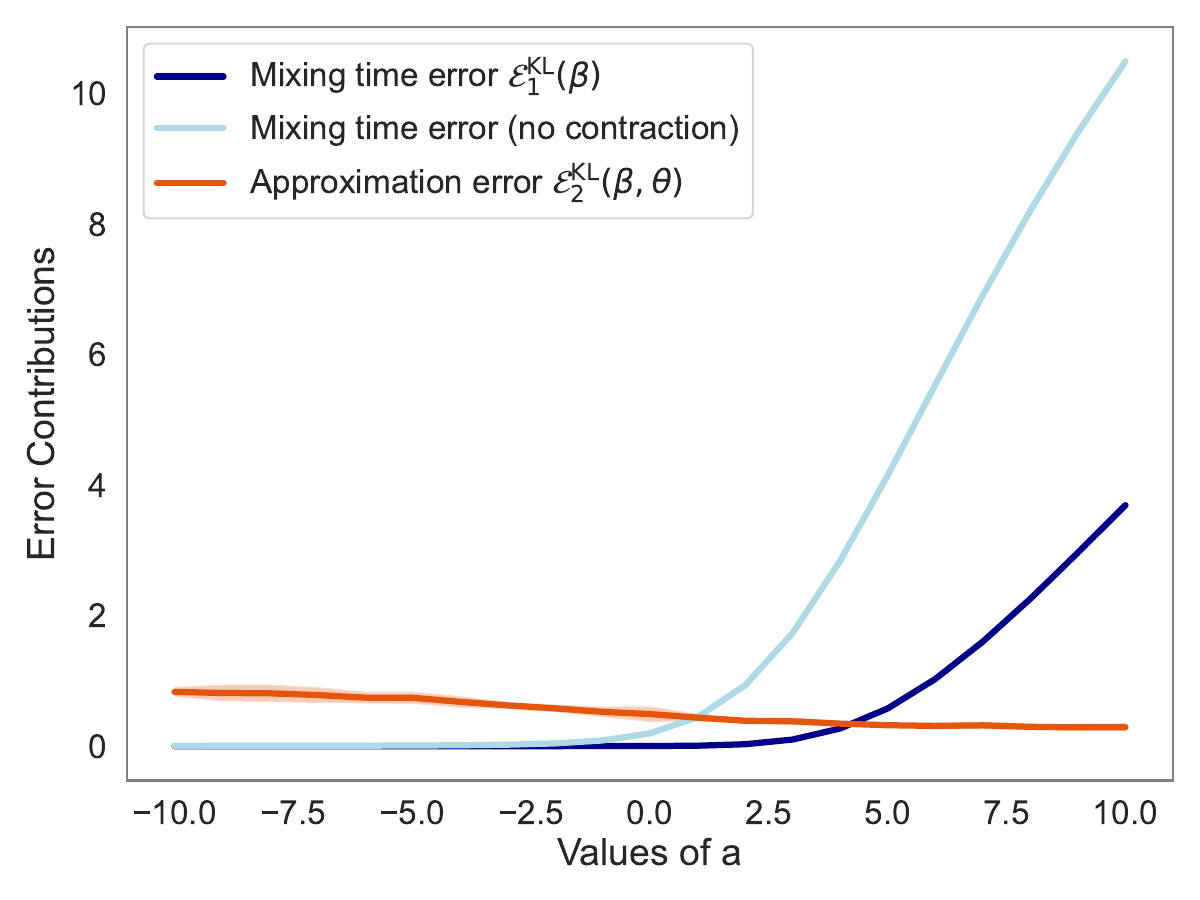}&
        \includegraphics[width=0.32\linewidth]{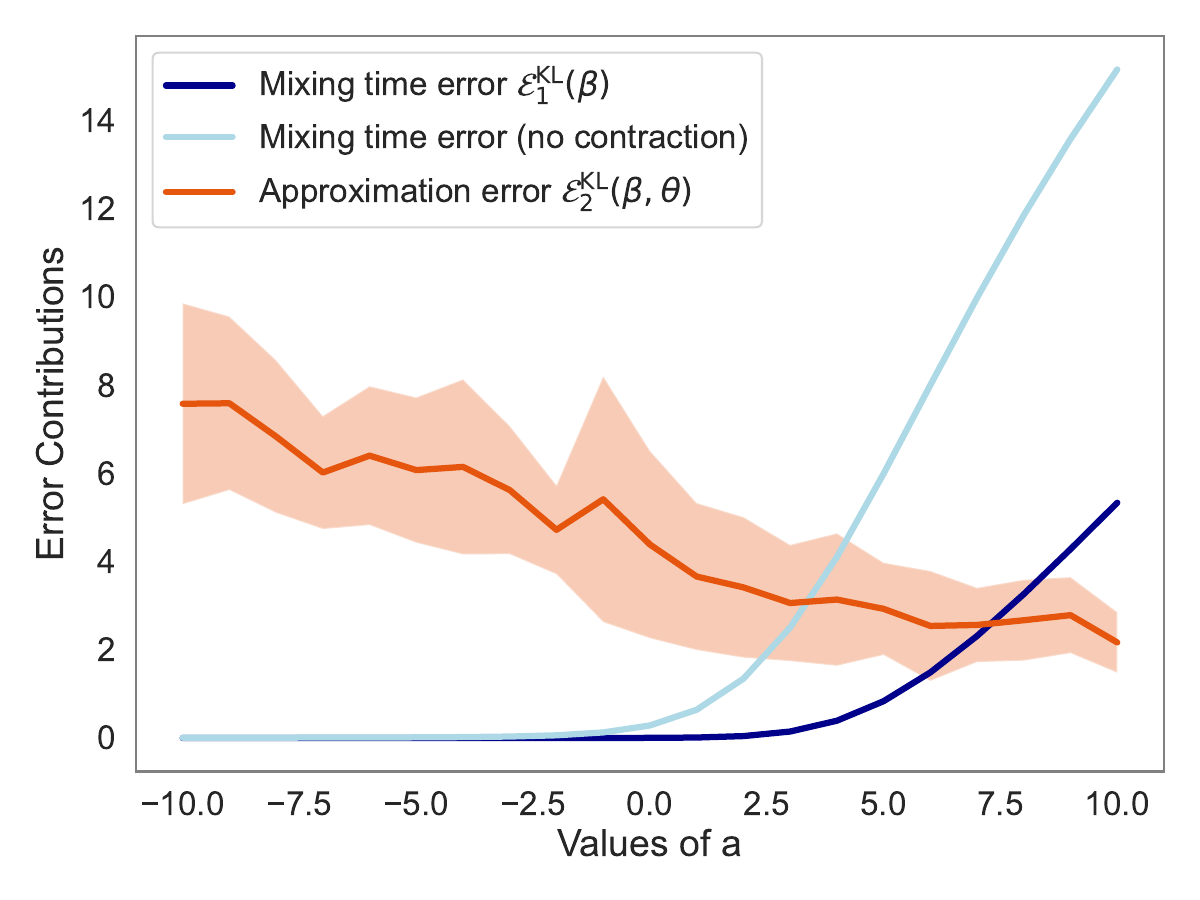}&
        \includegraphics[width=0.32\linewidth]{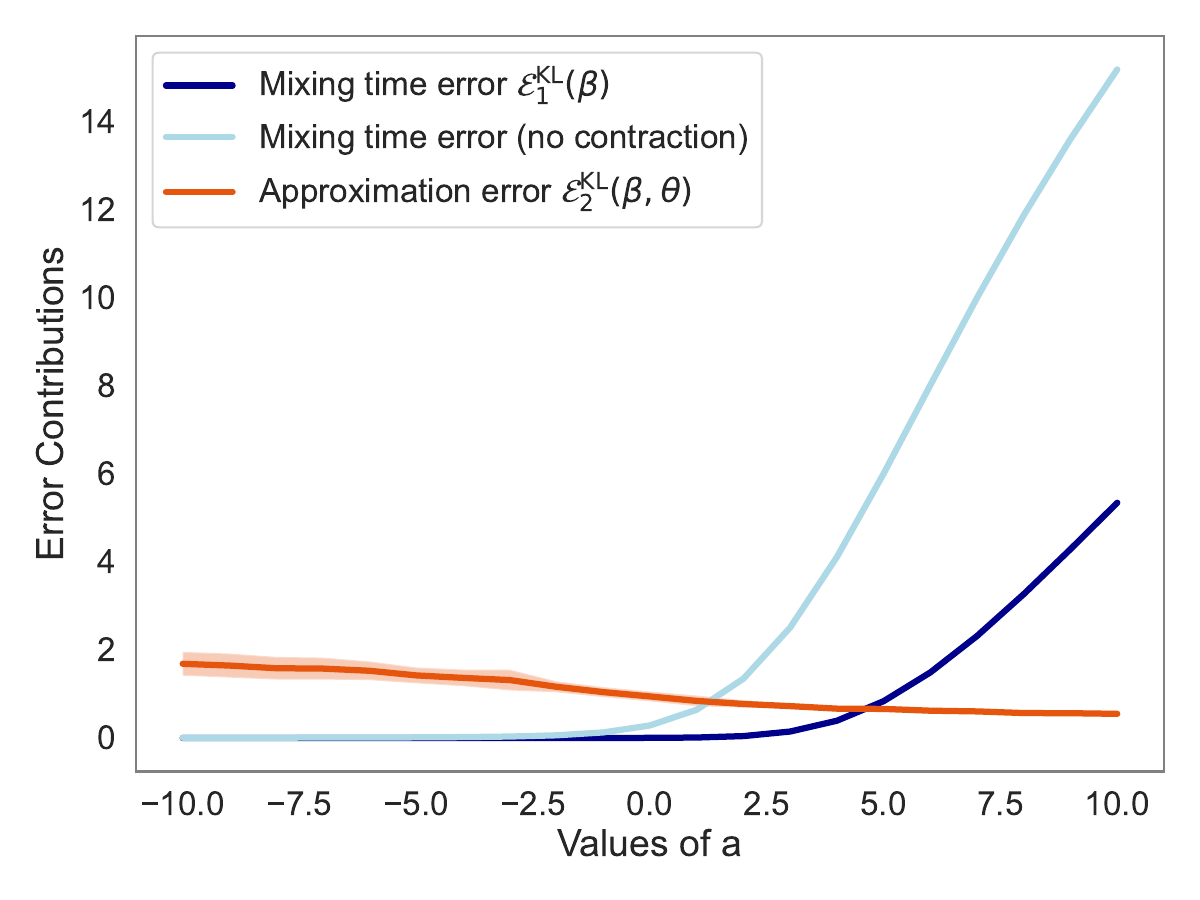} \\
        \hspace{-1cm} {(a) Isotropic setting 
        } & {(b) Heteroscedastic setting } & {(c) Correlated setting }
    \end{tabular}
    \caption{\label{fig:error_decomposition}Error terms contribution from Theorem \ref{th:main} displayed from the same examples as in Figure \ref{fig:gaussian_xp} (top).}
    \end{figure}

\paragraph{Optimal schedule versus classical choices.}
We investigate the gain from using the parametric schedule with $a^\star
$ minimising the upper bound from Theorem \ref{th:main} for $d \in \{ 5, 10, 25, 50 \}$ compared to using the linear and cosine schedules (see Appendix \ref{sec:noise_sched}) 
in the isotropic and correlated settings (as mentioned in Section \ref{sec:exp_gaussian}, up to rescaling the heteroscedastic setting boils down to an isotropic setting).

To determine the optimal value $a^\star$, upper bounds were initially calculated across various dimensions for a range of $a$ values from $\{-10, -9, \ldots, 10\}$. This initial calculation aimed to identify a preliminary minimum value. Subsequently, the search was refined around these preliminary values using finer step-sizes of 0.25 to more precisely locate $a^\star$. 

Results are given in tabular form in Table \ref{tab:KL_table_gaussian} and in Figure \ref{fig:gaussian_xp_dim}. The parametric schedule optimized to minimize the upper bound $\beta_{a^\star}$ consistently surpasses the linear schedule, delivering significant improvements. This enhanced performance is shown by lower average Kullback-Leibler divergence between $\pi_{\rm data}$ and the generated sample distribution, as well as a reduction in the standard deviation of these divergences, which contributes to more stable generation. These results are competitive with or even exceed those obtained with state-of-the-art schedules such as the cosine schedule, particularly in higher dimensions $d = 25$ and $d = 50$. However, one should note that this comparison may not be entirely fair, as the cosine schedule increases unboundedly near $T$, whereas we capped the parametric schedule at $\beta(T) = 20$ to align with the linear schedule described in \citet{song2021score}.

\begin{table}[ht]
\centering
\resizebox{\textwidth}{!}{\begin{tabular}{@{}cccccc@{}} 
\toprule
 & Dimension & 5 & 10 & 25 & 50\\ 
\midrule
\multirow{6}{*}{Isotropic} & Upper bound min $a^\star$ & 1.75 & 1.00 & 1.50 & 2.00 \\
& Generation value in $a^\star$ & 0.001607 $\pm$ 0.000462 & 0.005343 $\pm$ 0.001155 & \textbf{0.026724} $\pm$ 0.004046 & \textbf{0.095981} $\pm$ 0.005485 \\
& VPSDE (linear sched.) & 0.001935 $\pm$ 0.000405 & 0.005594 $\pm$ 0.001377 & 0.031748 $\pm$ 0.006158 & 0.105592 $\pm$ 0.019529 \\
& Cosine schedule & \textbf{0.001390} $\pm$ 0.000296 & \textbf{0.005097} $\pm$ 0.001064 & 0.026900 $\pm$ 0.001859 & 0.099917 $\pm$ 0.004375 \\
& \% gain (vs VPSDE) & +16.93 \% & +4.48 \% & +15.80 \% & +9.10 \% \\
& \% gain (vs Cosine) & -15.61 \% & -4.83 \% & +0.66 \% & +3.94 \% \\
\midrule
\multirow{6}{*}{Correlated} & Upper bound min $a^\star$ & 2.25 & 1.75 & 1.75 & 2.25 \\
& Generation value in $a^\star$ & 0.001861 $\pm$ 0.000880 & 0.005871 $\pm$ 0.001165 & \textbf{0.033156} $\pm$ 0.003785 & \textbf{0.109649} $\pm$ 0.008056 \\
& VPSDE (linear sched.) & 0.002568 $\pm$ 0.002708 & 0.006210 $\pm$ 0.001816 & 0.038434 $\pm$ 0.010313 & 0.134716 $\pm$ 0.016541 \\
& Cosine schedule & \textbf{0.001197} $\pm$ 0.000332 & \textbf{0.005515} $\pm$ 0.000775 & 0.040430 $\pm$ 0.003475 & 0.110515 $\pm$ 0.004646 \\
& \% gain (vs VPSDE) & +27.53 \% & +5.46 \% & +13.74 \% & +18.63 \% \\
& \% gain (vs Cosine) & -55.47 \% & -6.46 \% & +17.98 \% & +0.78 \% \\
\midrule
\multirow{2}{*}{Parameters} & Learning rate & 1e-4 & 1e-4 & 1e-3 & 1e-3 \\
& Epochs & 20 & 30 & 75 & 150 \\
\bottomrule \\
\end{tabular}}
\caption{
\label{tab:KL_table_gaussian}
Comparison of the KL divergence between the target value and the generated value at $a^{\star}$ (the minimum value of the upper bound from Theorem \ref{th:main}) with the KL divergence between the generated value by VPSDE with linear schedule and the target distribution. We display average KL divergences plus or minus standard deviations over 10 runs. The target distributions are chosen to be Gaussian with different covariance structures:  isotropic ($\pi_{\rm data}^{\mathrm{(iso)}}$), heteroscedastic ($\pi_{\rm data}^{\mathrm{(heterosc)}} $) and correlated ($\pi_{\rm data}^{\mathrm{(corr)}}$).  }
\end{table}

\begin{figure*}[h!]
    \centering
    \begin{tabular}{cc}
    \hspace{-1cm}
        \includegraphics[width=0.4\linewidth]{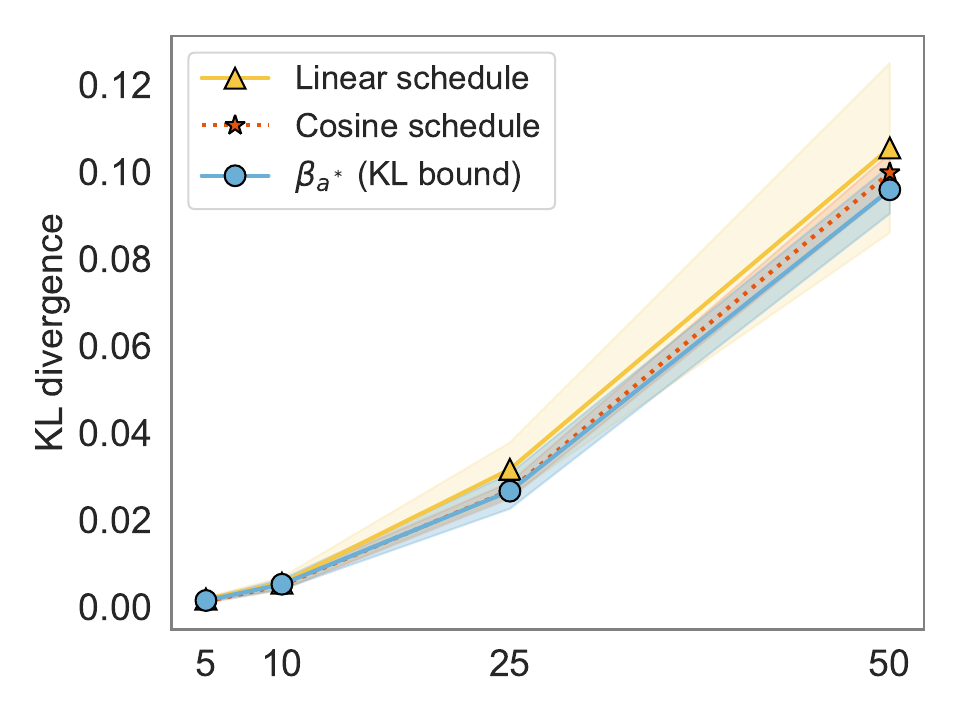}&
        \includegraphics[width=0.4\linewidth]{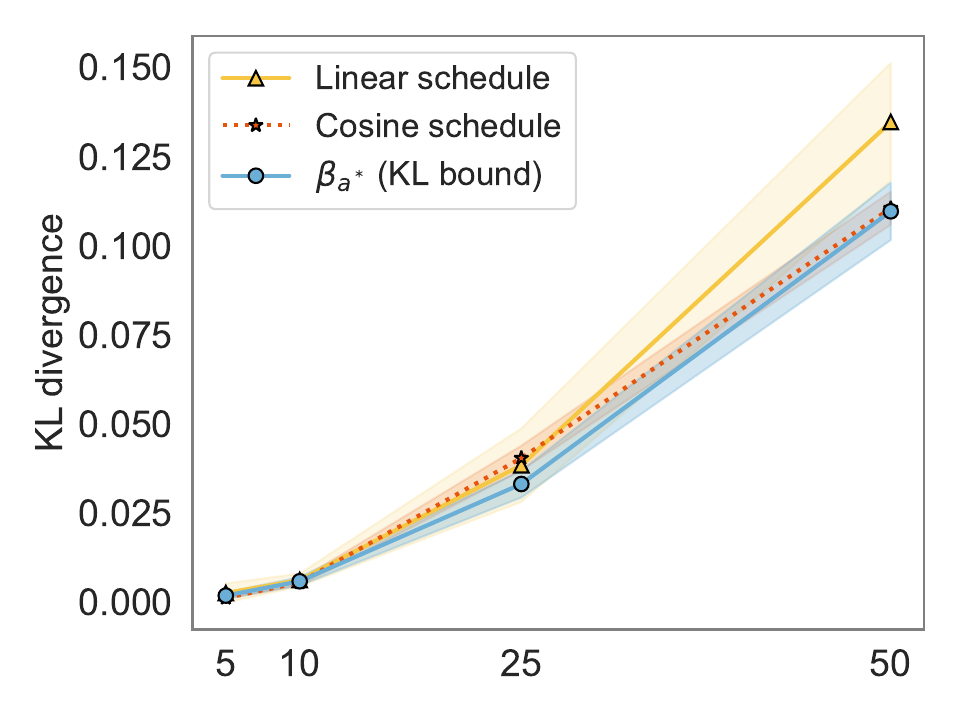} \\
        \hspace{-1cm}
        {(a) Isotropic setting $\pi_{\rm data}^{\mathrm{(iso)}}$} 
        & {(b) Correlated setting $\pi_{\rm data}^{\mathrm{(corr)}}$}
    \end{tabular}
    \caption{\label{fig:gaussian_xp_dim} Comparison of the empirical KL divergence (mean value $\pm$ std over 10 runs) between $\pi_{\mathrm{data}}$ and the generative distribution $\hat{\pi}$ for different values of the dimension.
    The generative distributions considered are $ \widehat{\pi}_N^{(\beta_{a^\star}, \theta)})$ obtained by the time-inhomogeneous SGM for $\beta_{a^\star}$ (blue plain), $ \widehat{\pi}_N^{(\beta_{0}, \theta)}$ obtained by a standard linear VPSDE model (yellow dashed) and $ \widehat{\pi}_N^{(\beta_{\cos}, \theta)}$ obtained by using a cosine schedule (orange dotted).   
    }
\end{figure*}

\subsubsection{Illustration of the Wasserstein bound in the Gaussian setting}
\label{sec:appendix_W2_bound}

\paragraph{Target distributions.} The target distributions are Gaussian and are the same as for the the Kullback-Leibler bound:  $\pi_{\rm data}^{\mathrm{(iso)}}$, $\pi_{\rm data}^{\mathrm{(heterosc)}}$ and $\pi_{\rm data}^{\mathrm{(corr)}}$.  

\paragraph{Upper bound evaluation.} We leverage the Gaussian nature of the target distribution to compute
explicitly all the terms in the bound from Theroem \ref{thm:wasserstein_bound}. For the mixing time $\mathcal{E}_1^{\mathcal{W}_2}$, the strong log-concavity constant $\bar C_t$ is derived using Lemma \ref{lem:gaussian:contract} and $\mathcal{W}_2 (\pi_{\rm data} , \pi_{\infty})$ is derived using Lemma \ref{W2 Gaussian rvs}. For $\mathcal{E}_2^{\mathcal{W}_2}$, the analytical expressions for $\bar L_t$ is given in Lemma \ref{lem:gaussian:contract} and an upper bound to $M$ is derived in Proposition \ref{prop:bar_l_t}. All non analytically solvable integrals estimated numerically using the trapezoidal rule, implemented with the built-in PyTorch function \texttt{torch.trapezoid}. To estimate $\varepsilon$, we use Monte-Carlo simulations with 500 samples (in the same manner as for the Kullback-Leibler bound):
\begin{align*}
    \sup_{k \in \{0, ..., N-1 \} } \sqrt{ \frac{1}{500} \sum_{i = 1}^{500} \left\| \nabla \log \tilde p_{T-t_k} \left( \ora X_{T-t_k}^{(i)} \right) - \tilde s_{\param}\left(T-t_k,  \ora X_{T-t_k}^{(i)}\right) \right\|^2  }
    \eqsp. 
\end{align*}

\paragraph{SGM data generation dimension 50.}
In Figures \ref{fig:W_2_result_gaussian} (and Figures \ref{fig:gaussian_xp} (bottom) of the main paper) we represent on the same graph, in dimension $d = 50$, for different values of a:
\begin{itemize}
    \item in blue the upper bound from Theorem \ref{thm:wasserstein_bound}.
    \item in orange (dotted line) the $\mathcal{W}_2$ distance between the target distribution $\pi_{\rm data}$ and the empirical mean and covariance of the data generated using the true score function from Lemma \ref{lem:exactscore}.
    \item in orange (plain line) we represent $\mathcal{W}_2(\pi_{\rm data} , \widehat{\pi}_N^{(\beta_a, \theta)})$ for $a \in \{ -10,-9, -8,.. ,10 \}$. That is, the 
    $\mathcal{W}_2$ distance between the target distribution $\pi_{\rm data}$ and the empirical mean and covariance of the data generated using the neural network architecture described in Figure \ref{fig:nn_architecture} to approximate the score function 
    \item in orange (dashed line)  we reprensent $\mathcal{W}_2(\pi_{\rm data} , \widehat{\pi}_N^{(\beta_0, \theta)})$. That is, the $\mathcal{W}_2$ distance between the target data $\pi_{\rm data}$ and the empirical mean and covariance of the data generated by the linear schedule VPSDE presented in \citet{song2021score} with the neural network architecture described in Figure \ref{fig:nn_architecture}.
\end{itemize}

\underline{First \phantom{p}results.} We generate 10 000 samples. The batch size is set to 64 and neural networks are optimized with Adam. All the $\mathcal{W}_2$ distances written above are computed using Lemma \ref{W2 Gaussian rvs}. Due to the stochastic nature of our experiments, they are repeated ten times so that the corresponding mean value and standard deviations of these results are respectively depicted using plain and fill-in-between plots.

\begin{figure*}[h] 
    \centering
    \begin{tabular}{ccc}
        \includegraphics[width=0.32\linewidth]{w2_iso_main.pdf}&
        \includegraphics[width=0.32\linewidth]{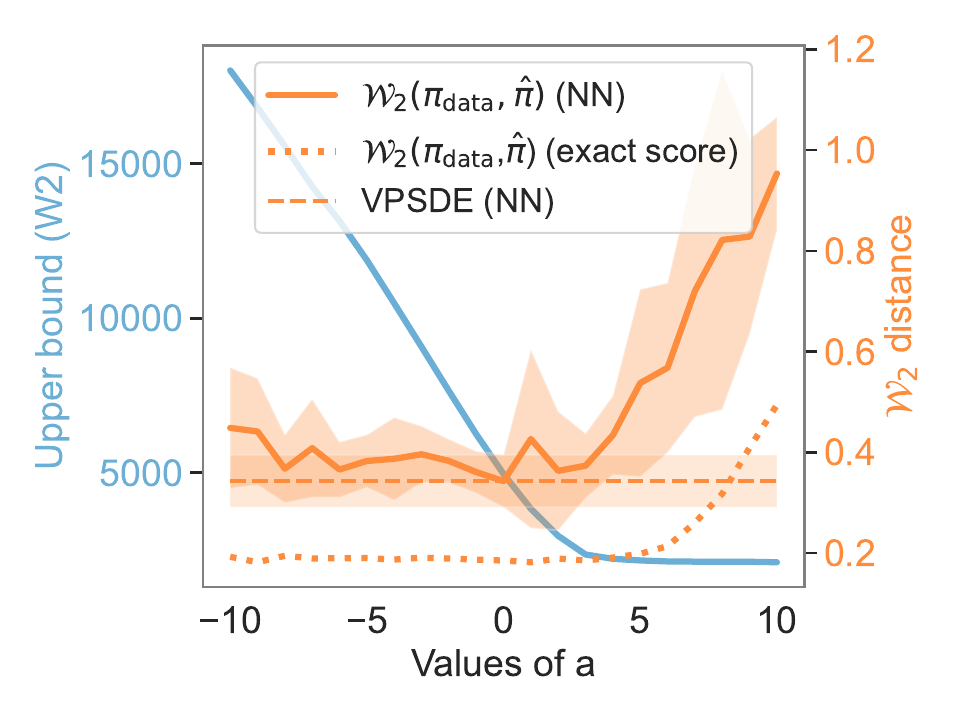}&
        \includegraphics[width=0.32\linewidth]{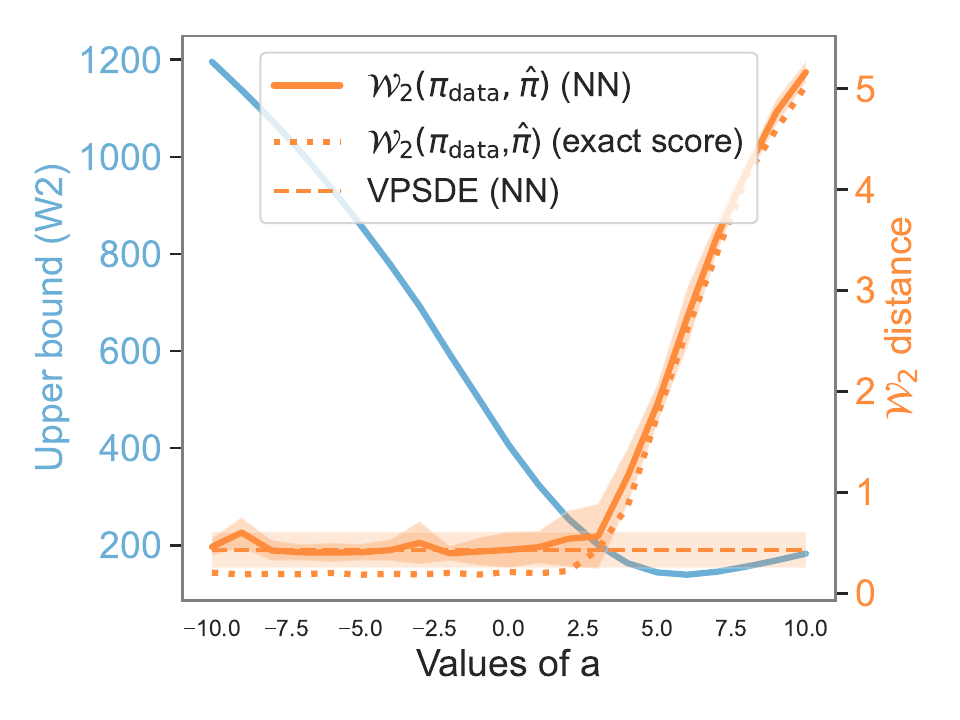} 
        \\
        {(a) Isotropic setting } & {(b) Heteroscedastic setting} & {(c) Correlated setting}
    \end{tabular}
    \caption{\label{fig:W_2_result_gaussian} Comparison of the empirical 2-Wasserstein distance (mean value $\pm$ std over 10 runs) between $\pi_{\mathrm{data}}$ and $ \widehat{\pi}_N^{(\beta_a, \theta)}$ (in orange) and the upper bound from Theorem \ref{thm:wasserstein_bound} (in blue) w.r.t.\ the parameter $a$ used in the definition of the noise schedule $\beta_a$, for $d=50$. We also represent the 2-Wasserstein distances obtained with the linear VPSDE model (dashed line) and the one obtained with the parametric model (dotted line) when the score is not approximated but exactly evaluated. The data distribution $\pi_{\mathrm{data}}$ is chosen Gaussian, corresponding to  (a) $\pi_{\rm data}^{\mathrm{(iso)}}$, (b) $\pi_{\rm data}^{\mathrm{(heterosc)}} $ and (c) $\pi_{\rm data}^{\mathrm{(corr)}}$. 
    }
\end{figure*}

Performances obtained from raw distributions for $\pi_{\rm data}^{\mathrm{(iso)}}$, $\pi_{\rm data}^{\mathrm{(heterosc)}}$ and $\pi_{\rm data}^{\mathrm{(corr)}}$ are displayed in Figure \ref{fig:W_2_result_gaussian}. In the isotropic case (Figure \ref{fig:W_2_result_gaussian} (a)) the curve for the upper bound (blue line) points a global minimum near the minimal values obtain by $\mathcal{W}_2(\pi_{\rm data}, \widehat{\pi}_N^{(\beta_a, \theta)})$ (plain orange line), which underlines that the upper bound is indeed informative in such a case.
However, the upper bounds obtained for the heteroscedastic and correlated settings (Figure \ref{fig:W_2_result_gaussian} (b,c)) are not in line with the generation results.

These observed discrepancies can be linked to the conditioning of the covariance matrices. In both heteroscedastic and correlated cases, the largest eigenvalue of the covariance matrices is not smaller than the variance stationary distribution of the forward process (set to $\sigma^2 = 1$ in those experiments) violating the requirements of Lemma \ref{lem:gaussian:contract} ($\lambda_{\max} \left( \Sigma^{\mathrm{(heterosc)}}\right) = 1$ and $\lambda_{\max} \left( \Sigma^{\mathrm{(corr)}}\right) \approx 15$). 
This induces the default of strong log-concavity of the renormalized densities $\tilde{p}_t$.
In this way, the Gaussian scenario highlights the critical influence of the covariance matrix conditioning on SGMs. 
Additionally, a smaller $\lambda_{\min} \left( \Sigma \right)$ would increase $L_t$ and $M$, which in turn would increase the bound from Theorem \ref{thm:wasserstein_bound}. 

\underline{Data preprocessing.} As frequently done in practice, we expect better conditioning by running SGMs on a standardized distribution. In this way, note that if $X_0 \sim \pi_{\rm data}$ we consider the centered standardized distribution $X_{\rm stand} = D \left( X_0 - \mu \right)$ with  $D=\mathrm{diag}(\sigma_1, \hdots, \sigma_d) \in \mathbb{R}^{d\times d}$ a diagonal matrix with diagonal entries $\sigma_j$ corresponding to the standard deviation of the $j$-th component of $X_0$ and with $\mu = \left[ \mathbb{E} \left[ X_{0,1} \right], ..., \mathbb{E} \left[ X_{0,d} \right] \right]^\top$.  
A last transformation shrinks the data into a rescaled version of $X_{\rm stand}$ defined as $X_{\rm scale} = \kappa D \left( X_0 - \mu \right)$ with $\kappa \eqdef 1/ (2 \lambda_{\max} \left( \Sigma_{ \rm (stand)}\right))^{1/2}$, where $\lambda_{\max} \left( \Sigma_{ \rm (stand)}\right)$ is the largest eigenvalue of the covariance matrix of $X_{\rm stand}$. 
We then train SGMs to approximate the distribution of $X_{\rm scale}$.
By doing so we ensure the applicability of Lemma \ref{lem:gaussian:contract} (with $\sigma^2 = 1$), as the largest eigenvalue of the covariance matrix of $X_{\rm scale}$ is no larger than 0.5.

\underline{Adapted upper bound.} We can finally adapt the upper bound from Theorem \ref{thm:wasserstein_bound} to a rescaled setting by noting that
\begin{equation}
\label{eq:w2_bound_recale}
\mathcal{W}_2 \left( \pi_{\rm data} , \tilde{\pi} \right) \leq \frac{1}{\kappa} \left( \max_{1\leq j \leq d} \sigma_j \right) 
\mathcal{W}_2 \left( \pi_{\rm scale} , \widehat{\pi}_{N, \text{scale}}^{(\beta_a, \theta)} \right)
\eqsp , 
\end{equation}
where 
\begin{itemize}
\item $\pi_{\mathrm{scale}}$ is the distribution of scaled sample $X_{\rm scale}$
\item $\widehat{\pi}_{N, \text{scale}}^{(\beta_a, \theta)}$ corresponds to the distribution of SGM trained on $X_{\rm scale}$ 
\item $\tilde{\pi}$ is the distribution of the descaled generated samples, i.e., the distribution of $ D^{-1} X/\kappa +\mu$ with $X\sim \widehat{\pi}_{N, \text{scale}}^{(\beta_a, \theta)}$.  
\end{itemize}
Therefore, we can evaluate the upper bound of Theorem \ref{thm:wasserstein_bound} for scaled samples (r.h.s.\ of \eqref{eq:w2_bound_recale}), and transfer it up to a constant to descaled generated samples (l.h.s.\ of \eqref{eq:w2_bound_recale}).

\underline{Results with scaled data preprocessing.} The results are detailed in Figure \ref{fig:W_2_result_gaussian_rescale} for the heteroscedastic case (e) $\pi_{\rm scale}^{\mathrm{(heterosc)}}$ and the correlated case (f) $\pi_{\rm scale}^{\mathrm{(corr)}}$, and are discussed extensively in Section \ref{sec:exp_gaussian} of the main paper. Note that the minima of the evaluated bounds now align closely with the empirical metrics. However, the upper bound profile for the correlated case has been shifted up. This increase was anticipated due to the effect of rescaling by the largest eigenvalue of $\Sigma_{\rm (stand)}^{\rm corr}$, approximately 15, which reduces the magnitude of the values in $\pi_{\rm scale}^{\mathrm{(corr)}}$. This tends to increase the values of $L_t$ and $M$ through the effect on $\lambda_{\min}( \Sigma^{\rm (corr)}_{\rm (stand)} )$ as explained above. Despite this effect, 
these experiments confirm the overall utility of the bound for selecting the appropriate noise schedule. The effect of data rescaling on the Lispschitz continuity and log concavity of the true score function $\nabla \log p_t$ are illustrated in Figure \ref{fig:lip_concav_hetero} on the Heteroscedastic setting. 

\begin{figure*}[h]
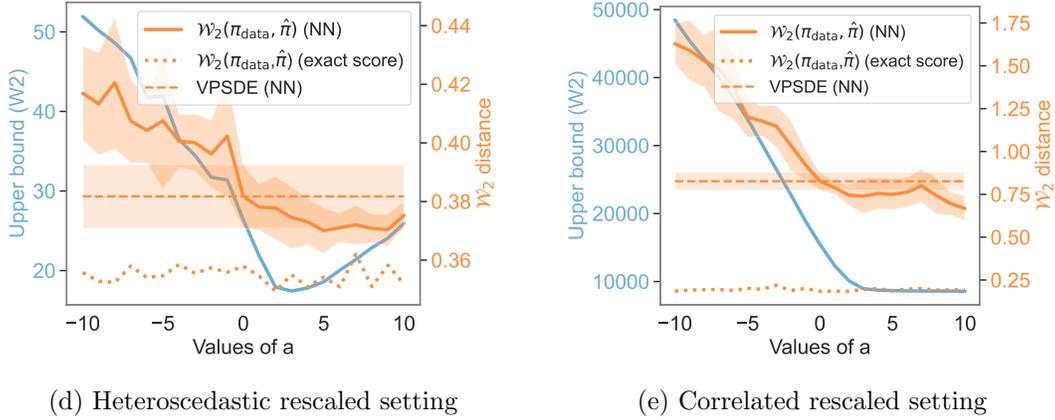
 
    \centering
    \begin{tabular}{cc}
        \includegraphics[width=0.45\linewidth]{w2_aniso_main.pdf}
        &
        \includegraphics[width=0.45\linewidth]{w2_covar_main.pdf}
        \\
        {(d) Heteroscedastic rescaled setting} & {(e) Correlated rescaled setting}
    \end{tabular}
    \caption{\label{fig:W_2_result_gaussian_rescale} Comparison of the empirical 2-Wasserstein distance on rescaled datasets for (d) $\pi_{\rm scale}^{\mathrm{(heterosc)}}$, (e) $\pi_{\rm scale}^{\mathrm{(corr)}}$.
    }
\end{figure*}


\paragraph{Optimal schedule versus classical choices.}

We investigate the gain from using SGM with the schedule $a^\star$ minimising the upper bound from Theorem \ref{thm:wasserstein_bound} for $d \in \{ 5, 10, 25, 50 \}$ compared to the linear and cosine schedules (see Appendix \ref{sec:noise_sched}). To determine the optimal value $a^\star$, upper bounds were initially calculated across various dimensions for a range of $a$ values from $\{-10, -9, \ldots, 10\}$. This initial calculation aimed to identify a preliminary minimum value. Subsequently, the search was refined around these preliminary values using finer step-sizes of 0.25 to more precisely locate $a^\star$.

Results for the isotropic, heteroscedastic, and correlated cases are presented in both tabular form in Table \ref{tab:W2_table_gaussian} and visually in Figure \ref{fig:gaussian_xp_dim_w2} within the main paper. These findings are discussed in Section \ref{sec:exp_gaussian} of the main paper. 

\begin{table}[ht]
\centering
\scalebox{0.7}{\begin{tabular}{@{}cccccc@{}} 
\toprule
 & Dimension & 5 & 10 & 25 & 50\\ 
\midrule
\multirow{6}{*}{\shortstack{Isotropic}} & Upper bound min $a^\star$ & 4.5 & 4.25 & 3.75 & 4.25 \\
& Generation value in $a^\star$ & 0.039241 $\pm$ 0.012572 & \textbf{0.059274} $\pm$ 0.009438 & \textbf{0.130829} $\pm$ 0.014245 & \textbf{0.233812} $\pm$ 0.010584 \\
& VPSDE (linear sched.) & 0.036995 $\pm$ 0.004663 & 0.063939 $\pm$ 0.010876 & 0.141601 $\pm$ 0.020447 & 0.256384 $\pm$ 0.032709 \\
& Cosine schedule & \textbf{0.030996} $\pm$ 0.003254 & 0.060649 $\pm$ 0.007117 & 0.131234 $\pm$ 0.004794 & 0.251959 $\pm$ 0.005588 \\
& \% gain (vs VPSDE) & -6.07 \% & +7.30 \% & +7.61 \% & +8.79 \% \\
& \% gain (vs Cosine) & -26.60 \% & +2.26 \% & +0.31 \% & +7.20 \% \\
\midrule
\multirow{6}{*}{\shortstack{Heterosc.\\ (with rescaling)}} & Upper bound min $a^\star$ & 4.00 & 3.25 & 2.00 & 2.75 \\
& Generation value in $a^\star$ & 0.096592 $\pm$ 0.003062 & \textbf{0.143224} $\pm$ 0.004899 & \textbf{0.242493} $\pm$ 0.004769 & \textbf{0.372292} $\pm$ 0.004694 \\
& VPSDE (linear sched.) & 0.098889 $\pm$ 0.003604 & 0.147478 $\pm$ 0.009638 & 0.249144 $\pm$ 0.011394 & 0.385612 $\pm$ 0.009333 \\
& Cosine schedule & \textbf{0.096437} $\pm$ 0.002380 & 0.143701 $\pm$ 0.002460 & 0.250520 $\pm$ 0.004448 & 0.374868 $\pm$ 0.003243 \\
& \% gain (vs VPSDE) & +2.32 \% & +2.89 \% & +2.67 \% & +3.46 \% \\
& \% gain (vs Cosine) & -0.16 \% & +0.33 \% & +3.20 \% & +0.69 \% \\
\midrule
\multirow{6}{*}{\shortstack{Correlated\\ (with rescaling)}} & Upper bound min $a^\star$ & 8.00 & 8.75 & 10.50 & 11.00 \\
& Generation value in $a^\star$ & 0.066548 $\pm$ 0.013873 & \textbf{0.107291} $\pm$ 0.028454 & \textbf{0.261075} $\pm$ 0.029533 & \textbf{0.676151} $\pm$ 0.123277 \\
& VPSDE (linear sched.) & 0.072068 $\pm$ 0.019861 & 0.138240 $\pm$ 0.031119 & 0.302986 $\pm$ 0.045539 & 0.897584 $\pm$ 0.079860 \\
& Cosine schedule & \textbf{0.048276} $\pm$ 0.008605 & 0.112898 $\pm$ 0.011284 & 0.391753 $\pm$ 0.030112 & 0.765524 $\pm$ 0.022376 \\
& \% gain (vs VPSDE) & +7.65 \% & +22.36 \% & +13.81 \% & +24.68 \% \\
& \% gain (vs Cosine) & -37.77 \% & +4.96 \% & +33.31 \% & +11.67 \% \\
\midrule
\multirow{2}{*}{Parameters} & Learning rate & 1e-4 & 1e-4 & {\shortstack{1e-3\\ (1e-4 for Corr.)}} & {\shortstack{1e-3\\ (1e-4 for Corr.)}} \\
& Epochs & 20 & 30 & 75 & 150 \\
\bottomrule \\
\end{tabular}}
\caption{
\label{tab:W2_table_gaussian}
Comparison of the $\mathcal{W}_2$ distance between the target value and the generated value at $a^\star$ (the minimum value of the upper bound from Theorem \ref{thm:wasserstein_bound}) with the $\mathcal{W}_2$ distance between the generated value by VPSDE and the target distribution. We display averages plus or minus standard deviations over 10 runs. The target distributions are chosen to be Gaussian with different covariance structures: isotropic, heteroscedastic (with rescaling applied), and correlated (with rescaling applied).
}
\end{table}

\begin{figure}[h]
    \centering
    \includegraphics[width=0.6\linewidth]{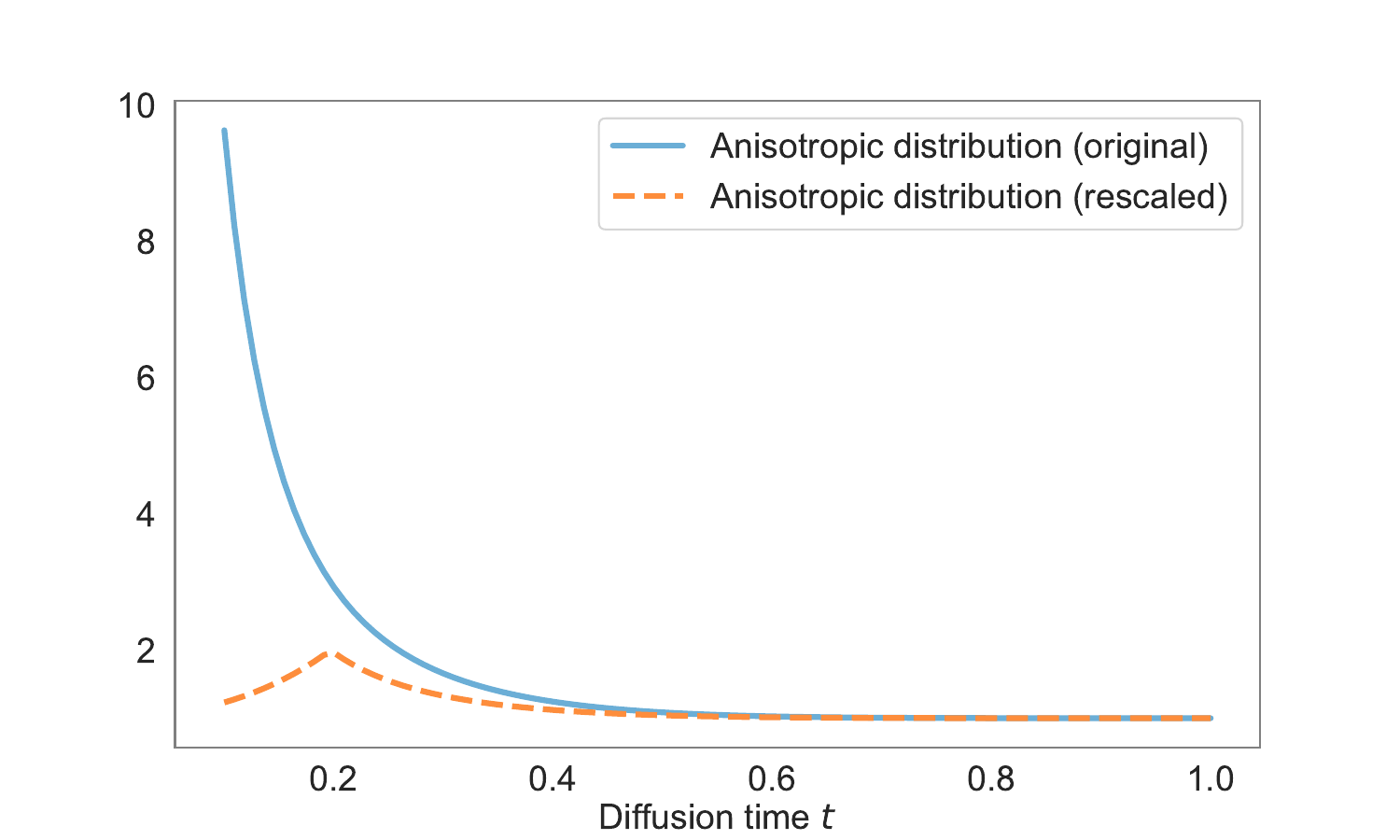}
    \caption{Comparison of the ratio strong concavity / Lipschitz continuity for the true score function $\nabla \log p_t$ in the Heteroscedastic setting before rescaling (Figure \ref{fig:W_2_result_gaussian} (b)) and after rescaling  (Figure \ref{fig:W_2_result_gaussian_rescale} (d)) throughout diffusion time $t \in (0,T]$.}
    \label{fig:lip_concav_hetero}
\end{figure}

\subsubsection{Numerical experiments on more complex synthetic data}
\label{sec:appendix_synthetic_data}

In the context of complex data distributions, the Kullback-Leibler bound (Theorem \ref{th:main}) appears to be of limited practical applicability. Specifically, $\mathcal{E}^{\rm KL}_2(\theta, \beta)$ implies that for each noise schedule tested, a distinct score approximation $\tilde{s}_{\theta}(t,x)$ must be trained. This requirement renders the bound computationally intensive and therefore not realistically usable. Additionally, $\mathcal{E}^{\rm KL}_3(\beta)$ is independent of the schedule choice over $(0,T)$, as it depends solely on its final value $\beta(T)$ which is set constant in our empirical setting (for all $a$, $\beta_a(T) = 20$).  As a consequence, the last remaining error term to analyse the bound through the lens of noise schedules is the mixing time $\mathcal{E}^{\rm KL}_1( \beta)$. However, relying exclusively on $\mathcal{E}^{\rm KL}_1 (\beta)$ would suggest selecting a schedule $t \mapsto \beta(t)$ that maximises $\int_0^T \beta(t) \rmd t$. As demonstrated in Section \ref{sec:exp_gaussian}, this approach clearly fails to yield the schedule choices near the optimal solution.

Therefore, a more reliable choice  would be to use the $\mathcal{W}_2$ bound of Theorem \ref{thm:wasserstein_bound} for which most of the terms can be computed explicitly with reasonable computational cost in the Gaussian setting. In particular, we leverage the Gaussian framework to estimate the constant terms and apply the rescaling defined in Appendix \ref{sec:appendix_W2_bound} to ensure that $C_t$ is non negative for $t \in (0,T]$. More precisely,
\begin{itemize}
    \item $L_t$ and $C_t$ are given in Lemma \ref{lem:gaussian:contract} and are computed using the empirical covariance matrix associated with $\pi_{\rm scale}$ (and using when applicable the refinements in Propositions \ref{prop:from_C_0_to_C_t} and \ref{prop:from_L_0_to_L_t}),
    \item $M$ is derived with Proposition \ref{prop:bar_l_t} with appropriate empirical estimators, 
    \item $\mathcal{W}_2(\pi_{\rm data}, \pi_\infty)$ is computed using closed-form formulas for Gaussian distributions, involving empirical estimators of the mean and covariance of  $\pi_{\rm scale}$,
    \item the term $\varepsilon$ is deliberately omitted to avoid the prohibitively high computational costs associated with training distinct models for different noise schedules.
\end{itemize} 

The experiments are run using the same neural network architecture as in the Gaussian illustrations of Appendices \ref{sec:appendix_KL_bound} and \ref{sec:appendix_W2_bound} (i.e., a dense neural network with 3 hidden layers
of width 256). The network was trained over 200 epochs for $a \in \{-10, -9, \dots, 19, 20 \}$. Contrary to the Gaussian case, conditional score matching $\mathcal{L}_{\mathrm{score}}(\param)$ \eqref{eq:cond_score_matching_objective} is used, as being closer to what is done in practice (explicit scores are now out of reach). To assess the quality of the data generation three metrics are used: 
\begin{enumerate}[(a)]
    \item an estimator of the KL-divergence based on $k$-nearest neighbors \citep{KLviaKNN} with $\lceil \sqrt{d} \rceil$ neighbors, 
    \item the sliced 2-Wasserstein distance \citep{pot_library} with 2000 projections,
    \item the negative log likelihood computed on 1000 samples defined as 
    $ - \frac{1}{1000} \sum_{i=1}^{1000} \log \pi_{\rm data} (x_i)$ with $(x_i)_{1 \leq i \leq 1000}$ samples from the generated distribution and $\pi_{\rm data}$ the probability density function to be estimated.
\end{enumerate}

\paragraph{Funnel distribution.} The first distribution considered is the Funnel distribution \citep{thin2021neo} in dimension $50$, defined as $$\pi_{\rm data} (x) =  \varphi_{a^2} (x_1)\prod_{j=2}^{d} \varphi_{\exp(2bx_1)}(x_j) \eqsp,$$
with $a=1$ and $b=0.5$. To ensure the applicability of Theorem \ref{thm:wasserstein_bound} and Lemma \ref{lem:gaussian:contract}
the samples are standardized and rescaled according to the method described in Appendix \ref{sec:appendix_W2_bound}. The results, illustrated in Figures \ref{fig:funnel_main} and \ref{fig:funnel_xp}, show that the upper bounds effectively mirror the generation outcomes across the three metrics considered. Moreover, the generation results for the parametric schedule $a^\star$ (the one that minimizes the upper bound) outperforms in all three metrics both the linear and cosine schedules (see Table \ref{tab:table_synthetic_distributions}).

\begin{figure*}[h] 
    \centering
    \begin{tabular}{cc}
        \includegraphics[width=0.4\linewidth]{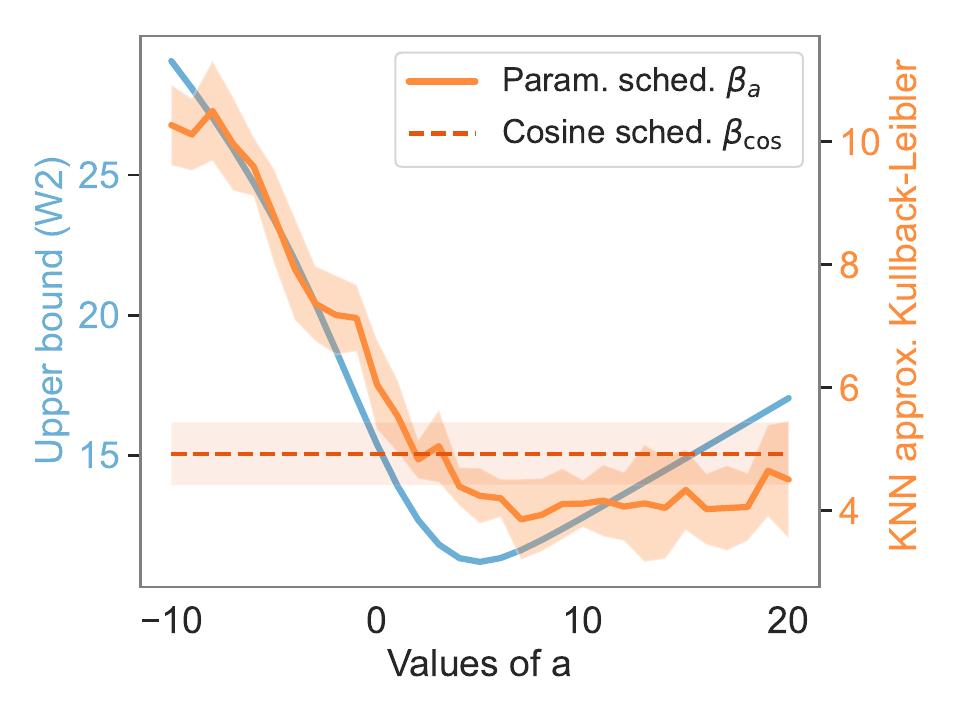}
        &
        \includegraphics[width=0.4\linewidth]{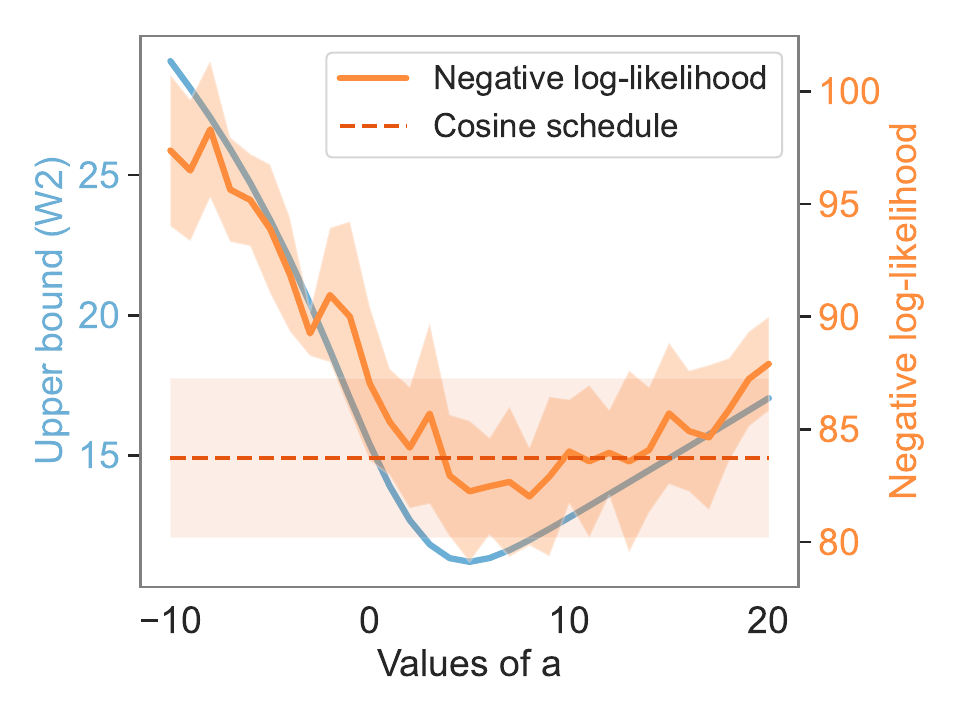}
        \\
        {(b) KL divergence with $k$-nearest neighbors estimate} & {(c) Negative log-likelihood}
    \end{tabular}
    \caption{\label{fig:funnel_xp} Upper bound and empirical distances between the data distribution and the generated samples for different metrics on a Funnel dataset in dimension 50.
    }
\end{figure*}

\paragraph{Gaussian mixture models.} The second distribution considered is a Gaussian mixture model with 25 modes in dimension $50$, defined as
$$ \pi_{\rm data} (x) = \frac{1}{25} \sum_{(j,k)\in\{-2, \dots , 2\}^2}\varphi_{\mu_{jk} , \Sigma_d} (x)$$

with $\varphi_{\mu_{jk} , \Sigma_d}$ denoting the probability density function of the Gaussian distribution with covariance matrix $\Sigma_d = \rm diag \left( 0.01, 0.01, 0.1, ..., 0.1 \right)$ and mean vector $\mu_{jk} = [j,k, 0,0,0...,0]^\top$. The results shown in Figure \ref{fig:MG25_xp} and Table \ref{tab:table_synthetic_distributions} confirm the relevance of the upper bound even for non-Gaussian datasets.

\begin{figure*}[h] 
    \centering
    \begin{tabular}{ccc}
        \includegraphics[width=0.31\linewidth]{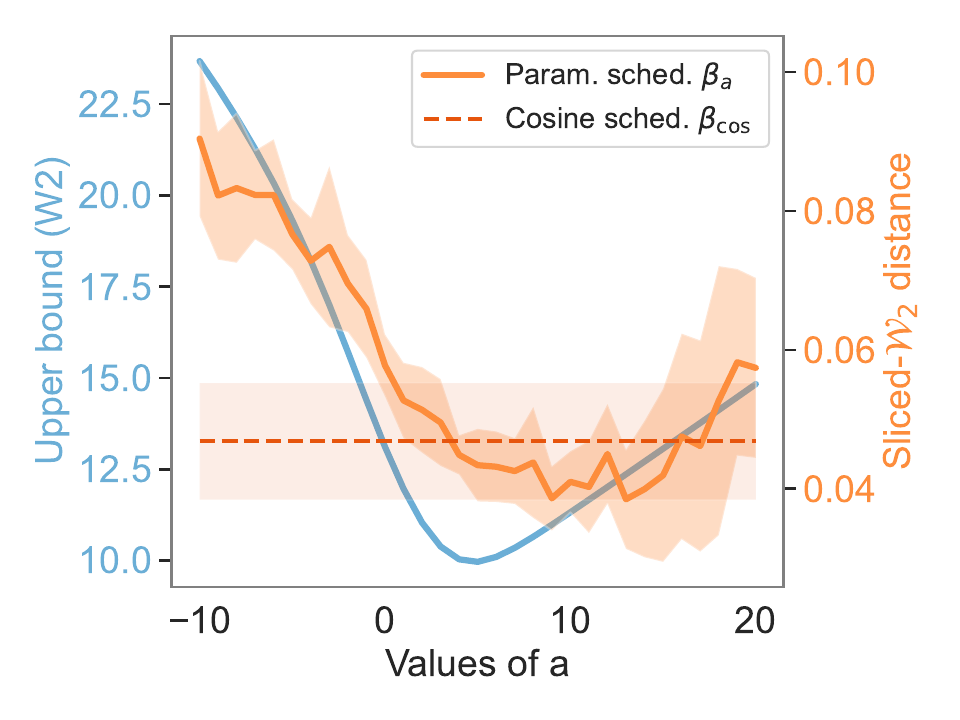}
        &
        \includegraphics[width=0.31\linewidth]{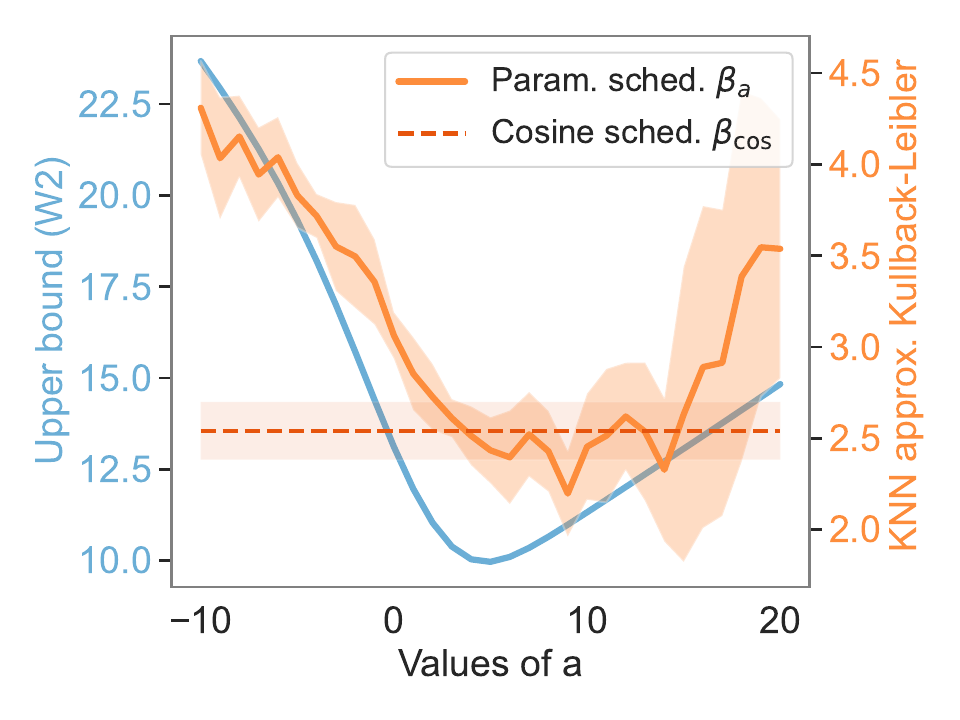}
         &
        \includegraphics[width=0.31\linewidth]{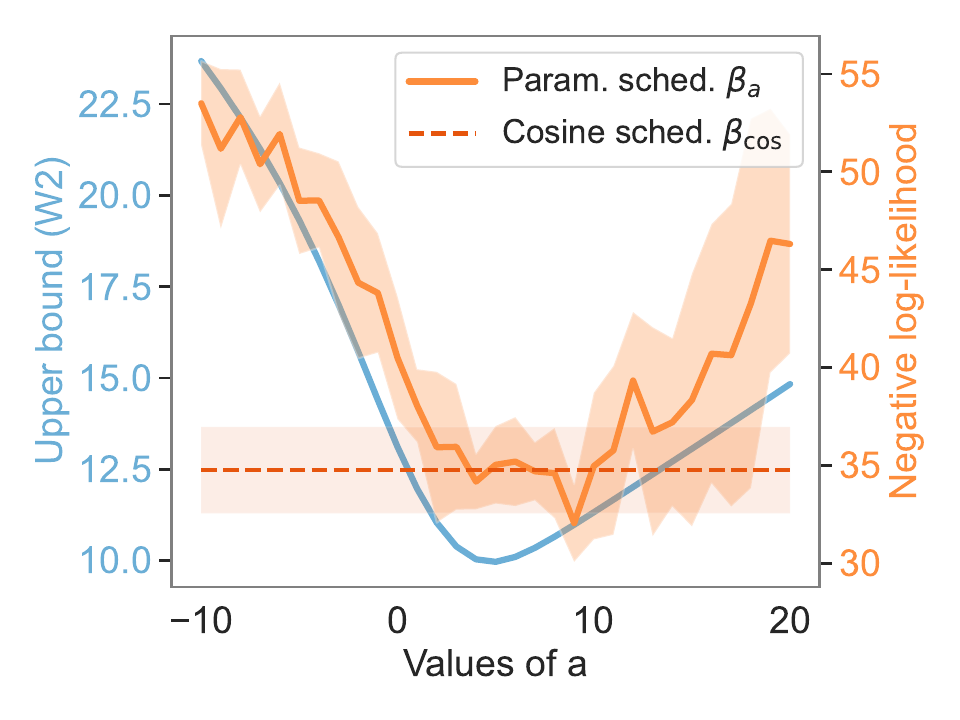}       
        \\
        {(a) Sliced-$\mathcal{W}_2$ distance} & {(b) $k$-nearest neighbors estimator} & {(c) Negative log-likelihood}
    \end{tabular}
    \caption{\label{fig:MG25_xp} Upper bound and empirical distances between the data distribution and the generated samples for different metrics on a mixture of 25 Gaussian variables dataset in dimension 50.
    }
\end{figure*}

\begin{table}[ht]
\centering
\scalebox{0.7}{\begin{tabular}{@{}ccccc@{}} 
\toprule
 & Metric & Sliced-Wasserstein & $k$-nn (Kullback-Leibler) & NLL \\ 
\midrule
\multirow{6}{*}{\shortstack{Funnel distribution}} 
& Generation value in $a^\star$ & \textbf{0.218498} $\pm$ 0.049882 & \textbf{4.242455} $\pm$ 0.450224 & \textbf{82.25179} $\pm$ 3.12809  \\
& VPSDE (linear sched.) & 0.240664 $\pm$ 0.036578 & 6.048403 $\pm$ 0.726221 & 87.02893 $\pm$ 3.40642 \\
& Cosine schedule & 0.221851 $\pm$ 0.054309 & 4.927209 $\pm$ 0.510968 & 83.73294 $\pm$ 3.53262 \\
& \% gain (vs VPSDE) & +9.21 \% & +29.88 \% & +5.49 \% \\
& \% gain (vs Cosine) & +1.51 \% & +13.91 \% & +1.77 \% \\
\midrule
\multirow{6}{*}{\shortstack{Gaussian mixture models}}
& Generation value in $a^\star$ & \textbf{0.043388} $\pm$ 0.005222 & \textbf{2.433759} $\pm$ 0.180652 & 35.033176 $\pm$ 1.97863  \\
& VPSDE (linear sched.) & 0.057763 $\pm$ 0.004450 & 3.063054 $\pm$ 0.126697 & 40.49867 $\pm$ 3.13705 \\
& Cosine schedule & 0.046816 $\pm$ 0.008402 & 2.541213 $\pm$ 0.158563 & \textbf{34.76353} $\pm$ 2.20980 \\
& \% gain (vs VPSDE) & +24.91 \% & +20.55 \% & +13.49 \% \\
& \% gain (vs Cosine) & +7.32 \% & +4.23 \% & -0.77 \%  \\
\midrule
\multirow{2}{*}{Parameters} & Learning rate & 1e-3 & 1e-3 & 1e-3  \\
& Epochs & 200 & 200 & 200  \\
\bottomrule \\
\end{tabular}}
\caption{
\label{tab:table_synthetic_distributions}
Comparison of the sliced-$\mathcal{W}_2$ distance, KL divergence coupled with $k$-nearest neighbors estimate and negative log-likelihood between the target distribution and the SGM-generated one. For the latter, the SGM is either trained with linear, cosine and $\beta_{a^\star}$ schedules. We display averages plus or minus standard deviations over 10 runs. The target distributions are chosen are Funnel and Gaussian mixture models.
}
\end{table}

\subsection{Numerical experiments on real-world datasets } \label{sec:CIFAR_EXP}

To evaluate the impact of the noise schedule on the performance of score-based generative models we evaluate the parametric family $\beta_a$ introduced in Equation \eqref{eq:noise_schedules} using CIFAR 10 dataset.  We suggest to analyse the FID (Fréchet Inception Distance) score on 50 000 samples generated for different noise schedules (different values of $a$ in $\beta_a$, see Figure \ref{fig:noise_schedules}) on CIFAR 10. 

We use pretrained models from \citet{karras2022edm} with the recommended hyperparameters designed to replicate the experiments in \cite{song2021score} corresponding to our linear schedule ($a=0$) as shown in Figure \ref{fig:noise_schedules}. In particular, we let $T=1$, $\beta (0) =0.1$, $\beta(T) = 20$, $1000$ discretization steps and sample over the diffusion $[\epsilon,1]$ with $\epsilon = 10^{-3}$.

The training process in \cite{karras2022edm} is slightly different, though equivalent, to the original implementation. In particular, the networks are not trained to directly estimate $\nabla \log p_t ( \overrightarrow X_t)$. Instead, a denoiser function $D_{\theta}(X,\sigma)$ is trained to isolate the noise from the signal for some noise level (see Equations (2) and (3) in \cite{karras2022edm}). With appropriate rescaling this denoiser can be used in the VP setting by letting 
$$ s_{\theta}(\overrightarrow X_t,t) = \frac{\sigma_t^2}{m_t}\left( D_{\theta}  \left(\frac{\overrightarrow X_t}{m_t} , \frac{\sigma_t}{m_t}\right) - \frac{\overrightarrow X_t}{m_t} \right) \eqsp,  $$
where $s_{\theta}$ is the score approximation as defined in our paper, $m_t = \exp\{-\int_0^t \beta(s) \rm d s/(2 \sigma^2)\}$ and $ \sigma^2_t = \sigma^2 (1-m_t^2)$. This formulation bridges the denoising approach with score-based methods in the VP framework. 

Figure \ref{fig:FID_SCORE_EM_CIFAR} displays the FID score for samples generated using the Euler-Maruyama discretization of the backward process for different choices of $\beta_a$ with $a \in \{ -10,-9,\ldots,10\}$ and cosine schedule $\beta_{\text{cos}}$. Although the assumptions of our results cannot be verified in such a setting, it is interesting to note that the empirical performance follows the same dynamics as in the toy numerical experiments. This indicates that the analysis and optimization of noise schedules is an interesting problem to be explored further for complex cases.

\begin{figure}[h]
    \centering
    \includegraphics[width=0.6\linewidth]{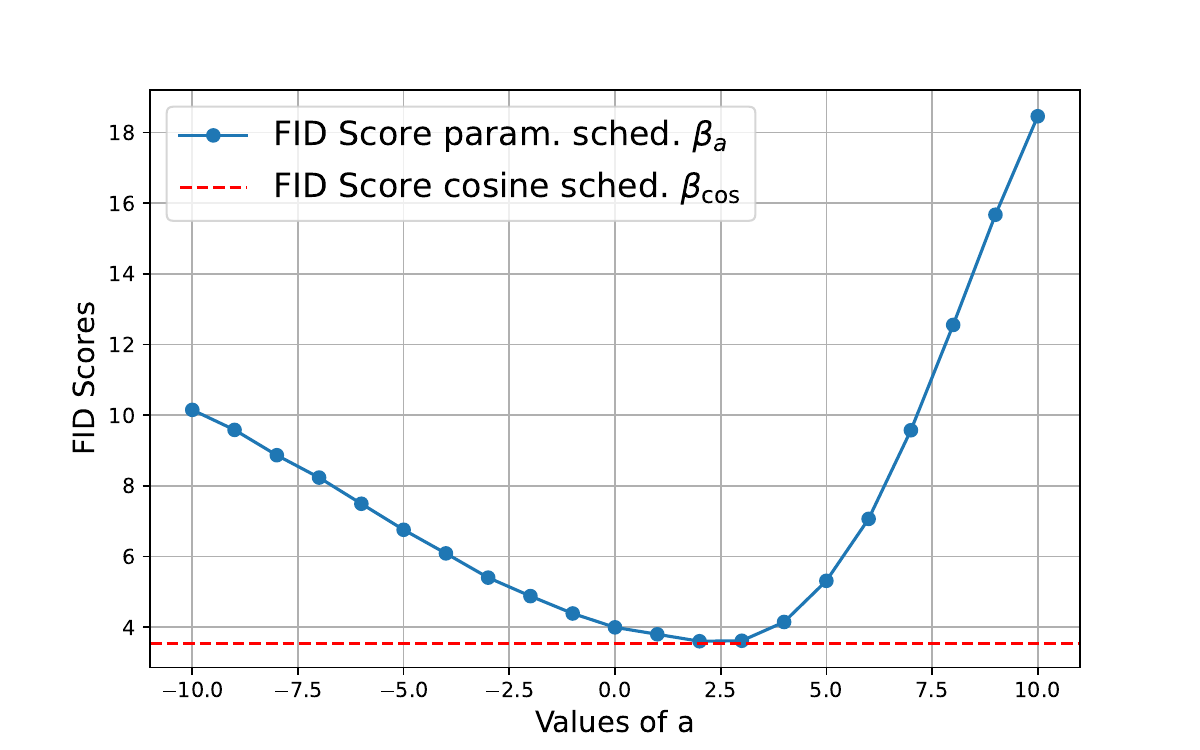}
    \caption{FID scores on 50 000 generated samples using pre-trained models from \cite{karras2022edm} with different noise schedules from the parametric family \eqref{eq:noise_schedules} and cosine schedule.}
    \label{fig:FID_SCORE_EM_CIFAR}
\end{figure}

\section{Conditional training in the Gaussian setting}
\label{sec:conditional_training_SGM}

Section \ref{sec:exp_gaussian} of this paper is dedicated to the illustration of the theoretical upper bounds and their relevance in the Gaussian setting (i.e., when $\pi_{\rm data}$ is Gaussian). This choice has been motivated by the fact that, under this setting, all constants in the upper bounds from Theorem \ref{th:main} and Theorem \ref{thm:wasserstein_bound} are either analytically available or could be precisely estimated (see Appendices \ref{sec:appendix_KL_bound} and \ref{sec:appendix_W2_bound}). 

In particular, both upper bounds display error terms proportional to $\mathcal{L}_{\rm explicit} (\theta)$ \eqref{eq:explicit_score_matching_objective}, which has motivated the use of explicit score matching during the training. To do so, we used a deep neural architecture (see Figure \ref{fig:nn_architecture}) trained to minimize $\mathcal{L}_{\rm explicit} (\theta)$ \eqref{eq:explicit_score_matching_objective} using as a target the true score function. This is possible because in the Gaussian setting, the true score function is analytically known (Lemma \ref{lem:exactscore}). However, in most applications the score function is not available, because the data distribution is not known and has to be learned. This is the reason why, in practice we rely on conditional score matching (i.e., the minimization of $\mathcal{L}_{\mathrm{score}}(\param)$ \eqref{eq:cond_score_matching_objective}). 
This approach is particularly relevant given the relationship between the explicit and conditional score functions:
$    \mathcal{L}_{\rm explicit} \left( \theta \right) = \mathcal{L}_{\rm score} \left( \theta \right) - \mathbb{E}\left[\|\nabla \log p_\tau(\ora X _\tau) - \nabla \log p_\tau(\ora X _\tau| X _0)\|^2\right] \eqsp.$

Consequently, all the theoretical upper bounds discussed in Sections \ref{sec:main} and \ref{sec:wasserstein} can be adjusted by a constant (with respect to $\theta$) to account for discrepancies between the score function learned through $\mathcal{L}_{\mathrm{score}}$ or $\mathcal{L}_{\rm explicit}$.

The rest of this section demonstrates the numerical effects of employing conditional score matching instead of explicit score matching, following the numerical set-up of Appendices \ref{sec:appendix_KL_bound} and \ref{sec:appendix_W2_bound}. In Figure \ref{fig:conditional_KL}, the Kullback-Leibler upper bound from Theorem \ref{th:main} is depicted in varying shades of blue, while the empirical $\mathrm{KL}(\pi_{\rm data} || \widehat{\pi}_N^{(\beta_a, \theta)})$  across parameters $a \in \{ -10, -9, -8, ..., 10 \}$ is shown in varying shades of orange.

In Figure \ref{fig:conditional_KL}, three learning scenarios are presented: one using explicit score matching (which exactly matches the results of Figure \ref{fig:gaussian_xp} (top)), another with conditional score matching over 150 epochs, and a third with 300 epochs. Both the generation results and the upper bounds show diminished performance as the curves are shifted upwards. Nonetheless, the overall curve shapes are similar, and the optimal points remain closely aligned. Interestingly, both the upper bounds and the generation outcomes in the conditional scenarios demonstrate more pronounced peaks near the minimum values. This suggests that precise noise schedule selection may yield even better performance gain when SGMs are trained using conditional score matching.

Additionally, Figure \ref{fig:conditional_1000_epochs} demonstrates that increasing the number of training iterations when using conditional score matching provides results more and more similar to that obtained with explicit score matching. This effect is noticeable in both the KL divergence and the $\mathcal{W}_2$ distance.

\begin{figure}[H]
    \centering
    \begin{subfigure}{}     
    \includegraphics[width=0.7\linewidth]{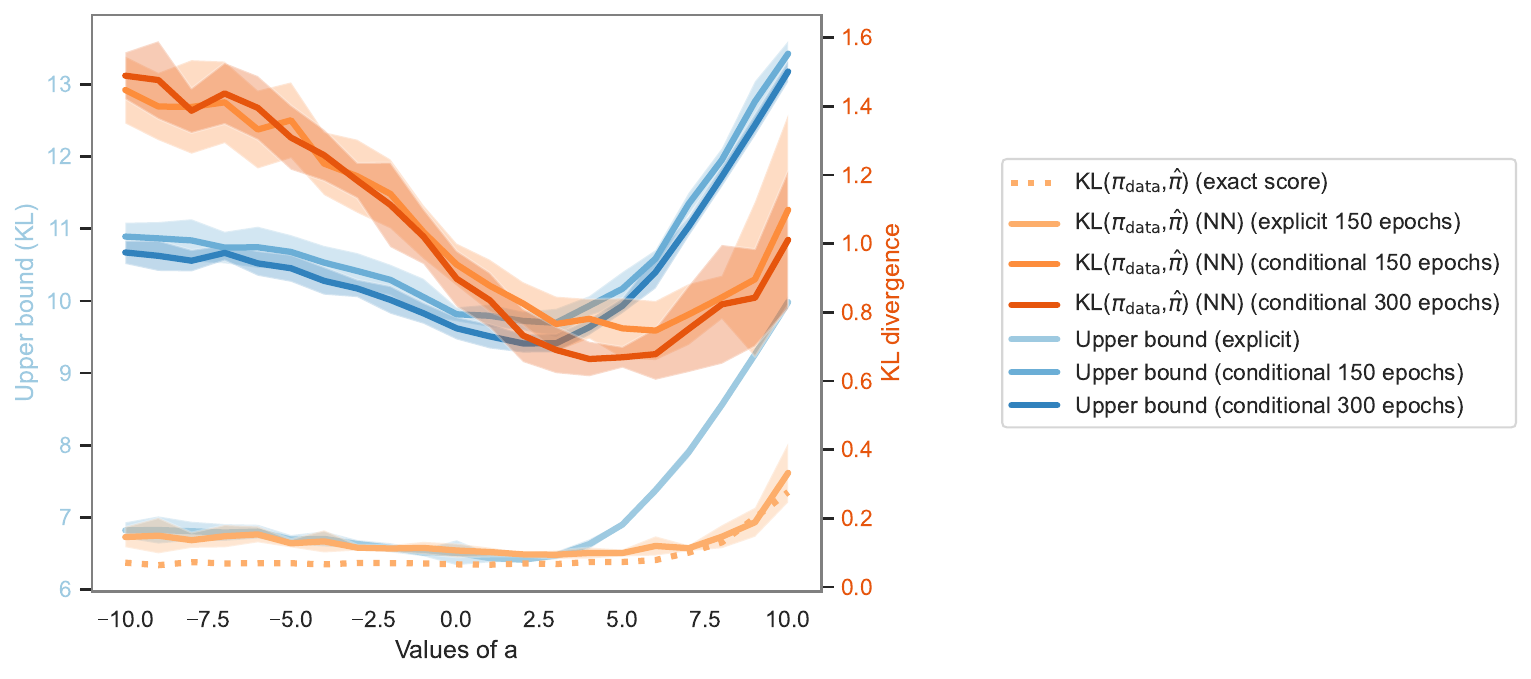}
        \caption*{Isotropic setting}
    \end{subfigure}
    \begin{subfigure}{}
        \includegraphics[width=0.7\linewidth]{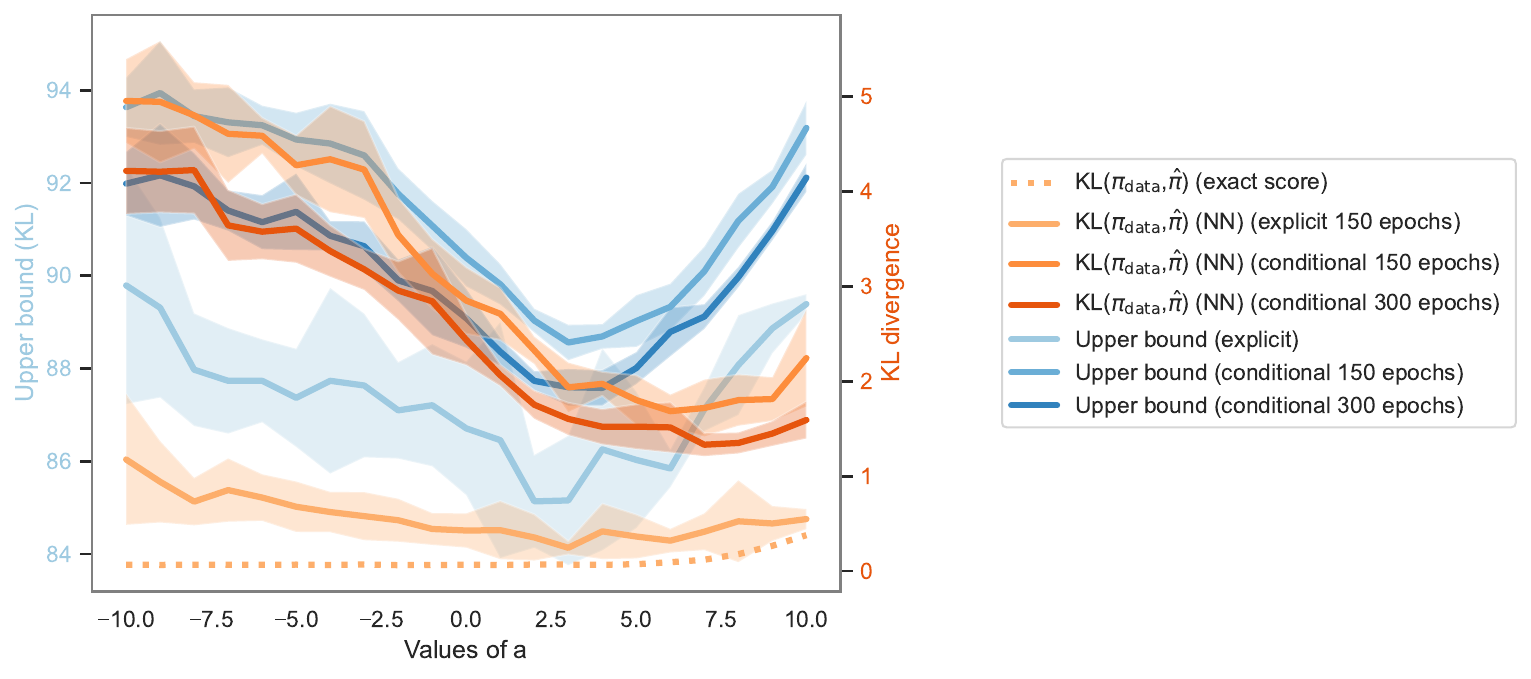}
        \caption*{Heteroscedastic setting}
    \end{subfigure}
    \begin{subfigure}{}
    \includegraphics[width=0.7\linewidth]{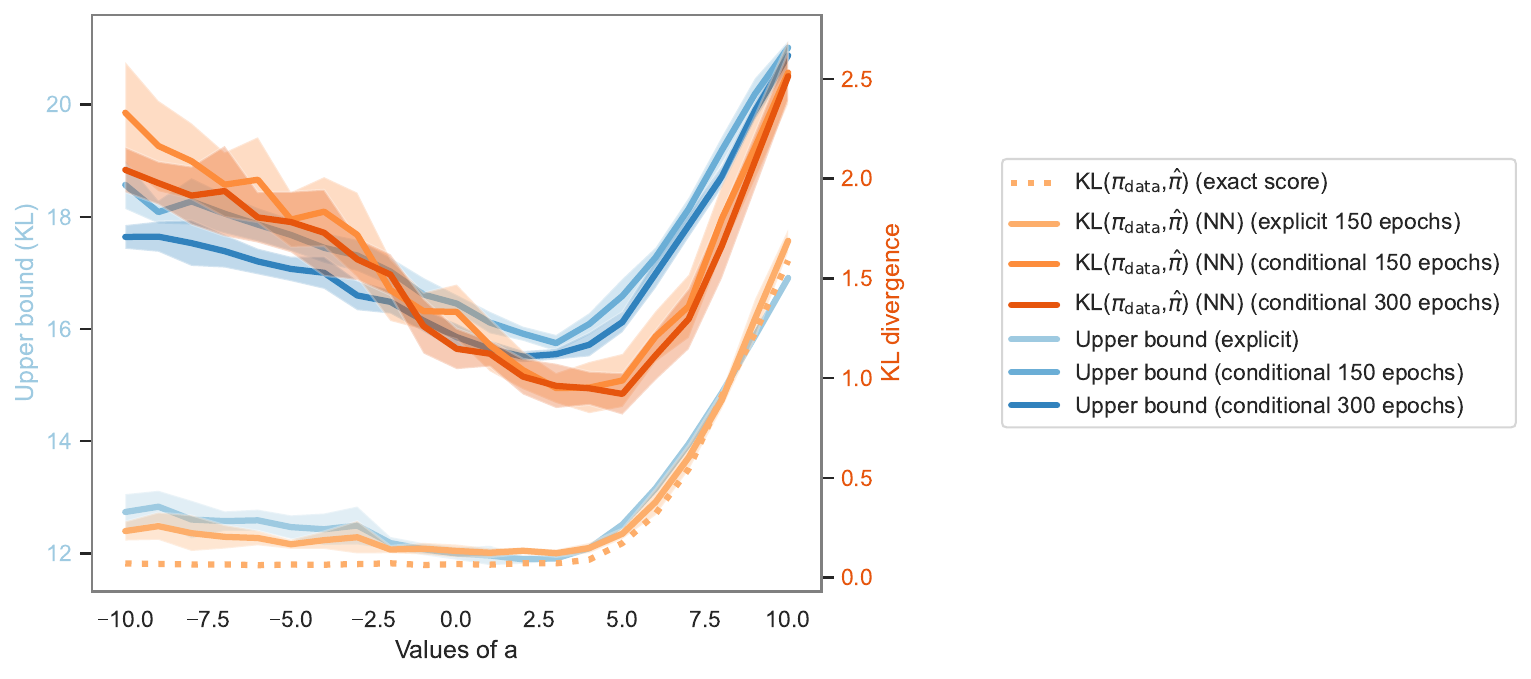}
        \caption*{Correlated setting}
    \end{subfigure}
    
    \caption{\label{fig:conditional_KL}
Comparison of the empirical KL divergence (mean value $\pm$ std over 10 runs) between $\pi_{\mathrm{data}}$ and $ \widehat{\pi}_N^{(\beta_{a}, \theta)}$ (in orange) and the upper bound of Theorem \ref{th:main} (in blue) w.r.t.\ the parameter $a$ used in the definition of the noise schedule $\beta_a$, for $d=50$.}
\end{figure}

\begin{figure}[H]
    \centering
    \begin{subfigure}{}     
    \includegraphics[width=0.7\linewidth]{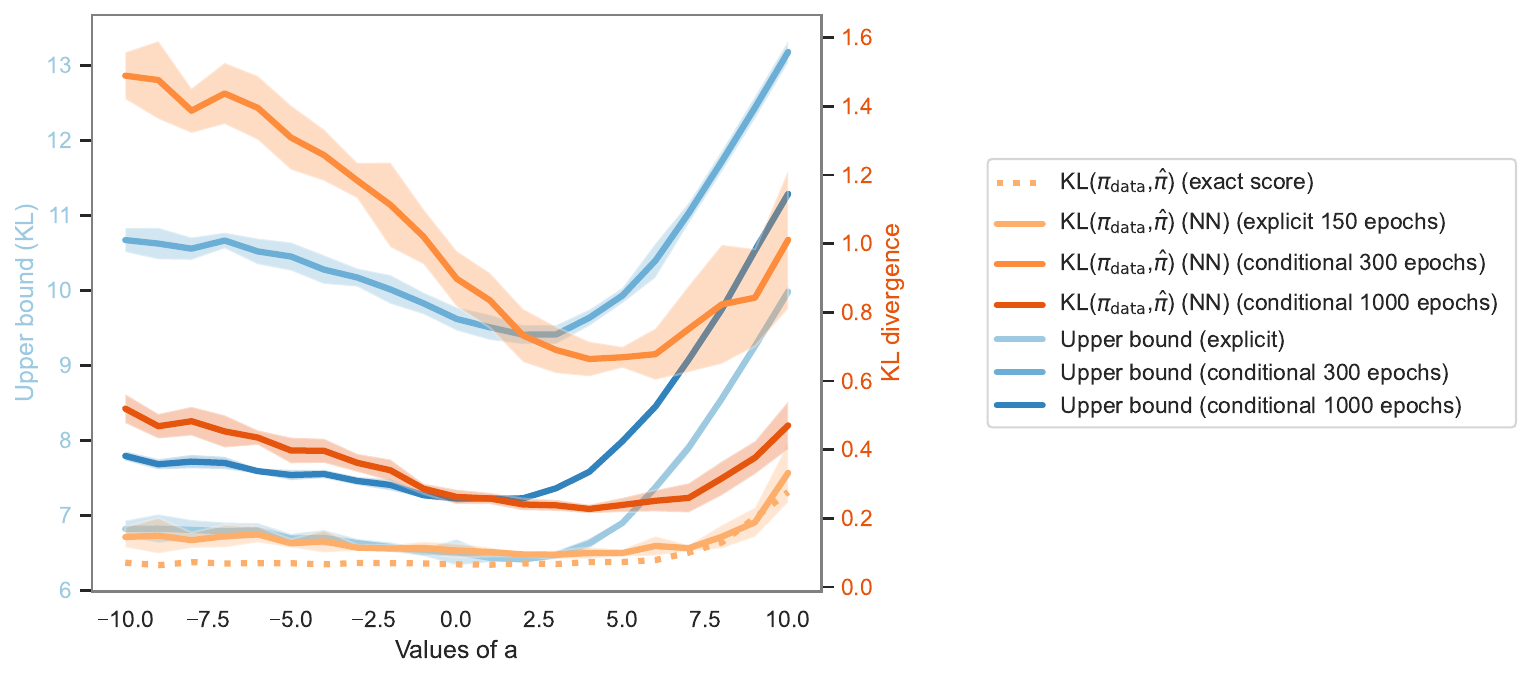}
    \end{subfigure}
    \begin{subfigure}{}
        \includegraphics[width=0.7\linewidth]{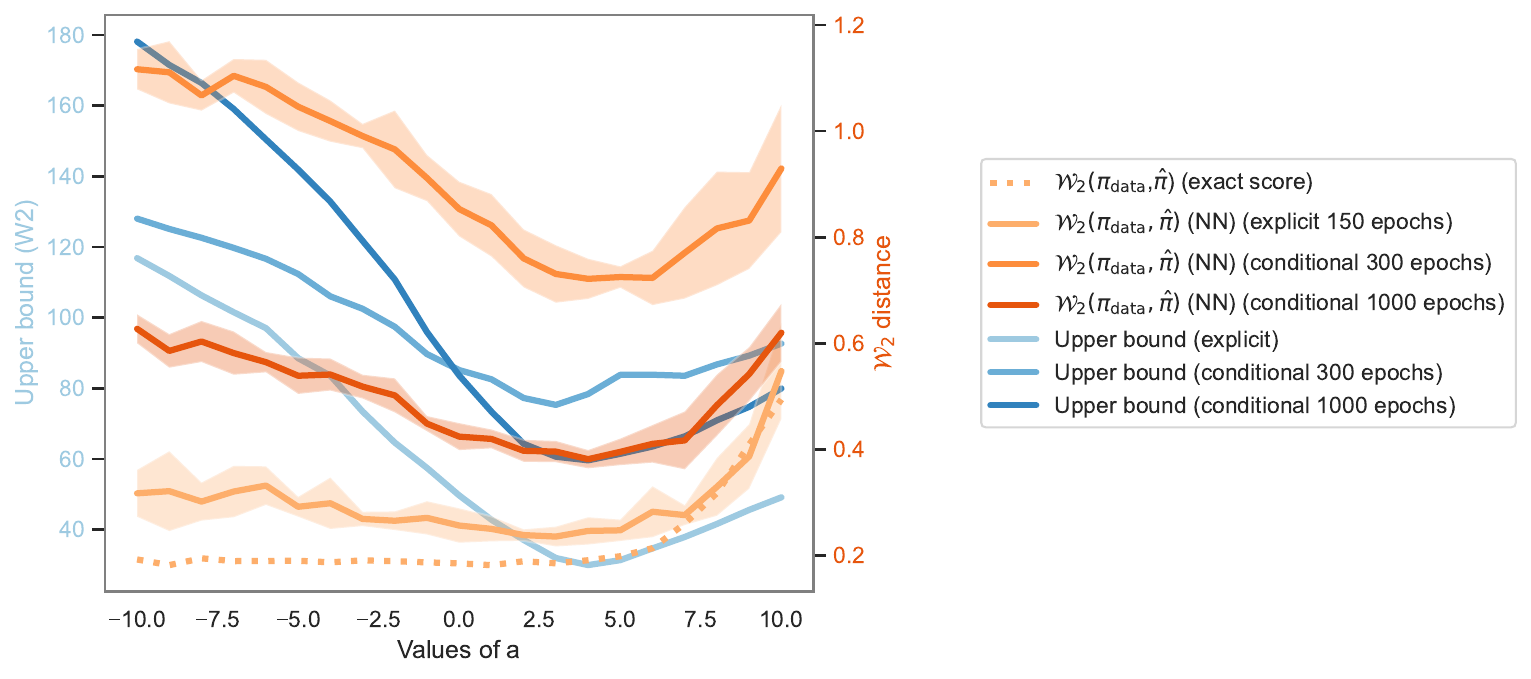}
    \end{subfigure}
 \caption{\label{fig:conditional_1000_epochs}
Comparison of the empirical KL divergence (top) and the $\mathcal{W}_2$ distance (bottom) (mean value $\pm$ std over 10 runs) between $\pi_{\mathrm{data}}=\pi_{\rm data}^{\mathrm{(iso)}}$ and $ \widehat{\pi}_N^{(\beta_{a}, \theta)}$ (in orange) and the upper bound of Theorem \ref{th:main} (top) and of Theorem \ref{thm:wasserstein_bound} (bottom) (in blue) w.r.t.\ the parameter $a$ used in the definition of the noise schedule $\beta_a$, for $d=50$.
}
\end{figure}

\end{document}